\documentclass[a4paper,11pt]{article}

\usepackage{amssymb,amsmath,amsthm}
\usepackage{bm}
\usepackage[shortlabels]{enumitem} 
\usepackage{a4wide} 
\usepackage{hyperref}
\usepackage{cite} 
\usepackage{mathrsfs} 
\usepackage{longtable} 
\usepackage{subcaption}

\usepackage[ruled, titlenotnumbered]{algorithm2e}
\SetKwInput{kwInit}{Initialization}

\usepackage{graphicx,psfrag}

\usepackage{ifthen}
\usepackage[dvipsnames]{xcolor} 
\newboolean{show_comments}
\setboolean{show_comments}{true} 

\DeclareRobustCommand{\nick}[1]{
\ifthenelse{\boolean{show_comments}}
{\begingroup\color{Blue}{[\textbf{Nick:} #1]}\endgroup}
{}
}

\numberwithin{equation}{section}
\allowdisplaybreaks[4]

\newtheorem{proposition}{Proposition}[section]
\newtheorem{theorem}[proposition]{Theorem}
\newtheorem{lemma}[proposition]{Lemma}

\newtheorem{definition}[proposition]{Definition}
\theoremstyle{definition}

\newtheorem{example}{Example}

\renewenvironment{proof}{\smallskip\noindent\emph{\textbf{Proof.}}%
  \hspace{1pt}}{\hspace{-5pt}{\nobreak\quad\nobreak\hfill\nobreak%
    $\square$\vspace{2pt}\par}\smallskip\goodbreak}
 
\newenvironment{proofof}[1]{\smallskip\noindent{\textbf{Proof~of~#1.}}%
  \hspace{1pt}}{\hspace{-5pt}{\nobreak\quad\nobreak\hfill\nobreak%
    $\square$\vspace{2pt}\par}\smallskip\goodbreak}

\newcommand{\Lip}{\mathrm{Lip}}
\renewcommand{\epsilon}{\varepsilon}
\renewcommand{\phi}{\varphi}
\renewcommand{\theta}{\vartheta}

\newcommand{\spt}{\mathop{\rm spt}}

\newcommand{\wto}{\rightharpoonup}
\newcommand{\diam}{\mathop{\rm diam}}
\renewcommand{\d}{{\,d}}

\newcommand{\graph}{\mathop{\rm graph}}
\newcommand{\id}{\mathbf{id}}

\begin{document}

\title{Modeling of crowds in the regions with moving obstacles via measure sweeping processes}

\author{Nadezhda Maltugueva\footnote{Matrosov Institute for System
    Dynamics and Control Theory, 134 Lermontov str., Irkutsk, 664033, Russia.
    Email: \href{mailto:nickpogo@gmail.com}{nickpogo@gmail.com}, \href{mailto:nmaltugueva@gmail.com}{nmaltugueva@gmail.com}}
  \footnote{The author is supported by the Russian
    Foundation for Basic Research, Project No. 18-31-20030.}\qquad
  Nikolay Pogodaev${}^{*}$\footnote{The author is supported by the Russian
    Foundation for Basic Research, Project No. 18-01-00026.}
}

\maketitle

\begin{abstract}

  \noindent
  We present a model of crowd motion in regions with moving obstacles, which is
  based on the notion of measure sweeping process~\cite{Marino2015}. The obstacle
  is modeled by a set-valued map, whose values are complements to
  \( r \)-prox-regular sets. The crowd motion obeys a nonlinear transport
  equation outside the obstacle and a normal cone condition (similar to that of the classical sweeping processes theory) on the boundary. We prove the well-posedness of the
  model, give an application to the environment
  optimization problems, and provide some results of numerical computations.
  \medskip

  \noindent\textit{2010~Mathematics Subject Classification:}
35F10, 
35R37, 
49J53, 
49Q20  

  \medskip

\end{abstract}

\section{Introduction}

Moving crowds are usually modeled, at the macroscopic level, by evolution PDEs
with nonlocal
terms~\cite{Carrillo2010,Colombo2011,ColomboNonloc2011,ColomboRossi2018,Mogilner1999,MauryBook2019,PiccoliRossi2018}. States of these equations are measures (or densities) which
describe the distribution of individuals (also called agents) on some configuration space, typically,
the space of agents' positions or position-velocity pairs. Nonlocal terms
appear due to the fact that the behavior of each agent depends on the positions
of other agents.
Such equations can often be expressed either as Wasserstein gradient
flows~\cite{Ambrosio2005} or nonlocal transport
equations~\cite{PiccoliRossi2013}. Each framework has its own
advantages:  the first one allows to deal with various
diffusion terms, the second one admits vector fields that do not possess the
gradient structure.
Stationary obstacles in both cases are handled by imposing either the Neumann \( \rho v\cdot n = 0 \) or the Dirichlet \( \rho=0 \)
boundary condition. The latter condition is more demanding: to achieve it one
has to adjust nonlocal terms~\cite{ColomboRossi2018} or introduce specific
distances in the space of measures\cite{Figalli2010}.

Measures evolving inside moving domains were considered, probably for the
first time, by Di Marino, Maury and Santambrogio~\cite{Marino2015}. They
described, in particular, how a measure \( \rho_{t} \)
supported on a
time dependent convex set \( \bm C(t) \) evolves when it is pushed by the
boundary of \( \bm C(t) \). To deal with the problem
they introduced a notion of measure sweeping process and extend the classical
Moreau catching-up scheme to the space of measures.

Our goal is to extend the approach of~\cite{Marino2015} in two
respects: we admit here non-convex (more precisely, \( r \)-prox-regular) driving sets \( \bm C(t) \) and measures \( \rho_{t} \) which are not only pushed by
\(\partial \bm C(t)  \) but also drift along a given nonlocal vector field
\( \mathscr V(\rho_{t}) \). One can naturally think of such evolution models as
\emph{perturbed} measure sweeping processes with perturbation given by
\( \mathscr V \). We shall see that these modifications turn the concept of
measure sweeping process into a usable tool of crowd dynamics. Moreover, we
shall prove that any solution of the measure sweeping process satisfies the
underlying PDE together with the Neumann type condition
\( \rho(n_{t} + v \cdot  n_{x}) = 0 \) on the boundary of the time dependent domain.
Let us stress that below we deal only with Lipschitz non-local vector field, so
basically we stay within the framework of~\cite{PiccoliRossi2013}.

We define a moving obstacle by a set-valued map
\( \bm O\colon [0,T]\rightrightarrows \mathbb{R}^{d} \) taking values among open subsets of
\( \mathbb{R}^{d} \). Instead of dealing with \( \bm O(t) \) we prefer to look at its complement
\( \bm C(t) = \mathbb{R}^{d}\setminus \bm O(t) \), which we call the viability region.
The values of \( \bm C \colon [0,T]\rightrightarrows \mathbb{R}^{d} \) are assumed to be closed bounded \( r \)-prox-regular
sets, for a given \( r>0 \). The perturbed measure sweeping process is the system of
the form
\begin{equation}
  \label{eq:sp}
  \begin{cases} 
    \partial_t \rho_t + \nabla\cdot (v_t\rho_t) = 0,\\
    v_t(x) - \mathscr{V}(\rho_t)(x)\in -N_{\bm C(t)}(x),\\
    \spt(\rho_t)\subset \bm C(t).
\end{cases}
\end{equation}
where $\mathscr{V}$ maps measures to vector fields and
\( N_{A}(x) \) denotes the proximal normal cone to \( A\subset \mathbb{R}^{d} \) at
\( x \).  In what follows, \( \mathcal P_{2}(\mathbb{R}^{d}) \) denotes the
space of probability measures with finite second moments equipped with the
Wasserstein distance \( W_{2} \).

\begin{definition}
  An absolutely continuous curve
  \( \rho\colon [0,T]\to \mathcal P_{2}(\mathbb{R}^{d}) \) is said to be a solution
  of the measure sweeping process~\eqref{eq:sp} if
  \begin{itemize}
    \item there exists a Borel
  vector field \( (t,x)\mapsto v_{t}(x) \) such that \( (\rho,v) \) satisfies
  \begin{displaymath}
    \partial_{t}\rho_{t} + \nabla \cdot (v_{t}\rho_{t})=0
  \end{displaymath}
      in the sense of distributions;
    \item the normal cone
  condition
  \begin{displaymath}
    v_t(x) - \mathscr{V}(\rho_t)(x)\in -N_{\bm C(t)}(x)
  \end{displaymath}
      holds for a.e. \( t\in [0,T] \) and \( \rho_{t} \)-a.e. \( x\in \mathbb{R}^{d} \);
    \item \( \spt(\rho_t)\subset \bm C(t) \) for all \( t\in [0,T] \).
  \end{itemize}
\end{definition}

Remark that the normal cone
condition implies that \( v_{t}(x) = \mathscr{V}(\rho_{t})(x) \) if \( x \) lies in
the interior of \( \bm C(t) \). Therefore, in the interior of the viability domain the crowd moves according to the nonlocal transport equation
\begin{displaymath}
  \partial_{t}\rho_{t} + \nabla \cdot(\mathscr V(\rho_{t})\rho_{t}) = 0.
\end{displaymath}
The inclusion \( \spt \rho_{t}\subset \bm C(t) \) guaranties that the crowd never leaves
the viability region.

Throughout the paper, we impose the following assumptions:
\begin{enumerate}
  \item[(\(\mathbf{A_1}\))] There exist \( L > 0 \) such that
    \( \mathscr{V}\colon \mathcal P_{2}(\mathbb{R}^{d})\to C(\mathbb{R}^{d};\mathbb{R}^{d}) \)
    satisfies
    \begin{gather*}
      \left\|\mathscr{V}(\rho_{1})-\mathscr{V}(\rho_{2})\right\|_{\infty}\leq L W_{2}(\rho_{1},\rho_{2}) \quad \forall \rho_{1},\rho_{2}\in \mathcal{P}_{2}(\mathbb{R}^{d})\\
      \left|\mathscr{V}(\rho)(x)-\mathscr{V}(\rho)(y)\right| \leq L|x-y|, \quad \left|\mathscr V(\rho)(x)\right|\leq L\quad \forall x,y\in \mathbb{R}^{d},\;\forall \rho\in \mathcal P_{2}(\mathbb{R}^{d}),
    \end{gather*}
  \item[(\(\mathbf{A_2}\))] There exist \( M>0 \) and \( r > 0 \) such that
      the set-valued map \( \bm C\colon [0,T]\rightrightarrows\mathbb{R}^{d} \) is
    \( L \)-Lipschitz in the Hausdorff distance \( d_{H} \):
      \begin{displaymath}
        d_{H}\left(\bm C(t),\bm C(s)\right)\leq M|t-s|\quad \forall t,s\in [0,T],
      \end{displaymath}
      and its values are compact \( r \)-prox-regular sets.

\end{enumerate}

Now we are ready to state the main result of the paper.

\begin{theorem}
  \label{thm:main}
  Let \( \mathscr V \) satisfy \( (\mathbf{A_{1}})\) and \( \bm C \) satisfy
  \( (\mathbf{A_{2}}) \). Then the following assertions hold:
  \begin{enumerate}
    \item[\( (1) \)]
    For any
  initial measure \( \theta\in \mathcal P_{2}(\mathbb{R}^{d}) \) such that
  \( \spt \theta\subset \bm C(0) \), there exists a unique solution
  \( \rho\colon [0,T]\to \mathcal P_{2}(\mathbb{R}^{d}) \) of the measure sweeping
      process~\eqref{eq:sp} with \( \rho_{0} = \theta \).
    \item [\( (2) \)]
      The corresponding vector
      field \( v \) satisfies
      \begin{displaymath}
      |v_{t}(x)|\leq 2L+M, \quad\text{for a.e. \( t \) and
      \( \rho_{t} \)-a.e. \( x \)}.
      \end{displaymath}
      In particular,
      \( t\mapsto\rho_{t} \) is \( (2L+M) \)-Lipschitz.
    \item [\( (3) \)]
      If \( \graph \bm C \) is \( r' \)-prox-regular for some
      \( r'>0 \) then
      \begin{displaymath}
        \xi + v_{t}(x)\cdot \eta = 0\quad \forall (\xi,\eta)\in N_{\graph \bm C}(t,x)
      \end{displaymath}
      for a.e. \( t \) and \( \rho_{t} \)-a.e. \( x \).
    \item [\( (4) \)]
      If \( \tilde{\bm C} \) is another set-valued map satisfying
      \( (\mathbf{A_{2}}) \) and \( \tilde \rho \) is a solution of the corresponding measure
      sweeping process then
      the estimate
      \begin{equation}
        \label{eq:mainestim}
      \bm r(t) \leq \Big(\bm r(0) + (6L+2M)\int_{0}^{t}\Delta(s)\d s\Big)e^{\left( 4L + \frac{3L+M}{2r}\right)t}
      \end{equation}
      holds for all \( t\in [0,T] \), where
      \( \bm r(t) = \frac12 \mathcal W_{2}^{2}(\rho_{t},\tilde \rho_{t}) \) and
      \( \Delta(t) = d_{H}(\bm C(t),\tilde{\bm C}(t)) \).
  \end{enumerate}
\end{theorem}

To prove the existence part we use the following version of Moreau's
catching-up algorithm, which yields, for every natural \( N \), a sequence of
probability measures \( \rho^{\tau}_{2k\tau} \), \( k=0,\ldots,N \).
\vspace{8pt}

\begin{algorithm}[H]
\SetAlgoLined
\caption{The catching-up scheme (see Figure~\ref{fig:catch})}
\kwInit{Split \( [0,T] \) into \( 2N \) segments of length \( \tau := T/(2N) \), then
  set \( \rho_{0}:=\theta \), \( k := 0 \).}
\While{\( 2k\tau < T \)}{
\nl Solve the linear continuity equation
    \begin{equation}
     \label{eq:loccont}
      \partial_{t}\mu_{t} + \nabla\cdot(2\mathscr V(\rho_{2k\tau})\mu_{t}) = 0,\quad \mu_{2k\tau} = \rho^{\tau}_{2k\tau},
    \end{equation}
    on the segment \( [2k\tau,(2k+1)\tau] \) and set
    \( \rho^{\tau}_{(2k+1)\tau} := \mu_{(2k+1)\tau} \).

\nl Project \( \rho^{\tau}_{(2k+1)\tau} \) onto \( \bm C((2k+2)\tau) \) and set
    \( \rho^{\tau}_{(2k+2)\tau} \) to be equal to this projection.

\nl    \( k:=k+1 \).
}
\end{algorithm}
\vspace{8pt}
With \( \rho_{2k\tau}^{\tau} \) at hand, we can construct two curves on
\( \mathcal P_{2}(\mathbb{R}^{d}) \):
\begin{itemize}
  \item a \emph{continuous} one \( \rho^{\tau} \), by connecting \( \rho_{2k\tau}^{\tau} \), \( \rho^{\tau}_{(2k+1)\tau} \) with a (unique)
trajectory of~\eqref{eq:loccont} and \( \rho^{\tau}_{(2k+1)\tau} \),
    \( \rho^{\tau}_{(2k+2)\tau} \) with a (unique) Wasserstein geodesic;
  \item a \emph{piecewise constant} one \( \bar \rho^{\tau} \), which equals to \( \rho^{\tau}_{2k\tau} \) on
\( [2k\tau,(2k+1)\tau] \) and \( \rho^{\tau}_{(2k+1)\tau} \) on \( [(2k+1)\tau,(2k+2)\tau] \).
\end{itemize}
It can be shown that \( \rho^{\tau} \) converges to some \( \rho \) as
\( \tau \to 0\). The latter curve, being absolutely continuous, satisfies
\(\partial_{t}\rho_{t}+\nabla\cdot(v_{t}\rho_{t})=0\), for some velocity \( v \), and
\( \spt\rho_{t}\subset \bm C(t) \), for all \( t \). The piecewise constant curve \( \bar \rho^{\tau} \), which has the same limit as \( \rho^{\tau} \), is used to
prove the normal cone condition.

Assertion (4) (and, thus, the uniqueness part) follows from the
standard representation for the time derivative of the squared Wasserstein
distance along a pair of absolutely continuous curves (see~\cite{Ambrosio2005} or Appendix~\ref{sec:Diff}).

\begin{figure}[h]
  \centering
  {\includegraphics[width=0.24\textwidth]{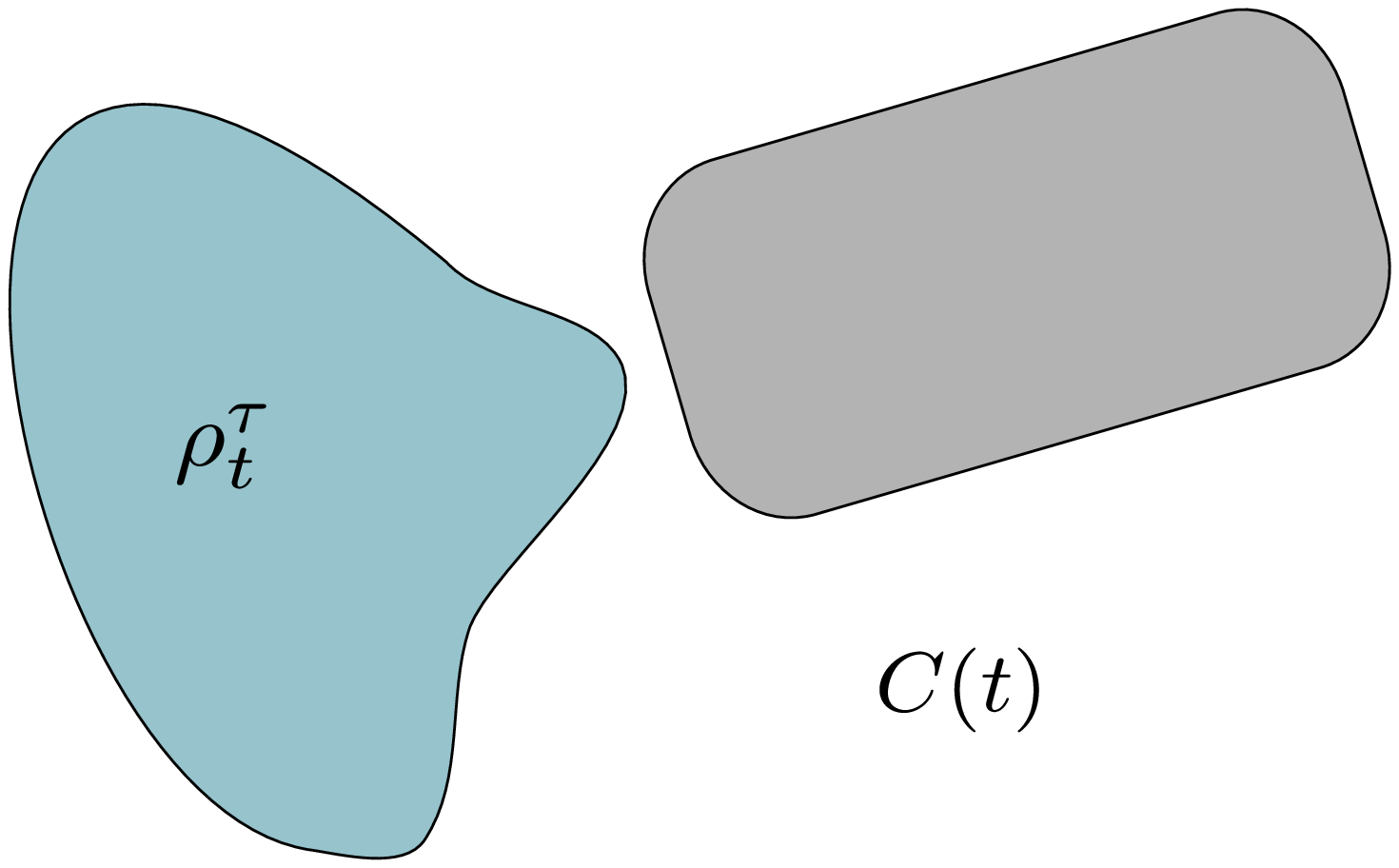}}
  {\includegraphics[width=0.24\textwidth]{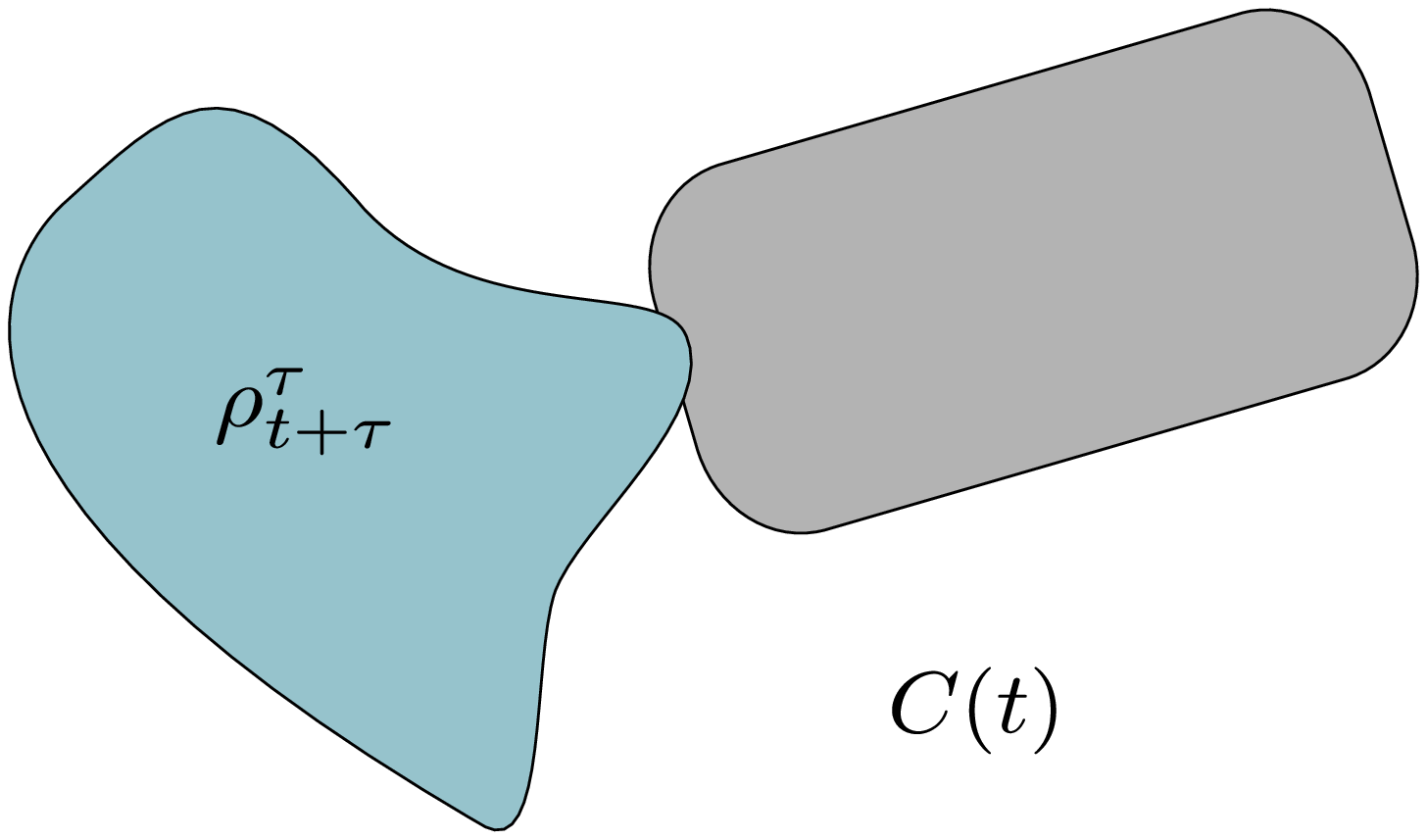}}
  {\includegraphics[width=0.24\textwidth]{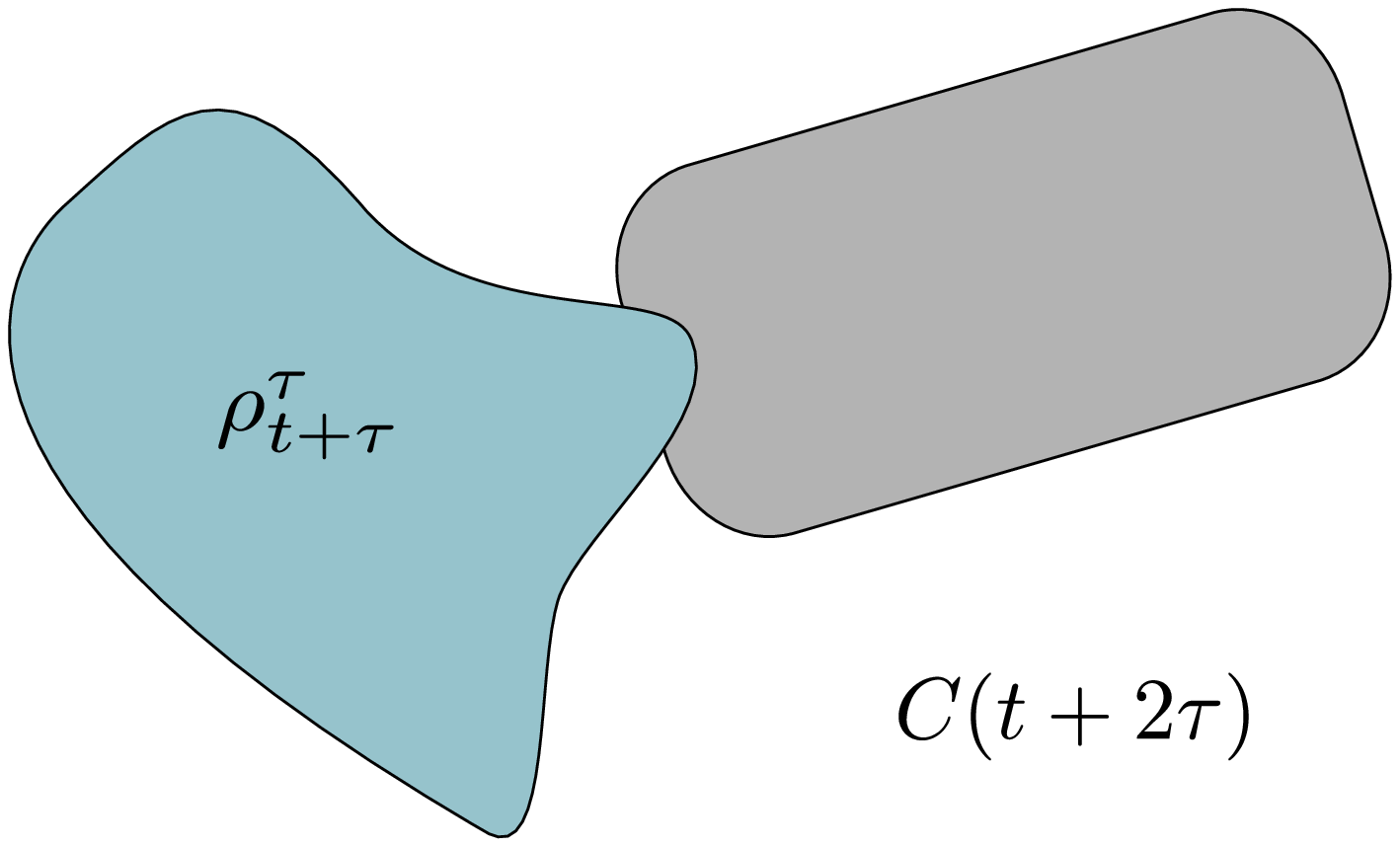}}
  {\includegraphics[width=0.24\textwidth]{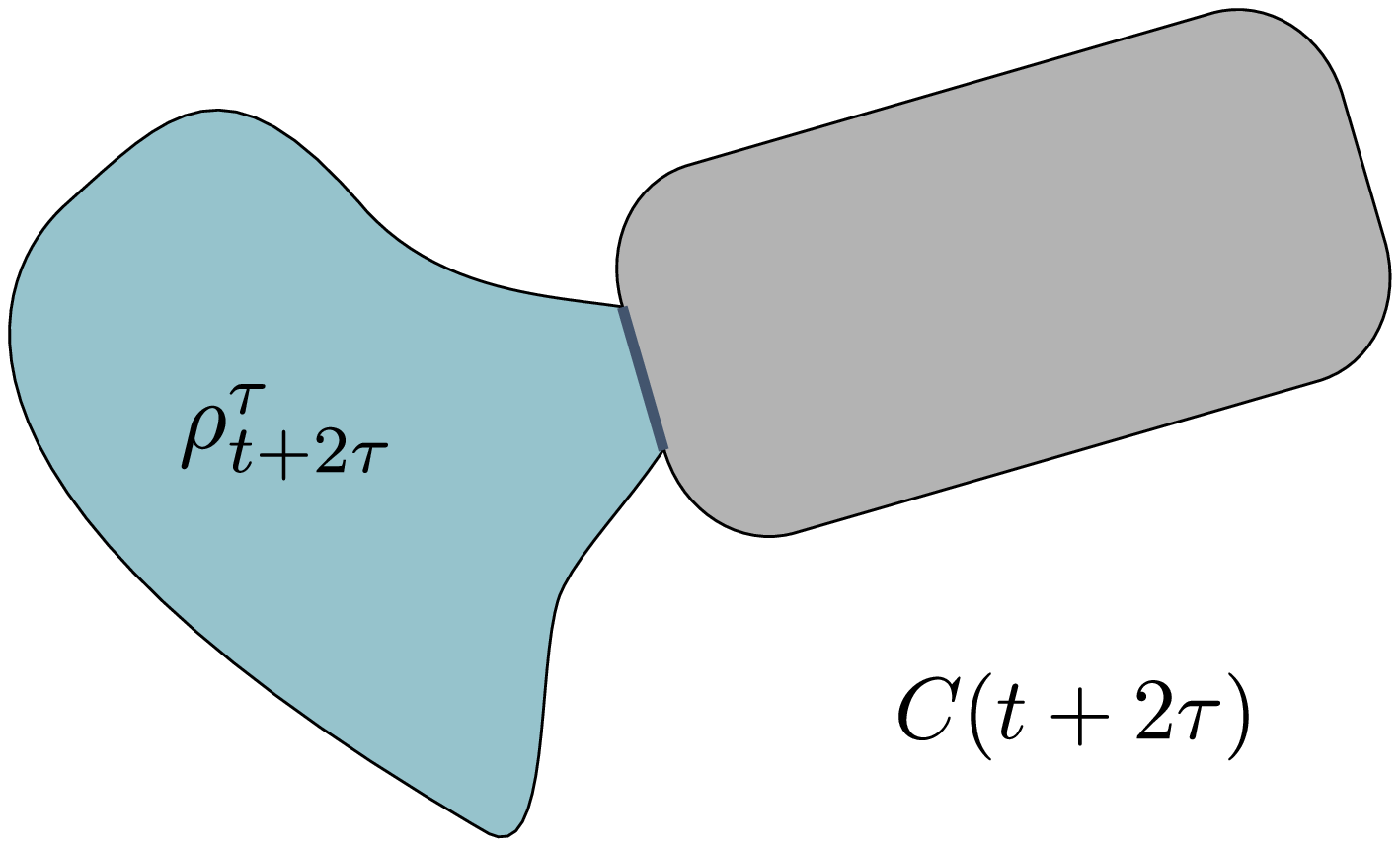}}
  \caption{One step of the catching-up algorithm. Here \( t = 2k\tau \), the gray
    rounded rectangle represents an obstacle.}
  \label{fig:catch}
\end{figure}

\paragraph{Structure of the paper.} In Section~\ref{sec:prelim} we introduce the notation,
recall basic properties of prox-regular sets and some standard
results concerning the geometry of \( \mathcal P_{2}(\mathbb{R}^{d}) \).
Section~\ref{sec:existence}, the most technical one, contains a proof of the
existence part of Theorem~\ref{thm:main}. The well-posedness part is proven in
Section~\ref{sec:cont}. We present in Section~\ref{sec:env} an application of
Theorem~\ref{thm:main} to the environment optimization problems.
Finally, Section~\ref{sec:numerics} contains some numerical computations
for the measure sweeping process~\eqref{eq:sp}.

\section{Preliminaries}
\label{sec:prelim}

\subsection{Notation}

Throughout this section, \( X \) and \( Y \) are metric spaces, \( U \) is an open subset of \( \mathbb{R}^{d} \).

\begin{longtable}{p{.15\textwidth} p{.85\textwidth}}
  \(\mathcal{P}(X) \) & the space of probability measures on \( X \)\\
  \(\mathcal{P}_{2}(\mathbb{R}^{d}) \) & the space of probability measures \( \mu \) on
                                         \( \mathbb{R}^{d} \) with \( \int |x|^{2}\d\mu<\infty \)\\
  \( \mathcal{P}_{c}(\mathbb{R}^{d}) \) & the space of compactly supported
                                          probability measures on \( \mathbb{R}^{d} \)\\
  \( \mathcal{M}(X; \mathbb{R}^d) \) & the space of finite Radon vector measures on \(X\)\\
  \( C(X;Y) \) & the space of continuous maps \( f\colon X\to Y \) \\
  \( C^{k}(U) \) & the space of \( k \) times continuously differentiable
  maps \( f\colon U\to \mathbb{R} \) \\
  \( C_{c}^{k}(U) \) & the space of all compactly supported maps from \( C^{k}(U) \)\\
  \( \mathcal{K}(\mathbb{R}^{d}) \) & the space of compact subsets of \( \mathbb{R}^{d} \)\\
  \( \mathcal{K}_{r}(\mathbb{R}^{d}) \) & the space of compact \( r \)-prox-regular subsets of \( \mathbb{R}^{d} \)\\
  \( \wto \) & the weak (narrow) convergence on \( \mathcal{M}(X;\mathbb{R}^{d}) \)\\
  \( W_{2} \) & the squared Wasserstein distance on \( \mathcal{P}_{2}(\mathbb{R}^{d}) \)\\
  \( d_{H} \) & the Hausdorff distance on \( \mathcal{K}(\mathbb{R}^{d}) \)\\

  \( \|\cdot\|_{\infty} \) & the supremum norm on \( C(X;Y) \)\\
  \( \Lip (f) \) & the minimal Lipschitz constant of \( f\in C(X;Y) \)\\
  \(\bm B\) & the closed unit ball in \( \mathbb{R}^{d} \) centered at \( 0 \)\\
  \(a+r\bm B\) & the closed ball in \( \mathbb{R}^{d} \) with center \( a\in \mathbb{R}^{d} \)
  and radius \( r\geq 0 \)\\
  \( d_{A}(x) \) & the distance between a compact set \( A\subset \mathbb{R}^{d} \) and a point \( x\in \mathbb{R}^{d} \)\\
  \( P_{A}(x) \) & the projection map, i.e., \( P_{A}(x) = \{a\in A\;\colon\;|x-a|=d_{A}(x)\} \)\\
  \( A^{\circ} \) & the interior of \( A\subset \mathbb{R}^{d} \)\\
  \( \partial A \) & the boundary of \( A\subset \mathbb{R}^{d} \)\\
  \( A^{c} \) & the complement of \( A\subset \mathbb{R}^{d} \)\\
  \( N_{A}(x) \) & the proximal normal cone to \( A\subset \mathbb{R}^{d} \) at \( x \)
  
\end{longtable}

\subsection{Prox-regular sets}

We collect here, for the future references, some basic properties of \emph{prox-regular sets} (also called sets with positive reach).

\begin{definition}
A closed set \( S\subset \mathbb{R}^{d} \) is called \textbf{\( r \)-prox-regular}, for \( r\in(0,+\infty] \), if the projection map \( P_{S} \) is single-valued and continuous within the open spherical neighborhood
\mbox{\( S +r\bm B^{\circ} =\{x\;\colon\;d_{S}(x)<r\} \).}
\end{definition}

A closely related notion of \emph{proximal normal} is defined as follows.
\begin{definition}
  Let \( S \) be a closed set and \( x \in S \).
  A vector \( v\in \mathbb{R}^{d} \) is called a \textbf{proximal normal} to the set \( S \) at
   \( x \) if there exists \( \sigma =  \sigma(x,v)\geq 0 \) such that
\begin{displaymath}
  \langle v, y-x\rangle \leq \sigma\, |y-x|^2\quad\forall y\in S.
\end{displaymath}
The set \( N_{S}(x) \) consisting of all such \( v \) defines the \textbf{proximal normal cone} to \( S \) at \( x \).
\end{definition}

The prox-regular sets can be characterized in several equivalent ways (see, e.g.,~\cite{Thibault2008}) gathered in the proposition below.

\begin{proposition}
  \label{prop:proxreg}
  The following assertions are equivalent:
  \begin{enumerate}
    \item[\( (a) \)] \( S \) is \( r \)-prox-regualar;
    \item[\( (b) \)] for any \( x\in S \), each nonzero proximal normal \( v\in N_{S}(x) \) is realized by an
      \( r \)-ball, i.e.,
      \begin{displaymath}
        \langle v, y-x\rangle \leq \frac{|v|}{2r}\, |x-y|^2\quad \forall y\in S;
      \end{displaymath}
    \item[\( (c) \)] \( d^{2}_{S} \) is continuously differentiable over \( S+r\bm B^{\circ} \).
  \end{enumerate}
      Moreover, one has
      \begin{displaymath}
        \frac{1}{2} \nabla d^{2}_{S}(x) = x - P_{S}(x)\quad \forall x\in S + r\bm B^{\circ},
      \end{displaymath}
      and, for any positive \( r' < r \),
      \begin{displaymath}
        \left|P_{S}(x) - P_{S}(y)\right|\leq \frac{r}{r-r'}\,|x-y|\quad \forall x,y\in S +r'\bm B^{\circ}.
      \end{displaymath}
    \end{proposition}







\subsection{Space of measures}

Here we briefly recall basic facts about the Wasserstein space \(
\mathcal{P}_{2}(\mathbb{R}^{d} ) \). 
The corresponding proofs can be found, e.g., in~\cite{Ambrosio2005,Villani2009}.

\noindent
\paragraph*{Space of measures as a metric space.}
Recall that \( \mathcal P_{2}(\mathbb{R}^{d}) \) is a \emph{complete separable metric
space} when equipped
with the quadratic \emph{Wasserstein distance}
\begin{equation}
  \label{eq:wasserstein}
  W_{2}(\mu,\nu) = \left(\inf_{\pi\in\Gamma(\mu,\nu)}\int|x-y|^{2}\d \pi(x,y)\right)^{1/2}.
\end{equation}
The infimum above is taken over the set \( \Gamma(\mu,\nu) \) of all transport plans between the measures
\( \mu \) and \( \nu \). Recall that a transport plan \( \Pi\in \Gamma(\mu,\nu) \) is a probability measure on \( \mathbb{R}^{d}\times \mathbb{R}^{d} \) whose projections on the first and the second factor are
\( \mu \) and \( \nu \), respectively. In other words, 
\begin{displaymath}
  \pi^{1}_{\sharp}\Pi = \mu,\quad \pi^{2}_{\sharp}\Pi=\nu,
\end{displaymath}
where \( \pi^{1}, \pi^{2} \) are the projection maps on the factors. Here
$\sharp$ is the pushforward functor, which works as follows: for any Borel map \( f\colon X\to Y \)
between metric spaces, it generates a map \( f_{\sharp}\colon \mathcal{P}(X)\to\mathcal{P}(Y) \)
by the rule
\begin{displaymath}
  (f_{\sharp}\mu)(A)\doteq \mu\left(f^{-1}(A)\right),\quad \text{for all Borel sets }A\subset Y.
\end{displaymath}

Remark that the minimum in~\eqref{eq:wasserstein} can always be achieved. Any transport plan that provides the minimum is called \emph{optimal}. A plan \( \pi \) is optimal if and only if its support \( \spt \pi \) is contained in a \emph{cyclically monotone set} \( A\subset \mathbb{R}^{d}\times \mathbb{R}^{d} \), which means that any finite collection of points \( (x_{i},y_{i})\in A \),
\( i=1,\ldots,k \),
satisfies
\begin{displaymath}
  \langle y_{1},x_{2}-x_{1}\rangle + \langle y_{2},x_{3}-x_{2}\rangle +\cdots + \langle y_{k},x_{1}-x_{k}\rangle\leq 0.
\end{displaymath}

The convergence in the Wasserstein distance is slightly stronger then the weak
convergence of measures. More precisely, \( W_{2}(\rho_{k},\rho)\to 0 \) if and only if
\begin{displaymath}
  \int\phi\,d\rho_{k}\to\int\phi\,d\rho,
\end{displaymath}
for any continuous \( \phi \) satisfying \( \phi(x)\leq C\left(1+|x|^{2}\right) \)
for some \( C>0 \).

\paragraph{Space of measures as a length space.} Recall that the \emph{length} of a continuous
curve \( \rho\colon [a,b]\to \mathcal P_{2}(\mathbb{R}^{d}) \) is given by
\begin{displaymath}
  L(\rho) = \sup \sum_{i=1}^{N}W_{2}(\rho_{t_{i-1}},\rho_{t_{i}}),
\end{displaymath}
where the supremum is taken among all finite partitions \( a=t_{0}\leq t_{1}\leq \cdots\leq
t_{N}=b \) of the interval \( [a,b] \).

The space \( \mathcal P_{2}(\mathbb{R}^{d}) \) is a \emph{strictly intrinistic
length space}, meaning that any two measures \( \mu,\nu\in \mathcal
P_{2}(\mathbb{R}^{d}) \)
can be connected with a continuous curve whose length is \( W_{2}(\mu,\nu)
\) (see~\cite{Burago2001}). Such a curve is called a \emph{minimal geodesic} from
\( \mu \) to \( \nu \). Any
minimal geodesic joining \( \mu \) and \( \nu \) can be uniquely parametrized by \( t\in [0,1] \) so that
\begin{displaymath}
  W_{2}(\rho_{t},\rho_{s}) = |t-s|W_{2}(\mu,\nu).
\end{displaymath}
In what follows, by saying that \( \rho \) is a \emph{geodesic} joining \( \mu \) and \( \nu
\) we mean that \( \rho \) is a minimal geodesic from \( \mu \) to \( \nu \) parametrized in this way.

It is known that any geodesic \( \rho \) joining \( \mu \) and \( \nu \) takes the form
\( \rho_{t} = \left((1-t)\pi^{1}+t \pi^{2}\right)_{\sharp}\Pi \), where \( \Pi \) is an optimal plan between
\( \mu \) and \( \nu \). In particular, if the optimal plan is unique, the geodesic is unique as well.
If an optimal plan \( \Pi\in\Gamma(\mu,\nu) \) is realized by a transport map \( F\colon \mathbb{R}^{d}\to \mathbb{R}^{d} \), i.e.,
\begin{displaymath}
 F_{\sharp}\mu = \nu,\quad W_{2}(\mu,\nu)^{2} = \int |x-F(x)|^{2}\,d\mu(x),
\end{displaymath}
then \( \Pi = (\id,F)_{\sharp}\mu \) and thus the geodesic takes the form \( \rho_{t}=\left((1-t)\id + tF\right)_{\sharp}\mu \).

\paragraph{Curves in the space of measures.}
A map \( \rho\colon [0,T]\to \mathcal{P}_{2}(\mathbb{R}^{d}) \) is called Borel
measurable if \( t\mapsto \rho_{t}(B) \) is Borel measurable for any Borel set \( B\subset \mathbb{R}^{d} \).
Any Borel measurable \( \rho \) produces a measure \( \bm \rho \) on \( [0,T]\times
\mathbb{R}^{d} \) by the rule
\begin{displaymath}
  \int_{[0,T]\times \mathbb{R}^{d}}\phi(t,x)\d{\bm\rho(t,x)} \doteq
  \int_{0}^{T}\int_{\mathbb{R}^{d}}\phi(t,x) \d\rho_{t}(x)\d t,
  \quad \forall\phi \in C^{\infty}_{c}\left((0,T)\times \mathbb{R}^{d}\right).
\end{displaymath}
Below, we never distinguish between \( \rho \) and \( \bm \rho \).

\begin{lemma}
 \label{lem:rho_conv}
  If \( \rho^{k}\to \rho \)  in \( C\left([0,T];\mathcal P_{2}(\mathbb{R}^{d})\right) \)
  then \( \rho^{k}\wto \rho \) in \( \mathcal M\left([0,T]\times \mathbb{R}^{d}; \mathbb{R}\right) \)
\end{lemma}
\begin{proof}
  For any \( \phi\in C_{b}\left([0,T]\times \mathbb{R}^{d}\right) \), one has
  \begin{displaymath}
    \int \phi(\d{\rho^{k}} - \d{\rho}) = \int_{0}^{T}\int
    \phi(t,x)\left(\d{\rho^{k}_{t}(x)}-\d{\rho_{t}(x)}\right)\d t.
  \end{displaymath}
  Since \( \phi \) is bounded and \( \rho^{k}_{t}\to\rho_{t} \) in \( \mathcal
  P_{2}(\mathbb{R}^{d}) \) for each \( t\in [0,T] \), we conclude that \( f^{k}(t)\doteq  \int
  \phi(t,x)\left(\d{\rho^{k}_{t}(x)}-\d{\rho_{t}(x)}\right)\) converges to \( 0 \) for
  all \( t\in [0,T] \).
  On the other hand, \( |f^{k}|\leq 2\|\phi\|_{\infty} \), hence the
  assertion follows from Lebesgue's dominated convergence theorem.
\end{proof}

A curve \( \rho\colon [0,T]\to \mathcal P_{2}(\mathbb{R}^{d}) \) is called \emph{absolutely
continuous} if there exists a function \( g\in L^{1}([0,T]) \) such that
\begin{displaymath}
  W_{2}(\rho_{t},\rho_{s})\leq \int_{s}^{t}g(\tau)\d\tau,\quad \forall t\geq s,\quad s,t\in [a,b].
\end{displaymath}
If \( \rho\colon [0,T]\to \mathcal{P}_{2}(\mathbb{R}^{d}) \) is absolutely continuous, the limit
\begin{displaymath}
  |\rho'|(t) \doteq \lim_{\epsilon \to 0} \frac{W_{2}(\rho_{t+\epsilon},\rho_{t})}{|\epsilon|}
\end{displaymath}
exists for a.e. \( t\in [0,T] \) and is called the \emph{speed} (or the \emph{metric derivative}) of
\( \rho \) at \( t \). Absolutely continuous curves has finite length which can be
expressed as
\begin{displaymath}
  L(\rho) = \int_{0}^{T}|\rho'|(t)\d t.
\end{displaymath}

The following theorem from~\cite[Section 8.3]{Ambrosio2005} shows that
absolutely continuous curves on \(
\mathcal P_{2}(\mathbb{R}^{d}) \) are
completely characterized by the continuity equations.
In the statement, \( L^{2}_{\mu}(\mathbb{R}^{d};\mathbb{R}^{d}) \) denotes the space of \( \mu
\)-measurable vector fields \( v\colon \mathbb{R}^{d}\to \mathbb{R}^{d} \) such that
\begin{displaymath}
  \|v\|_{\mu}\doteq \left(\int_{\mathbb{R}^{d}} |v|^{2}\d\mu\right)^{1/2}<\infty.
\end{displaymath}
As usual, two maps \( v \) and \( v' \) are considered equivalent if they
coincide for \( \mu \)-a.e. \( x \).

\begin{theorem}
  \label{thm:AC-curves}
Let  \( \rho\colon [0,T]\to \mathcal P_{2}(\mathbb{R}^{d}) \) be an absolutely
continuous curve. Then there exists a Borel vector field \( (t,x)\mapsto v_{t}(x) \) such that
\begin{displaymath}
  v_{t}\in L^{2}_{\rho_{t}}(\mathbb{R}^{d};\mathbb{R}^{d}),\quad  \|v_{t}\|_{\rho_{t}}\leq
  |\rho'|(t)\quad
  \text{for a.e. } t\in [0,T],
\end{displaymath}
and the continuity equation
\begin{equation}
  \label{eq:conteq}
  \partial_{t}\rho_{t} + \nabla \cdot (v_{t}\rho_{t}) = 0
\end{equation}
holds in the sense of distributions.
Conversely, if a Borel map \( \rho\colon [0,T]\to \mathcal P_{2}(\mathbb{R}^{d}) \) satisfies
equation~\eqref{eq:conteq} for some Borel vector field \( v \) with
\begin{equation}
  \label{eq:vfield_bound}
  \int_{0}^{T}\|v_{t}\|^{2}_{\rho_{t}}\d t<\infty,
\end{equation}
then \( t\mapsto \rho_{t} \) admits an absolutely continuous representative \( t\mapsto  \varrho_{t} \) with
\begin{displaymath}
  |\varrho'|(t)\leq \|v_{t}\|_{\varrho_{t}}\quad
  \text{for a.e. } t\in [0,T].
\end{displaymath}
\end{theorem}

Any Borel vector field \( v
\) satisfying~\eqref{eq:conteq} is called a \emph{velocity} of the absolutely
continuous curve \( \rho \). If a velocity
\( v \) is sufficiently regular, e.g.,
\begin{equation}
  \label{eq:CL}
  v_{t}\in C(\mathbb{R}^{d};\mathbb{R}^{d})\quad \forall t\in [0,T]\quad\text{and}\quad
  \int_{0}^{T}\left(\|v_{t}\|_{\infty} + \Lip(v_{t})\right)\d t<\infty,
\end{equation}
it generates a map \( \Psi\colon [0,T]\times [0,T]\times \mathbb{R}^{d}\to \mathbb{R}^{d} \) by
the rule
\begin{displaymath}
  \Psi_{s,t}(x)\doteq y(t),\quad s,t\in [0,T],\quad x\in \mathbb{R}^{d},
\end{displaymath}
where \( y\colon [0,T]\to \mathbb{R}^{d} \) is a unique solution of
the Cauchy problem
\begin{displaymath}
  \dot y(t) = v_{t}\left(y(t)\right),\quad y(s)=x.
\end{displaymath}

This map \( \Psi \), called the \emph{flow} of \( v \), is well-defined and satisfies the identities
\begin{displaymath}
  \Psi_{s,s} = \id,\quad \Psi_{s,t} = \Psi_{\tau,t}\circ\Psi_{s,\tau}\quad \forall s,t,\tau\in [0,T].
\end{displaymath}
Moreover, for any \( s,t\in [0,T] \), the map \( \Psi_{s,t}\colon \mathbb{R}^{d}\to
\mathbb{R}^{d} \) is a homeomorphism with
\begin{displaymath}
  \Lip(\Psi_{s,t})\leq \exp\left(\int_{s}^{t}\Lip(v_{\tau})\,d\tau\right).
\end{displaymath}

Absolutely continuous curves generated by regular vector fields admit a nice
representation formula given in following theorem
(see~\cite[Proposition 4]{AmbrosioCrippa2014} and~\cite[Proposition 8.1.7]{Ambrosio2005}).

\begin{theorem}
  \label{thm:regAC}
Let \( \rho\colon [0,T]\to \mathcal{P}_{2}(\mathbb{R}^{d}) \) be an absolutely
continuous curve that starts at \( \theta \) and  whose velocity \( v \) satisfies~\eqref{eq:CL}. Then
\begin{enumerate}
  \item[\( (a) \)] \( \rho_{t} = \Psi_{0,t\sharp}\theta \) for all \( t\in [0,T] \),
  \item[\( (b) \)] \( \rho \) is a unique solution of~\eqref{eq:conteq} that satisfies \( \rho_{0}=\theta \).
\end{enumerate}
\end{theorem}

The following result shows that the boundedness of a regular vector field implies the Lipschitz continuity of the
corresponding curve. 
\begin{lemma}
  \label{lem:Lip1}
  Let \( \Psi \) be the flow of a bounded vector field \( w \) satisfying~\eqref{eq:CL}, and
  \( \theta\in \mathcal{P}_2(\mathbb{R} ^d) \).
  Then the curve \(\rho\colon [0,T]\to \mathcal{P}_2(\mathbb{R} ^d) \) defined by \(\rho_t =  \Psi_{0,t\sharp}\theta \) is \( \|w\|_\infty \)-Lipschitz.
\end{lemma}
\begin{proof}
  Take two time moments \( s,t\in [a,b] \) such that \( s< t \). Since \( \Psi_{s,t}
  \) is a transport map between  \( \rho_s \) and \(\rho_t \), we have
  \begin{displaymath}
    W_2^{2}(\rho_t,\rho_s) \leq \int |x - \Psi_{s,t}(x)|^2\d\rho_s(x).
  \end{displaymath}
  On the other hand,
  \begin{displaymath}
    \Psi_{s,t}(x) - x = \int_s^t w_\tau \left( \Psi_{s,\tau}(x) \right)\d \tau.
  \end{displaymath}
  Hence the boundedness of \( w \) implies the Lipschitz continuity.
\end{proof}

General absolutely continuous curves admit another useful representation
formula. To describe it, we first define, for every \( t\in [0,T] \), the
evaluation map \( e_{t}\colon \mathbb{R}^{d}\times
C\left([0,T];\mathbb{R}^{d}\right)\to \mathbb{R}^{d} \) by \(e_{t}(x,\gamma) = \gamma(t)\).

\begin{theorem}
  \label{thm:evaluation}
  Let \( \rho\colon [0,T]\to \mathcal{P}_{2}(\mathbb{R}^{d}) \) be an absolutely
  continuous curve and \( v \) be its velocity field such
  that~\eqref{eq:vfield_bound} holds. Then \( \rho_{t} =
  (e_{t})_{\sharp}\eta \) for a suitable Borel probability measure \( \eta \) on \( \mathbb{R}^{d}\times
  C\left([0,T];\mathbb{R}^{d}\right) \). This measure is concentrated on the set
  \( \Gamma \) of pairs \( (x,\gamma) \) such that \( \gamma \) is an absolutely continuous
  solutions of the equation \( \dot x(t) = v_{t}\left(x(t)\right) \), for a.e.
  \( t\in [0,T] \), with \( \gamma(0)=x \).
\end{theorem}

\subsection{Projecting measures on sets}

Let \( C\subset \mathbb{R}^{d} \) be a bounded \( r \)-prox-regular set. Consider all measures \( \rho \)
supported in \( C \):
\begin{displaymath}
  \mathscr{C} = \left\{ \rho\in \mathcal{P}_{2}(\mathbb{R}^{d})\;\colon\; \spt\rho \subset C  \right\}.
\end{displaymath}

\begin{lemma}
  \label{lem:projection}
  Let \( \theta\in \mathcal{P}_2(\mathbb{R} ^d) \) satisfy \( \spt \theta\subset C + r\bm B^\circ \).
  Then
  \begin{enumerate}[(i)]
    \item there exists \( \theta_C\in \mathscr{C} \) such that
      \( W_2(\theta,\theta_C) = \inf_{\rho\in \mathscr{C}} W_2(\theta,\rho) \);
    \item \( \theta_C \) is unique and given by \( \theta_C = (P_C)_\sharp\theta \);
    \item \( \theta \) and \( \theta_C \) are connected with the unique geodesic \( \rho_t = \left( (1-t)\id + tP_C \right)_\sharp\theta \).
  \end{enumerate}
\end{lemma}
\begin{proof}
  First note that \( P_C \) is single-valued and continuous on \( C + r\bm B^\circ \) by the definition of \( r \)-prox-regularity. Consider the transport plan \( \Pi \) given by
  \begin{displaymath}
    \Pi(A) = \theta \left(\left\{ x\in C+r\bm B^\circ\;\colon\; (x,P_C(x))\in A \right \}  \right).
  \end{displaymath}
  Its support belongs to \( \graph P_C \). Hence the cyclical monotonicity of \( \graph P_C \) would imply the optimality of \( \Pi \). Cyclical monotonicity can be expressed as follows:
  \begin{displaymath}
    \sum_i|x_i-P_C(x_i)|^2 \leq \sum_i |x_i-P_C(x_{i+1})|^2\quad\text{for any}\quad x_1,\ldots,x_k\in 
    C+r\bm B^\circ.
  \end{displaymath}
  This property clearly holds because
  \begin{displaymath}
    |x_i-P_C(x_i)| = d_C(x_i)\leq |x_i - P_C(x_{i+1})|.
  \end{displaymath}
  
  Let \( \theta_C'\in \mathscr{C} \) satisfy
  \( W_2(\theta_C,\theta)=W_2(\theta_C',\theta) \) and
  \( \Pi' \) be the corresponding optimal transport plan. Since \( \pi^1_\sharp\Pi' = \theta \), the previous identity can be rewritten as follows:
  \begin{displaymath}
    \int \left(|x-y|^2 - |x-P_C(x)|^2\right)\d\pi'(x,y) = 0.
  \end{displaymath}
  The integrand is nonnegative on \( \spt\Pi' \) because
  \( |x-y| \geq |x-P_C(x)| \) for all \( x\in C+r\bm B^\circ \) and \( y\in C \).
  Therefore, for \( \Pi' \)-a.e. \( (x,y) \), we have
  \begin{displaymath}
    |x-y|^2 = |x-P_C(x)|^2.
  \end{displaymath}
  Since in the open \( r \)-neighborhood of \( C \) the projection \( P_C \) is unique, we conclude that \( y=P_C(x) \), for \( \Pi' \)-a.e. \( (x,y) \). In other words, \( \Pi' \) is supported on \( \graph P_C \).
Now from \(
\pi^1_\sharp\Pi= \pi^1_\sharp\Pi'=\theta\) it follows that \( \Pi=\Pi' \).

There are two consequences of this fact: 1)  \( \theta_C=\theta_C' \) (so we established uniqueness) and 2) 
the optimal plan \( \Pi \) between \( \theta_C \) and \( \theta \) is unique. Recall that any geodesic between \( \theta \) and \( \theta_C \) takes the form \( \rho_{t}=\left( (1-t)\pi^1 + t\pi^2 \right)_\sharp\Pi  \),
where \( \Pi \) is an optimal plan between \( \theta \) and \( \theta_C \).
Thus, the geodesic is unique due to the uniqueness of the optimal plan.
\end{proof}

\begin{definition}
The measure \( \theta_C \) defined in the previous lemma is called a \textbf{projection} of \( \theta \) on \( C \).
\end{definition}

\section{Existence}
\label{sec:existence}

\subsection{Continuous approximation}

We consider two \emph{processes} on \( \mathbb{R} ^d \). The first one \( \Phi_{s,t}^\theta \) is the flow of the vector field \( 2\mathscr{V}(\theta) \). The second one \( \Psi_{s,t}^\tau \) is defined only
for \( x\in \bm C(s+\tau) + r\bm B^\circ \) and \( t\in [s, s+\tau] \) by
\begin{displaymath}
  \Psi^\tau_{s,t}(x) =
  \left[
  \left(
    1-\frac{t-s}{\tau}
  \right)\id
  +
  \frac{t-s}{\tau}P_{\bm C(s+\tau)}
\right](x).
\end{displaymath}
Both processes generate maps in the space of measures:
\begin{displaymath}
  \widetilde{\Phi}_{s,t}(\theta) = \left( \Phi^\theta_{s,t} \right)_\sharp\theta, \quad
  \widetilde{\Psi}^\tau_{s,t}(\theta) = \left( \Psi^\tau_{s,t} \right)_\sharp(\theta).
\end{displaymath}

We merge these maps to construct a curve \(t\mapsto \rho^\tau_{t}\) in the following
way:
\begin{equation}
  \label{eq:contrho}
  \rho^\tau_0 = \theta,\quad
  \rho^\tau_{t} =
  \begin{cases}
    \widetilde\Phi_{2k\tau,t}\left(\rho^\tau_{2k\tau}\right), & t\in
    \left[2k\tau,(2k+1)\tau\right],\\
    \widetilde\Psi^\tau_{(2k+1)\tau,t}\left(\rho^\tau_{(2k+1)\tau}\right) & t\in
    \left[(2k+1)\tau,(2k+2)\tau\right],
  \end{cases}
  \quad k = 0,1,\ldots
\end{equation}
Let us find a velocity of this curve. To that end, take
\begin{displaymath}
  w_{s}^\tau(x) \doteq \frac{P_{\bm C(s+\tau)}(x) - x}{\tau},\quad x\in\bm C(s+\tau) + r \bm B^{\circ},
\end{displaymath}
and note that
\begin{displaymath}
  \Psi_{s,t}^\tau(x) = x + (t-s) w_s^\tau(x).
\end{displaymath}
Hence the time dependent vector field \((t,x)\mapsto w^{\tau}_{s,t}(x) \) generating the map
\((t,x)\mapsto \Psi^{\tau}_{s,t}(x) \) satisfies
\begin{displaymath}
  \frac{d}{dt}\Psi^\tau_{s,t}(x) = w_{s,t}^\tau\left(\Psi^\tau_{s,t}(x)\right) = w_{s}^\tau(x).
\end{displaymath}

Thus we conclude that \( \rho^{\tau} \) satisfies the continuity equation with the vector field given by
\begin{equation}
  \label{eq:contv}
  v^\tau_t =
  \begin{cases}
    2\mathscr{V}\left(\rho^\tau_{2k\tau}\right), & t\in 
    \left[2k\tau,(2k+1)\tau\right),\\
    w_{(2k+1)\tau}^\tau\circ \left[\Psi^\tau_{(2k+1)\tau,t}\right]^{-1} & t\in
    \left[(2k+1)\tau,(2k+2)\tau\right],
  \end{cases}
  \quad k = 0,1,\ldots
\end{equation}

\subsection{Properties of the continuous approximation}

\begin{lemma}
  \label{lem:rhoLip}
  For all sufficiently small \( \tau \), the curve \( \rho^\tau\colon [0,T] \to \mathcal{P}_2(\mathbb{R} ^d) \) is well-defined and Lipschitz with constant \( 2(L+M) \). 
  Moreover, \( |v^{\tau}_{t}(x)|\leq 2(L+M) \) for all \( t \) and \( \rho^{\tau}_{t} \)-a.e. \( x \).
\end{lemma}
\begin{proof}
  \textbf{1.} Assume that \( \rho^{\tau} \) is well-defined up to a time moment \( s=2k\tau \) (we can always choose \( k=0 \)). Then \( \spt\rho^\tau_{s}\subset \bm C(s) \) because the image of \( \Psi^{\tau}_{s-\tau,s} \)
  belongs to \( \bm C(s) \). In order to construct \( \rho^{\tau} \) on the interval
  \( \left[s,s+2\tau\right] \), we must show that \( \spt \rho^{\tau}_{s+\tau} \) lies in
  the domain of \( \Psi^{\tau}_{s+\tau,t} \) for all sufficiently small \( \tau \). Indeed,
  by Lipschitz continuity of \( \bm C \),
  \begin{displaymath}
    \bm C(s) \subset \bm C\left(s+2\tau\right) + 2\tau M\bm B.
  \end{displaymath}
  This inclusion together with \( \spt\rho^\tau_{s}\subset \bm C(s) \) implies
\begin{equation}
  \label{eq:sptincl}
  \spt \rho^\tau_{s+\tau}\subset \Phi_{s,s+\tau}\left(\bm C(s)\right) \subset \bm C(s) + 2\tau L\bm B\subset
  \bm C\left(s+ 2\tau\right) + 2\tau(L+M)\bm B.
\end{equation}
Thus we conclude that
\(  \spt \rho^\tau_{s+\tau}\subset\bm C\left(s+ 2\tau\right) + r\bm B^\circ\), whenever \( \tau \) is sufficiently small.
This proves that \( \rho^{\tau} \) is well-defined on \( [0,T] \).

\textbf{2.} Let us estimate \( v^{\tau} \). For each \( t\in \left[s,s+\tau\right)  \) (here again
\(s=2k\tau\)),
we have
\begin{displaymath}
  |v^\tau_t(x)| = |2\mathscr{V}(\rho^{\tau}_{s})(x)| \leq  2L\quad \forall x\in \mathbb{R}^{d}.
\end{displaymath}
Now suppose that \( t\in [s+\tau,s+2\tau]\) and let
\begin{displaymath}
  \alpha^{*} \doteq \inf\left\{\alpha\;\colon\; \spt \rho^{\tau}_{s+\tau}\subset \bm C(s+2\tau)+\alpha\bm B\right\}.
\end{displaymath}
Then, by construction,
\begin{displaymath}
  |v^\tau_t(x)| \leq \frac{\alpha^{*}}{\tau}\quad \forall x\in \bm C(s+2\tau)+\alpha^{*} \bm B.
\end{displaymath}
It follows from~\eqref{eq:sptincl} that \( \alpha^{*}\leq 2\tau(L+M) \).
Thus, for each \( t\in [s+\tau,s+2\tau] \), we have
\begin{displaymath}
  |v^\tau_t(x)| \leq 2(L+M)\quad \forall x\in \bm C(s+2\tau) + \alpha^{*}\bm B.
\end{displaymath}
Since \( \spt \rho^{\tau}_{t} \subset \bm C(s+2\tau) + \alpha^{*}\bm B  \), for any \( t\in [s+\tau,s+2\tau] \),
we get the desired estimate on \( v^{\tau} \).
The lipschitzeanity of \( \rho^{\tau} \) now
follows from Lemma~\ref{lem:Lip1}.
\end{proof}


\subsection{Piecewise constant approximation}

Let us introduce the map
\begin{displaymath}
  R^\tau(t) = k\tau,\quad t\in [k\tau, (k+1)\tau),
\end{displaymath}
which can be roughly thought as a ``projection'' of \( t \)
on the mesh  \( \{k\tau\}_{k=0}^{\infty} \).

Taking the curves \( t\mapsto \rho^\tau_t \) and \( t\mapsto v^\tau_t \), we construct two piecewise constant curves \( t\mapsto \bar\rho^\tau_t \) and \( t\mapsto \bar v^\tau_t \) in the following way:
\begin{displaymath}
  \bar\rho^{\tau}_{t} =  \rho^\tau_{R^\tau(t)},\quad \bar v^{\tau}_{t}= v^\tau_{R^\tau(t)}, \quad t\in [0,T].
\end{displaymath}


The next lemmas establish the relationship between the continuous and the piecewise constant
approximations.

\begin{lemma}
  \label{lem:same-limit}
  Consider two families of vector measures \( E^\tau = v^\tau \rho^\tau \) and
  \( \bar E^\tau_t = E^\tau_{R^\tau(t)} \), where \( \rho^{\tau} \) and \( v^{\tau} \) are defined
  by~\eqref{eq:contrho} and~\eqref{eq:contv}. Assume that \( E^\tau\wto E \), \( \bar E^\tau \wto \bar E \),
  \( \rho^{\tau}\wto \rho \), \( \bar \rho^{\tau}\wto\bar \rho \).
  Then
  \begin{enumerate}
    \item[\((1)\)]  \( E = \bar E \) and \( \rho = \bar \rho \);
    \item[\((2)\)]  \( E = v\rho \) for some \( v\);
  \end{enumerate}

\end{lemma}
\begin{proof}
  \textbf{1.} Since \( C_{c}(\mathbb{R}^{d+1}) \) is dense in
  \( C_{0}(\mathbb{R}^{d+1}) \), two Radon measures \( E \) and \( \bar E \) coincide if
  \( \int \phi\cdot \d E = \int \phi\cdot \d{\bar E} \) for any \( \phi \in C^{\infty}_{c}(\mathbb{R}^{d+1}) \). Fix \( \phi\in C_{c}^{\infty}(\mathbb{R}
  ^{d+1}) \) and note that
  \begin{displaymath}
    \int \phi\cdot \d{E^\tau} - \int \phi\cdot \d{\bar E^\tau}
    = \int_{0}^{T}
    \Big(
      \int \phi \cdot v_t^\tau\d\rho^\tau_t
      -
      \int \phi \cdot \bar v_{t}^\tau\d{\bar \rho^\tau_{t}}
    \Big)
    \d t.
  \end{displaymath}

  If \( t\in \left[(2k+1)\tau,(2k+2)\tau\right) \),
  the definitions of \( \rho^\tau \) and \( v^\tau \) allow us to write
  \begin{align*}
    \int\phi\cdot v_t^\tau\d\rho_t^\tau &=
      \int \phi \cdot w_{(2k+1)\tau}\circ \left[\Psi_{(2k+1)\tau,t}^\tau\right]^{-1}
      \d{\left[\Psi^\tau_{(2k+1)\tau,t}\right]_\sharp\rho^\tau_{(2k+1)\tau}} \\
    &=
      \int \phi\circ \Psi_{(2k+1)\tau,t}^\tau \cdot w_{(2k+1)\tau}
    \d{\rho^\tau_{(2k+1)\tau}},\\
    \int\phi\cdot \bar v_t^\tau\d{\bar \rho_t^\tau} &=
      \int \phi\cdot w_{(2k+1)\tau} \d{\rho^\tau_{(2k+1)\tau}}.
  \end{align*}
  Hence, by Lemma~\ref{lem:rhoLip},
  \begin{align*}
    \Big|
      \int \phi \cdot v_t^\tau\d\rho^\tau_t
      -
      \int \phi \cdot \bar v_{t}^\tau\d{\bar\rho^\tau_{t}}
    \Big|
    \leq
    \int \left|\phi\circ\Psi^\tau_{(2k+1)\tau,t} - \phi\right|
    \cdot\left|w_{(2k+1)\tau}\right|
    \d{\rho_{(2k+1)\tau}^\tau}
    \\
    \leq
    2\Lip(\phi)(L+M)\int \left|\Psi^{\tau}_{(2k+1)\tau,t}(x)-x\right|\d\rho^\tau_{(2k+1)\tau}(x),
  \end{align*}
  Similarly, for \( t\in \left[2k\tau,(2k+1)\tau\right) \), we have
  \begin{align*}
    \int\phi\cdot v_t^\tau\d\rho_t^\tau &=
    2\int \phi \cdot \mathscr{V}(\rho_{2k\tau})
      \d{\left[\Phi^\tau_{2k\tau,t}\right]_\sharp\rho^\tau_{2k\tau}}
    =
    2\int \phi\circ \Phi_{2k\tau,t}^\tau \cdot \mathscr{V}(\rho_{2k\tau})\circ \Phi_{2k\tau,t}^\tau
    \d{\rho^\tau_{2k\tau}},\\
    \int\phi\cdot \bar v_t^\tau\d{\bar \rho_t^\tau} &=
    2\int \phi \cdot \mathscr{V}(\rho_{2k\tau})
    \d{\rho^\tau_{2k\tau}}.
  \end{align*}
  Again by Lemma~\ref{lem:rhoLip},
  \begin{align*}
    \Big|\int\phi\cdot v_t^\tau\d\rho_t^\tau - \int \phi \cdot \bar v_t^\tau\d{\bar\rho_t^\tau}\Big|
    &\leq
    2\Big|\int \phi\circ \Phi_{2k\tau,t}^\tau \cdot \left(\mathscr{V}(\rho_{2k\tau})\circ \Phi_{2k\tau,t}^\tau -  \mathscr{V}(\rho_{2k\tau})\right)
    \d{\rho^\tau_{2k\tau}}\Big|\\
    &+
    2\Big|\int \left(\phi\circ \Phi_{2k\tau,t}^\tau - \phi\right) \cdot \mathscr{V}(\rho_{2k\tau})
      \d{\rho^\tau_{2k\tau}}\Big|\\
    &\leq
    2\|\phi\|_{\infty}\int
    \left|
    \mathscr{V}(\rho_{2k\tau}^\tau)\circ \Phi_{2k\tau,t} - \mathscr{V}(\rho_{2k\tau}^\tau)
    \right|\d\rho^\tau_{2k\tau}\\
    &+2L\int \left|\phi\circ \Phi_{2k\tau,t} - \phi \right|\d\rho^\tau_{2k\tau}\\
    &\leq
    2L\left(\|\phi\|_{\infty}+\Lip(\phi)\right) \int \left|\Phi_{2k\tau,t}(x) - x\right|\d\rho^\tau_{2k\tau}(x).
  \end{align*}

  Recalling that 
  \begin{gather*}
    \left|\Phi_{2k\tau,t}(x) - x\right|\leq 2 \tau L,\quad t\in \left[2k\tau,(2k+1)\tau\right),\\
    \left|\Psi_{(2k+1)\tau,t}(x) - x\right|\leq 2 \tau (L+M),\quad t\in \left[(2k+1)\tau,(2k+2)\tau\right),
  \end{gather*}
  we conclude that
  \begin{displaymath}
   \Big|\int\phi\cdot v_t^\tau\d\rho_t^\tau - \int \phi \cdot \bar v_t^\tau\d{\bar\rho_t^\tau}\Big| \leq C\tau\quad \forall t\in [0,T],
 \end{displaymath}
 for some \( C>0 \) that does not depend on \( \tau \). As a consequence,
  \begin{displaymath}
    \Big|\int \phi\cdot \d{E^\tau} - \int \phi\cdot \d{\bar E^\tau}\Big|
    \leq TC\tau,
  \end{displaymath}
  which implies \( E = \bar E \).
  By the same arguments, one can show that \( \rho = \bar\rho \).

  \textbf{2.} We derive the last assertion from the properties of the Benamou-Brenier functional \( \mathcal{B}_{2} \) (see Appendix~\ref{sec:A-BB}). Since \( E^\tau =v^\tau\rho^\tau \), Proposition~\ref{prop:BB}(i) and Lemma~\ref{lem:rhoLip}
  yield
    \begin{displaymath}
      \mathcal{B}_2(\rho^\tau,E^\tau) = \frac{1}{2}\int |v^\tau|^2\d\rho^\tau \leq C,
    \end{displaymath}
    for some \( C>0 \) which does not depend on \( \tau \). Hence
    the lower semicontinuity of \( \mathcal{B}_{2} \) implies
    \begin{displaymath}
      \mathcal{B}_2(\rho, E)\leq \liminf_{\tau\downarrow 0} \mathcal{B}_2(\rho^\tau, E^\tau)<+\infty,
    \end{displaymath}
    which means, by Proposition~\ref{prop:BB}(iv), that there exists \( v \) such that \( E = v\rho \).
\end{proof}

\begin{lemma}
  \label{lem:Vlimit}
  Let \( u^{\tau}_{t} = \mathscr{V}(\rho^{\tau}_{t}) \) and \( \bar u^{\tau}_{t} =
  \mathscr{V}(\bar \rho^{\tau}_{t}) \). If \( \|\rho^{\tau}- \rho\|_{\infty}\to 0 \) then  \(
  \|\bar \rho^{\tau} - \rho\|_{\infty}\to 0 \) and the sequences of vector measures
  \( u^{\tau}\rho^{\tau} \) and \( \bar u^{\tau}\bar \rho^{\tau} \) converge to
  \( u\rho \), where \( u_{t} = \mathscr{V}(\rho_{t}) \).
\end{lemma}
\begin{proof}
According to Lemma~\ref{lem:rhoLip}, on has
\begin{align*}
  \|\bar \rho^{\tau}-\rho\|_{\infty} &\leq\|\bar \rho^{\tau}-\rho^{\tau}\|_{\infty} + \|\rho^{\tau}-\rho\|_{\infty}\\
                       &= \sup_{t\in [0,T]}W_{2}(\rho^{\tau}_{R^{\tau}(t)},\rho^{\tau}_{t}) + \|\rho^{\tau}-\rho\|_{\infty}\\
                       &\leq 2(L+M)\tau + \|\rho^{\tau}-\rho\|_{\infty}.
\end{align*}
This proves the first assertion.

Let us show that the limit of \( u^\tau\rho^\tau \) is \( u\rho \). Indeed, for any
    \( \phi\in C_{b}(\mathbb{R}^{d+1}) \), one has
\begin{align*}
  \int_{0}^{T}\int_{\mathbb{R}^{d}}\phi(t,x)&\cdot \mathscr{V}(\rho^\tau_t)(x)\d{\rho^\tau_t(x)}\d t-
  \int_{0}^{T}\int_{\mathbb{R}^{d}}\phi(t,x)\cdot \mathscr{V}(\rho_t)(x)\d{\rho_t(x)}\d t \\&=
  \int_{0}^{T}\int_{\mathbb{R}^{d}}\phi(t,x)\cdot \left[\mathscr{V}(\rho^\tau_t)(x) -
    \mathscr{V}(\rho_t)(x)\right]\d{\rho^\tau_t(x)}\d t \\&+
  \int_{0}^{T}\int_{\mathbb{R}^{d}}\phi(t,x)\cdot \mathscr{V}(\rho_t)(x)\left[\d\rho^{\tau}_{t}(x)-\d{\rho_t(x)}\right]\d t.
\end{align*}
The absolute value of the first integral from the right hand side is estimated
by
\begin{displaymath}
 L\|\phi\|_{\infty}\sup_{t\in [0,T]}W_{2}(\rho^{\tau}_{t},\rho_{t}),
\end{displaymath}
and thus tends to \( 0 \). The second integral converges to \( 0 \),
because \( (t,x)\mapsto u_t(x) \) is continuous and bounded.

It remains to check that \( \bar u^\tau\bar\rho^\tau\wto u\rho \). For any \( \phi\in
C_{b}(\mathbb{R}^{d+1}) \), one has
  \begin{align*}
     \int \phi\cdot \bar u^{\tau} \d {\bar \rho^{\tau}} - \int \phi\cdot
     u \d \rho =
     \int \phi \cdot u\left(\d{\bar \rho^{\tau}}-\d{\rho}\right)+
    \int \phi\cdot \left(\bar u^{\tau}-u^{\tau}\right) \d {\bar \rho^{\tau}}
    +
    \int \phi\cdot \left(u^{\tau}-u\right) \d {\bar \rho^{\tau}}.
  \end{align*}
  The first integral from the right-hand side, which can be rewritten as
  \begin{displaymath}
    \int_{0}^{T}\int_{\mathbb{R}^{d}}\phi(t,x)\cdot u_{t}(x)
    \left(\d{\bar\rho^{\tau}_{t}}-\d{\rho^{\tau}_{t}}\right)\d t,
  \end{displaymath}
  converges to \( 0 \) since \( \bar \rho^{\tau}_{t}\wto \rho_{t} \) for a.e. \( t\in [0,T]
  \). The absolute values of the last two integrals can be estimated by
\begin{displaymath}
 L\|\phi\|_{\infty}\sup_{t\in [0,T]}W_{2}(\bar \rho^{\tau}_{t},\rho_{t})\quad \text{and}\quad
 L\|\phi\|_{\infty}\sup_{t\in [0,T]}W_{2}(\rho^{\tau}_{t},\rho_{t}),
\end{displaymath}
respectively. This proves the last assertion.
\end{proof}


\subsection{Normal cone inclusion}

The curves \( \rho^{\tau} \) and the vector fields \( v^{\tau} \) are constructed so that
\begin{displaymath}
  \partial_{t}\rho^{\tau}_{t} + \nabla\cdot (v^{\tau}_{t}\rho^{\tau}_{t}) = 0.
\end{displaymath}
The sequence \( \rho^{\tau}\), being uniformly Lipschitz by Lemma~\ref{lem:rhoLip},
converges (up to a subsequence) to a Lipschitz map \( \rho \) in \(
C\left([0,T];\mathcal{P}_{2}(\mathbb{R}^{d})\right) \).
Lemmas~\ref{lem:rho_conv} and~\ref{lem:same-limit} yield that \( \rho^{\tau}\wto \rho \)
and \( v^{\tau}\rho^{\tau}\wto v\rho \) for some Borel map \( v \). In particular, it
follows that
\begin{displaymath}
  \partial_{t}\rho_{t} + \nabla\cdot (v_{t}\rho_{t}) = 0.
\end{displaymath}
The inclusion \( \spt\rho_t \subset \bm C(t) \) is a direct consequence of the
following lemma.
\begin{lemma}
  \label{lem:sptconv}
  Let \( \mu_{k}\in \mathcal{P}_{2}(\mathbb{R}^{d}) \) and \( \spt\mu_k\subset A+r_k\bm B \), where \( A \) is compact. If \( \mu_k\wto\mu \) and \( r_k \to 0 \) then \( \spt \mu\subset A \).
\end{lemma}
\begin{proof}
  The set \( U_n = \left( A+r_n\bm B \right)^c \) is open. Hence, for each \( n \), we have
  \begin{displaymath}
    0= \liminf_{k\to\infty} \mu_k(U_n)\geq \mu(U_n).
  \end{displaymath}
  This means that \( \mu \left( \bigcup_{n}U_n \right) = 0\). It remains to note that
  \begin{displaymath}
    \bigcup_{n}U_n = \Big(\bigcap_n (A+ r_n\bm B) \Big)^c = A^c,
  \end{displaymath}
  completing the proof.

\end{proof}

To prove that \( \rho \) is a solution of~\eqref{eq:sp}, it remains to establish
the following Proposition~\ref{prop:normalcone}, whose proof heavily relies on
the properties of the piecewise constant approximation.

\begin{proposition}
  \label{prop:normalcone}
  For a.e. \( t\in [0,T] \), there exists a set \( A_{t}\subset \mathbb{R}^{d} \)
  such that
  \begin{equation}
    \label{eq:normalcone}
    v_t(x) - \mathscr{V}(\rho_t)(x)\in -N_{\bm C(t)}(x)\quad \forall x\in A_{t}
  \end{equation}
  and \( \rho_{t}(A_{t})=1 \).
\end{proposition}

Consider a set-valued map \( t\mapsto \bar {\bm C}^{\tau}(t) \) defined by
\begin{displaymath}
  \bar {\bm C}^{\tau}(t) = \bm C\left((2k+2)\tau\right),\quad t\in \left(2k\tau,(2k+2)\tau \right].
\end{displaymath}
Given a measurable selection \( y(t)\in \bm C(t) \), we define in the same way
\begin{displaymath}
  \bar y^{\tau}(t)= y\left((2k+2)\tau\right)=y_{(2k+2)\tau},\quad t\in \left(2k\tau,(2k+2)\tau \right].
\end{displaymath}

Let us introduce the integral
\begin{align}
  J^\tau =   \int_0^T\int_{\mathbb{R}^{d}}a(t)b(x)
&\Big(\left\langle 
  \bar v^\tau_{t}(x) - \mathscr{V}\left(\bar\rho^\tau_{t}\right)(x), P_{\bar {\bm C}^{\tau}(t)}(x)-\bar y^{\tau}(t)
\right\rangle\notag\\
&-\frac{1}{2r}\big|\bar v^\tau_{t}(x) - \mathscr{V}\left(\bar\rho^\tau_{t}\right)(x)\big|\cdot
\big| P_{\bar {\bm C}^{\tau}(t)}(x)-\bar y^{\tau}(t)\big|^2\Big)
\d{\bar\rho^\tau_{t}(x)}\d t,
\label{eq:Jtau}
\end{align}
where \( a \) and \( b \) are nonnegative bounded Lipschitz functions. Our aim
is to pass to the limit in the integral as \( \tau\to 0 \). Without loss of
generality, we can consider only those \( \tau \) that satisfy \( (2N+2)\tau = T \)
for some \( N\in \mathbb{N} \).

Recall that
\begin{displaymath}
  \bar v^\tau_{t} - \mathscr{V}(\bar\rho^\tau_{t}) =
  \begin{cases}
    \mathscr{V}(\rho^\tau_{2k\tau}), & t\in 
    \left[2k\tau,(2k+1)\tau\right],\\
    w_{(2k+1)\tau}^\tau- \mathscr{V}\left(\rho^\tau_{(2k+1)\tau}\right) & t\in
    \left[(2k+1)\tau,(2k+2)\tau\right].
  \end{cases}
\end{displaymath}
Hence \( J^\tau \) can be rewritten as
\begin{multline*}
  \sum_{k=0}^{N} \int_{2k\tau}^{(2k+1)\tau}a(t)\d t
\int b(x)\Big( \big\langle \mathscr{V}\left(\rho_{2k\tau}^\tau\right)(x),
    P_{\bm C\left((2k+2)\tau\right)}(x)-y_{(2k+2)\tau}
  \big\rangle\\
  -\frac{1}{2r}\big|\mathscr{V}\left(\rho^\tau_{2k\tau}\right)(x)\big|\cdot
\big| P_{\bm C\left((2k+2)\tau \right)}(x)-y_{(2k+2)\tau}\big|^2\Big)
\d{\rho_{2k\tau}^\tau(x)}\\
  +
  \sum_{k=0}^{N}
  \int_{2k\tau}^{(2k+1)\tau}a(t+\tau)\d t
  \int b(x) \Big(\big\langle
  w^\tau_{(2k+1)\tau}(x)-\mathscr{V}\left(\rho_{(2k+1)\tau}^\tau\right)(x),P_{\bm C\left((2k+2)\tau\right)}(x)-y_{(2k+2)\tau}
  \big\rangle\\
      -\frac{1}{2r}\big|w^\tau_{(2k+1)\tau}(x) - \mathscr{V}\left(\rho^\tau_{(2k+1)\tau}\right)(x)\big|\cdot
\big| P_{\bm C\left((2k+2)\tau \right)}(x)-y_{(2k+2)\tau}\big|^2\Big)
  \d{\rho_{(2k+1)\tau}^\tau(x)}.
\end{multline*}
Since 
\begin{displaymath}
  \big|w^\tau_{(2k+1)\tau}(x) - \mathscr{V}\left(\rho^\tau_{(2k+1)\tau}\right)(x)\big|
  \geq 
\big|w^\tau_{(2k+1)\tau}(x)\big| - \big|\mathscr{V}\left(\rho^\tau_{(2k+1)\tau}\right)(x)\big|,
\end{displaymath}
we conclude that \(J^\tau\leq J^\tau_1 +J^\tau_2+J^\tau_3\), where 
\begin{multline*}
J^\tau_1 =   \sum_{k=0}^{N}\Big( \int_{2k\tau}^{(2k+1)\tau}a(t)\d t
\int b(x)\big\langle \mathscr{V}\left(\rho_{2k\tau}^\tau\right)(x),
    P_{\bm C\left((2k+2)\tau\right)}(x)-y_{(2k+2)\tau}
  \big\rangle\d{\rho_{2k\tau}^\tau(x)}\\
  -
  \int_{2k\tau}^{(2k+1)\tau}a(t+\tau)\d t
  \int b(x) \big\langle \mathscr{V}\left(\rho_{(2k+1)\tau}^\tau\right)(x),P_{\bm C\left((2k+2)\tau\right)}(x)-y_{(2k+2)\tau}
                            \big\rangle\d{\rho_{(2k+1)\tau}^\tau(x)}\Big),
\end{multline*}
\begin{multline*}
J^\tau_2 =   -\frac{1}{2r}\sum_{k=0}^{N}\Big( \int_{2k\tau}^{(2k+1)\tau}a(t)\d t
\int b(x)\big|\mathscr{V}\left(\rho_{2k\tau}^\tau\right)(x)\big|\cdot
    \big|P_{\bm C\left((2k+2)\tau\right)}(x)-y_{(2k+2)\tau}\big|^2
  \d{\rho_{2k\tau}^\tau(x)}\\
  -
  \int_{2k\tau}^{(2k+1)\tau}a(t+\tau)\d t
  \int b(x) \big| \mathscr{V}\big(\rho_{(2k+1)\tau}^\tau\big)(x)\big|\cdot
  \big| P_{\bm C\left((2k+2)\tau\right)}(x)-y_{(2k+2)\tau}\big|^2
  \d{\rho_{(2k+1)\tau}^\tau(x)}\Big),
\end{multline*}
\begin{multline*}
  J^\tau_3 = \sum_{k=0}^{N}
  \int_{2k\tau}^{(2k+1)\tau}a(t+\tau)\d t
  \int b(x) \Big(\big\langle w^\tau_{(2k+1)\tau}(x),P_{\bm C\left((2k+2)\tau\right)}(x)-y_{(2k+2)\tau}
  \big\rangle\\
      -\frac{1}{2r}\big|w^\tau_{(2k+1)\tau}(x)\big|\cdot
\big| P_{\bm C\left((2k+2)\tau \right)}(x)-y_{(2k+2)\tau}\big|^2\Big)
  \d{\rho_{(2k+1)\tau}^\tau(x)}.
\end{multline*}

We are going to show that \(J^\tau_1 = O(\tau)\) and \(J^\tau_2 = O(\tau)\).

\begin{lemma}
  \label{lem:lim1}
  Given \( \mu_{1},\mu_2\in \mathcal{P}_2(\mathbb{R} ^d) \) and a Lipschitz function \(
  b\in C(\mathbb{R}^{d}) \), suppose that
  \begin{enumerate}[\((a)\)]
    \item  there exists \( K\in \mathcal K_{r}(\mathbb{R}^{d}) \) such that
      \( \spt \mu_1 \cup \spt\mu_2 \subset K + \frac{r}{2}\bm B \);
    \item
    there exist a Borel measurable map
      \( \psi\colon \mathbb{R} ^d\to \mathbb{R} ^d \) and \( C>0 \) such that \( \mu_2 = \psi_\sharp\mu_1 \)
      and \( |x-\psi(x)|\leq C\tau \) for all \( x \);
  \end{enumerate}
  Then there exists \( C_{1}>0 \) such that
\begin{align*}
  &\Big|\int b(x) \left\langle \mathscr{V}(\mu_1)(x),P_{K}(x)-y \right\rangle\d\mu_1(x)-
  \int b(x)\left\langle \mathscr{V}(\mu_2)(x),P_{K}(x)-y \right\rangle\d\mu_2(x)\Big|
  \leq C_1\tau,\\
  &\Big|\int b(x) \left|\mathscr{V}(\mu_1)(x)\right|\cdot
  \left|P_{K}(x)-y\right|^2\d\mu_1(x)-
  \int b(x)\left| \mathscr{V}(\mu_2)(x)\right|\cdot \left
  |P_{K}(x)-y \right|^2\d\mu_2(x)\Big|
  \leq C_1\tau,
\end{align*}
for all \( y\in K + \frac{r}{2} \bm B \).
\end{lemma}
\begin{proof}
  \textbf{1.} We start with the first inequality. Note that 
  \begin{displaymath}
    \int b(x) \left\langle\mathscr{V}(\mu_2)(x),P_{K}(x)-y\right\rangle\d{\mu_2(x)} =
    \int b\circ \psi(x)\left\langle\mathscr{V}(\mu_2)\circ{\psi}(x),P_{K}\circ\psi(x)-y\right\rangle\d{\mu_1(x)}.
  \end{displaymath}
  Now taking \( x_1 = x \) and \( x_2 = \psi(x) \), we may write
\begin{align*}
  b(x_1)&\left\langle \mathscr{V}(\mu_1)(x_1),P_{K}(x_1) -y \right\rangle -
  b(x_2)\left\langle \mathscr{V}(\mu_2)(x_2),P_{K}(x_2)-y \right\rangle
  \\
        & =
        \left(b(x_1)- b(x_2)\right)\cdot
        \left\langle \mathscr{V}(\mu_1)(x_1),P_{K}(x_1) -y \right\rangle\\
        & +
        b(x_2)
          \left(
            \left\langle \mathscr{V}(\mu_1)(x_1),P_{K}(x_1)-y \right\rangle
            -
            \left\langle \mathscr{V}(\mu_2)(x_2),P_{K}(x_2)-y \right\rangle
          \right).
\end{align*}
The first term from the right-hand side can be estimated by 
\begin{displaymath}
  \Lip(b) L |x_1-x_2|\cdot |P_K(x_1)-y|.
\end{displaymath}
To deal with the second term, consider the difference
\begin{align*}
 \left\langle \mathscr{V}(\mu_1)(x_1),P_{K}(x_1)-y \right\rangle
 & -
 \left\langle \mathscr{V}(\mu_2)(x_2),P_{K}(x_2)-y \right\rangle
 \\
 & = 
  \left\langle \mathscr{V}(\mu_1)(x_1) - \mathscr{V}(\mu_2)(x_1), P_{K}(x_1)-y\right\rangle
  \\
 & +
  \left\langle \mathscr{V}(\mu_2)(x_1) - \mathscr{V}(\mu_2)(x_2), P_{K}(x_1)-y\right\rangle
  \\
 & +
 \left\langle \mathscr{V}(\mu_2)(x_2),P_{K}(x_1)-P_K(x_2) \right\rangle.
\end{align*}
Our assumptions imply that the first term from the right-hand side can be estimated by 
\( LW_2(\mu_1,\mu_2)|P_{K}(x_1)-y|\),
the second term by 
\( L|x_1-x_2|\cdot |P_{K}(x_1)-y| \), and
the third term by
\( L \left| P_{K}(x_1) - P_{K}(x_2) \right| \).

Since \( b \) is bounded on \( K+\frac{r}{2}\bm B \) and, for all \( x_1\in \spt\mu_1 \), we have
\begin{align}
  & \left| x_1 -x_2  \right| = \left| x -\psi(x)  \right| \leq C\tau,\notag\\
  & W_2(\mu_1,\mu_2)\leq \Big(\int |x-\psi(x)|^2\d\mu_1(x)\Big)^{1/2}\leq C\tau,\notag\\
  & \left|P_{K}(x_1) - P_{K}(x_2)\right|\leq 2 |x_1-x_2|\leq 2C\tau,\quad \text{(by
    Proposition~\ref{prop:proxreg})}
    \label{eq:PP}
  \\
  & \left|P_{K}(x_{1})-y\right|\leq \diam K +r.\notag
\end{align}
After combining all the estimates, we get the desired inequality.

\textbf{2.} The second inequality can be proven in the similar way. We begin with the identity
  \begin{displaymath}
    \int b(x) \left|\mathscr{V}(\mu_2)(x)\right|\cdot
    \left|P_{K}(x)-y\right|^2\d{\mu_2(x)} =
    \int b\circ \psi(x)\left|\mathscr{V}(\mu_2)\circ{\psi}(x)\right|\cdot
    \left|P_{K}\circ\psi(x)-y\right|^2\d{\mu_1(x)}.
  \end{displaymath}
Again, by taking \( x_1 = x \) and \( x_2 = \psi(x) \), we get
\begin{align*}
  b(x_1)&\left|\mathscr{V}(\mu_1)(x_1)\right|\cdot \left|P_{K}(x_1) -y \right|^2 -
  b(x_2)\left|\mathscr{V}(\mu_2)(x_2)\right|\cdot \left|P_{K}(x_2)-y \right|^2  \\
        & =
        \left(b(x_1)- b(x_2)\right)\cdot
        \left|\mathscr{V}(\mu_1)(x_1)\right|\cdot\left|P_{K}(x_1) -y \right|^2\\
        & +
        b(x_2)
          \left(
            \left| \mathscr{V}(\mu_1)(x_1)\right|\cdot\left|P_{K}(x_1)-y \right|^2
            -
            \left| \mathscr{V}(\mu_2)(x_2)\right|\cdot\left|P_{K}(x_2)-y \right|^2
          \right).
\end{align*}
The first term from the right-hand side can be estimated by
\(\Lip(b)CL\left(\diam K +r\right)^{2}\tau\). As for the second term, we have
\begin{align}
\left| \mathscr{V}(\mu_1)(x_1)\right|\cdot\left|P_{K}(x_1)-y \right|^2
-
\left| \mathscr{V}(\mu_2)(x_2)\right|\cdot\left|P_{K}(x_2)-y \right|^2 \notag\\
=
\big(\left| \mathscr{V}(\mu_1)(x_1)\right| - \left| \mathscr{V}(\mu_2)(x_2)\right|\big)
\cdot\left|P_{K}(x_1)-y \right|^2\notag\\
+
\left| \mathscr{V}(\mu_2)(x_2)\right|
\cdot
  \big(\left|P_{K}(x_1)-y \right|^2 - \left|P_{K}(x_2)-y \right|^2\big)
  \label{eq:VP}
\end{align}
We can easily estimate the first term in~\eqref{eq:VP} because
\begin{displaymath}
  \big|\left| \mathscr{V}(\mu_1)(x_1)\right| - \left| \mathscr{V}(\mu_2)(x_2)\right|\big|
  \leq
  \left| \mathscr{V}(\mu_1)(x_1) - \mathscr{V}(\mu_2)(x_2)\right|
  \leq
  L|x_{1}-x_{2}| + LW_{2}(\mu_{1},\mu_{2})\leq 2LC\tau.
\end{displaymath}
Thanks to~\eqref{eq:PP} and the identity
\begin{displaymath}
  \left|P_{K}(x_1)-y \right|^2 - \left|P_{K}(x_2)-y \right|^2 = 
\left|P_{K}(x_1)-P_{K}(x_2)\right|\cdot
\left|P_{K}(x_1)+ P_{K}(x_2)-2y \right|,
\end{displaymath}
we can estimate the second term in~\eqref{eq:VP} by $4LC\tau\left(\diam K
  +r\right)$. Combining all the estimates above, we obtain the desired inequality.
\end{proof}

\begin{lemma}
  \label{lem:lim2}
  Let \( a\colon [0,T] \to \mathbb{R} \) be Lipschitz, \( s\in [0,T-\tau] \), and  \( \alpha,\beta\in \mathbb{R} \). Then
    \begin{displaymath}
      \Big| \alpha\int_s^{s+\tau}a(t)\d t - \beta \int_{s}^{s+\tau}a(t+\tau)\d t\Big|\leq |\alpha|\Lip(a) \tau^2 + |\alpha- \beta|\cdot \|a\|_\infty\tau.
    \end{displaymath}
\end{lemma}
\begin{proof}
 After rearranging the left-hand side can be written as follows: 
    \begin{displaymath}
      \Big| \alpha\int_s^{s+\tau}\left[a(t)-a(t+\tau)\right]\d t +(\alpha- \beta) \int_{s}^{s+\tau}a(t+\tau)\d t\Big|.
    \end{displaymath}
    Now the required estimate easily follows from the Lipschitz continuity of \( a \).
\end{proof}

\begin{lemma}
  \label{lem:JO}
  One has \(J^\tau_1 = O(\tau)\), \(J^\tau_2 = O(\tau)\).
\end{lemma}
\begin{proof}
  \textbf{1.} We begin with \( J^\tau_1 \).
  Let us take 
  \begin{align*}
    \alpha &= \int b(x) \left\langle \mathscr{V}\left(\rho_{2k\tau}^\tau\right)(x),
  P_{\bm C\left((2k+2)\tau\right)}(x)-y_{(2k+2)\tau}
  \right\rangle\d{\rho_{2k\tau}^\tau(x)},
  \\
      \beta &= \int b(x) \left\langle \mathscr{V}\left(\rho_{(2k+1)\tau}^\tau\right)(x),
    P_{\bm C\left((2k+2)\tau\right)}(x)-y_{(2k+2)\tau}
  \right\rangle\d{\rho_{(2k+1)\tau}^\tau(x)}.
  \end{align*}
  We know that \( \spt \rho_{(2k+2)\tau}\subset \bm C\left((2k+2)\tau\right) \). Hence if \( \tau \)
  is small enough then 
  \begin{displaymath}
    \spt\rho_{2k\tau}\cup \spt\rho_{(2k+1)\tau}\subset \bm C\left((2k+2)\tau\right) +\frac{r}{2}\bm B.
  \end{displaymath}
  Lemma~\ref{lem:lim1} implies that \( |\alpha-\beta|\leq C_{1}\tau \) for some \( C_{1}>0 \).
  Now from Lemma~\ref{lem:lim2}
  it follows that
  \begin{equation}
    \label{eq:lemJ}
    \alpha \int_{2k\tau}^{(2k+1)\tau}a(t)\d t - \beta\int_{2k\tau}^{(2k+1)\tau}a(t+\tau)\d t = O(\tau^2).
  \end{equation}
  This gives \( J^\tau_1 = (N+1)O(\tau^{2}) = \frac{T}{2\tau}O(\tau^{2}) = O(\tau) \).

\textbf{2.} To deal with \( J^\tau_2 \) we take 
  \begin{align*}
    \alpha &= 
    \int b(x)\big|\mathscr{V}\left(\rho_{2k\tau}^\tau\right)(x)\big|\cdot 
    \big|P_{\bm C\left((2k+2)\tau\right)}(x)-y_{(2k+2)\tau}\big|^2
  \d{\rho_{2k\tau}^\tau(x)},\\
    \beta &= 
    \int b(x)\big|\mathscr{V}\left(\rho_{(2k+1)\tau}^\tau\right)(x)\big|\cdot 
    \big|P_{\bm C\left((2k+2)\tau\right)}(x)-y_{(2k+2)\tau}\big|^2
    \d{\rho_{(2k+1)\tau}^\tau(x)}.
  \end{align*}
  Then Lemma~\ref{lem:lim2} gives~\eqref{eq:lemJ} and, as a consequence, \(
  J^\tau_2=O(\tau) \), completing the proof.
\end{proof}

Since \(-w^\tau_{(2k+1)\tau}(x)\) is a proximal normal to \(\bm
C\left((2k+2)\tau\right)\) at \(P_{\bm C(2k+2)\tau}(x)\), we conclude that
\begin{multline*}
 \left\langle 
   w_{(2k+1)\tau}^\tau(x), P_{\bm C\left((2k+2)\tau\right)}(x)-y_{(2k+2)\tau}
\right\rangle  \\
-\frac{1}{2r}\big| w_{(2k+1)\tau}^\tau(x)\big|\cdot \big|P_{\bm C\left((2k+2)\tau\right)}(x)-y_{(2k+2)\tau}\big|^2
\leq 0,
\end{multline*}
for all \(x\in \bm C\left((2k+2)\tau\right)+r\bm B^\circ\). This means that \( J^\tau_3\leq 0 \). So we have
\begin{equation}
  \label{eq:Jbelow0}
  J^\tau + O(\tau)\leq 0.
\end{equation}

\begin{lemma}
  \label{lem:int-below0}
  Let \( y(\cdot) \) be a Lipschitz continuous selection of \( \bm C(\cdot) \) and \( a\in
  C(\mathbb{R}) \), \( b\in C(\mathbb{R}^{d}) \) be nonnegative bounded Lipschitz functions.
  Then
\begin{align*}
  \int_0^T\int_{\mathbb{R}^{d}} a(t)b(x)
\Big(\left\langle 
  v_{t}(x) - \mathscr{V}\left(\rho_{t}\right)(x), x-y(t)
\right\rangle
&-\sigma(t,x)\cdot
\big| x-y(t)\big|^2\Big)
\d{\rho_{t}(x)}\d t\leq 0,
\end{align*}
for some nonnegative Borel map \( \sigma\colon [0,T]\times \mathbb{R}^{d}\to \mathbb{R} \).
\end{lemma}
\begin{proof}
  \textbf{1.} We shall prove the lemma by passing to the limit in~\eqref{eq:Jbelow0} as \( \tau\to 0 \). But
  first, let us show that \(\bar E^{\tau} \doteq \bar \sigma^{\tau}\bar \rho^{\tau} \) with
\begin{displaymath}
  \bar\sigma^\tau(t,x) = \frac{1}{2r}\big|\bar v_{t}^\tau(x) - \mathscr{V}\left(\bar \rho_{t}^\tau\right)(x)\big|
\end{displaymath}
  tends to \( \sigma\rho \) for some Borel map \( \sigma \). Since all \( \bar E^{\tau} \) are supported on the compact set
\begin{displaymath}
  \mathcal C_{r} = \left\{(t,x)\,\colon\, x\in \bm C(t) + \frac{r}{2}\bm B,\; t\in [0,T]\right\}
\end{displaymath}
and, by Lemma~\ref{lem:rhoLip}, their total variations are uniformly bounded:
\begin{displaymath}
  \|\bar E^{\tau}\| \doteq \int\bar \sigma^{\tau}\d{\bar \rho^{\tau}}\leq \frac{1}{2r}\int_{0}^{T}(3L+2M)\d t,
\end{displaymath}
we conclude that \( \bar E^{\tau} \) weakly converges (up to a subsequence) to some nonnegative measure \( E \).
  As in Lemma~\ref{lem:same-limit}, we the corresponding
  Benamou-Brenier functional is uniformly bounded:
  \begin{displaymath}
    \mathcal{B}_{2}(\bar \rho^{\tau},\bar E^{\tau})
    = \frac{1}{2}\int|\bar \sigma^{\tau}|^{2}\d{\bar
      \rho^{\tau}}
    \leq \int_{0}^{T}\!\!\left(3L+2M\right)^{2}\d t.
  \end{displaymath}
  Hence the lower semicontinuity of \( \mathcal B_{2} \)
  (Proposition~\ref{prop:BB}) implies that \( \mathcal{B}_{2}(\rho,E)<+\infty \), and
  therefore \( E=\sigma\rho \), for a Borel map \( \sigma \).

  \textbf{2.} Let us show that we get the desired limit if we replace
  \( P_{\bar {\bm C}^{\tau}(t)}(x) - \bar y^{\tau}(t) \) with \( f(t,x) = P_{\bm C(t)}(x) -y(t) \).
  Indeed, if \( \tau \) is small then \( \spt \bar\rho^\tau\subset \mathcal C_{r} \).  The function \( f \)
  is continuous inside \( \mathcal C_{r} \) thanks to Lemma~\ref{lem:prox}. Recalling
  Lemmas~\ref{lem:same-limit} and~\ref{lem:Vlimit}, we obtain
  \begin{displaymath}
    \int_{0}^{T}\int_{\mathbb{R}^{d}} a(t)b(x)\left(\left\langle v_{t}(x) -
      \mathscr{V}(\rho_{t})(x),f(t,x)\right\rangle - \sigma(t,x) f^{2}(t,x)\right)\d
  {\rho_{t}(x)}\d t
\end{displaymath}
in the limit. Since \( \spt\rho_t\subset \bm C(t) \), we have \( P_{\bm C(t)}(x)=x \) for
all \( x\in \spt\rho_t \). This gives the desired inequality.

\textbf{3.}  The function \( f \), being defined on a compact set, is uniformly continuous. In particular, for any \( \epsilon>0 \) there exists \( \delta \) such that for all \( \tau < \delta \)
  \begin{displaymath}
    \left|f(t,x) - f(R^{2\tau}(t),x)\right|\leq \epsilon \quad\forall (t,x)\in \mathcal{C}_{r}.
  \end{displaymath}
  Hence letting
  and \( \bar u^\tau_t = \bar v_t^\tau - \mathscr{V}(\bar \rho^\tau_t) \) we get
  \begin{displaymath}
    \Big|\iint a(t)b(x) \left[f(R^{2\tau}(t),x) - f(t,x)\right]\cdot \d {(\bar u^\tau \bar\rho^\tau)(t,x)}\Big|\to 0.
  \end{displaymath}
  Similarly using uniform continuity of \( |f|^2 \) we can show that
  \begin{displaymath}
    \Big|\iint a(t)b(x) \left[|f(R^{2\tau}(t),x)|^2 - |f(t,x)|^2\right]\cdot \d {(\bar \sigma^\tau \bar\rho^\tau)(t,x)}\Big|\to 0,
  \end{displaymath}
  which completes the proof.
\end{proof}

To proceed, we need one more technical lemma.
\begin{lemma}
  \label{lem:int-ineq}
  Let \( \mu \) be a Borel measure on \( \mathbb{R} ^d \) with compact support 
  and \( \phi\colon \mathbb{R} ^d\to \mathbb{R} \) 
  be a bounded Borel measurable function. If for any 
  smooth function \( a\colon \mathbb{R} ^d\to [0,1] \) we have
  \begin{displaymath}
    \int a(x)\phi(x)\d{\mu(x)}\leq 0
  \end{displaymath}
  then \( \phi(x)\leq 0 \) for \( \mu \)-a.e. \( x \).
\end{lemma}

\begin{proof}
  Let \( A = \left\{x\colon \phi(x)>0\right\} \). Since \( A \) is measurable and \(\mu\) is regular
  then, for any \( \epsilon>0 \), there exist a compact set \( F_\epsilon\subset A\) and an open set 
  \( G_\epsilon\supset A \) such that \( \mu(A\setminus F_\epsilon) <\epsilon  \) and 
  \( \mu(G_\epsilon\setminus A)<\epsilon \). By Urysohn's lemma there exists a smooth function 
  \( a_\epsilon\colon \mathbb{R} ^d \to [0,1] \) which is \( 1 \) on \( F_\epsilon \) and \( 0 \) outside of \( G_\epsilon \). Consider the obvious identity
  \begin{displaymath}
    \int a_\epsilon\phi\d\mu = \int_{F_\epsilon}a_\epsilon\phi\d\mu + 
    \int_{A\setminus F_\epsilon}a_\epsilon\phi\d\mu + 
    \int_{G_\epsilon\setminus A}a_\epsilon\phi\d\mu.
  \end{displaymath}
  Since \( \phi>0 \) on \( A\setminus F_\epsilon \) and \( \phi\leq 0 \) on 
  \( G_\epsilon\setminus A \), we obtain
  \begin{align*}
    &\int_{F_\epsilon}a_\epsilon\phi \d\mu = \int_{F_\epsilon}\phi \d\mu \geq \int_A\phi\d\mu - c\mu(A\setminus F_\epsilon),\\
    &\int_{A\setminus F_\epsilon}a_\epsilon\phi\d\mu \geq 0,\\
    &\int_{G_\epsilon\setminus A}a_\epsilon\phi\d\mu \geq 
    \int_{G_\epsilon\setminus A}\phi\d\mu \geq -c \mu(G_\epsilon\setminus A),
  \end{align*}
  where \( c \) is chosen so that \( |\phi(x)|\leq c \) for all \( x\in \mathbb{R} ^d \).
  These inequalities imply 
  \begin{equation}
    \label{eq:integral}
    \int a_\epsilon \phi\d\mu\geq \int_A\phi\d\mu - 2c\epsilon.
  \end{equation}
  
  Now suppose that \( \mu(A)>0 \). In this case \( \int_A \phi\d\mu > 0 \). Indeed,
  \( A \) contains a density point \( y \) of \( \phi \) (see, e.g.,~\cite[Theorem 5.8.8]{Bogachev2007}) and from
  \( \phi(y)=\lim_{r\downarrow 0}\frac{1}{\mu(y+r\bm B)}\int_{y+r\bm B}\phi\d\mu>0 \) it follows that
  \( \int_A\phi\d\mu\geq \int_{y+r\bm B}\phi\d\mu>0 \) for some \( r \). Thus, choosing \(\varepsilon\) small enough makes the right-hand side of~\eqref{eq:integral} strictly positive and leads to a contradiction.
\end{proof}

\begin{proofof}{Proposition~\ref{prop:normalcone}}
  Take a countable dense subset subset \( \{t_n\}_n \) of \( [0,T] \). Then, for each \( t_n \),
  choose a countable dense subset \( \{x_n^k\}_k \) of \( \bm C(t_n) \). The set of pairs
  \( \{(t_n,x_n^k)\}_{n,k} \) is also countable. For each \( (t_n,x_n^k) \) we consider the map 
  \( y_{n,k}\colon [0,T]\to \mathbb{R} ^d \) defined by
  \begin{displaymath}
    -\dot y_{n,k}(t) \in N_{\bm C(t)},\quad y_{n,k}(t_n) = x_n^k.
  \end{displaymath}
  This map is uniquely defined and Lipschitz continuous. We state that the set
  \( \{y_{n,k}(t)\}_{n,k} \) is dense in \(\bm C(t) \) for each \( t\in [0,T] \). Indeed,
  since \( \bm C \) is lower semicontinuous, for any \( t \) and any open ball
  \( x+\epsilon \bm B^\circ \)
  such that \( \bm C(t)\cap \{x+ \epsilon\bm B^{\circ}\}\ne \varnothing \) there exists \( t_n \) such that
  \( \bm C(t_n)\cap \{x+ \epsilon\bm B^{\circ}\}\ne\varnothing \). The latter set has nonempty interior and we can
  select from it some \( x_n^k \). Since \( t_n \) can be arbitrary close to \( t \) 
  then \( y_{n,k}(t)\in x +\epsilon \bm B^{\circ} \), as desired.

  Now, for each \( y_{n,k} \), we apply Lemma~\ref{lem:int-ineq} to the inequality established in
  Lemma~\ref{lem:int-below0}. Then we get 
  \begin{displaymath}
    \int_{\mathbb{R}^{d}} b(x)
    \Big(\left\langle 
        v_{t}(x) - \mathscr{V}\left(\rho_{t}\right)(x), x-y_{n,k}(t)
    \right\rangle
    -\sigma(t,x)\cdot
    | x-y_{n,k}(t)|^2\Big)
    \d{\rho_{t}(x)}\leq 0
  \end{displaymath}
  for all \( t\in [0,T]\setminus I_{n,k} \), where each \( I_{n,k} \) is a set of Lebesgue 
  measure zero. The union \( I \) of these sets also has measure zero. Since \( y_{n,k}(t) \) are
  dense in \( \bm C(t) \), we have
  \begin{displaymath}
    \int_{\mathbb{R}^{d}} b(x)\max_{y\in \bm C(t)}
    \Big(\left\langle 
        v_{t}(x) - \mathscr{V}\left(\rho_{t}\right)(x), x-y
    \right\rangle
    -\sigma(t,x)\cdot
    \left| x-y\right|^2\Big)
    \d{\rho_{t}(x)}\leq 0,\quad t\in [0,T]\setminus I.
  \end{displaymath}
  Using again Lemma~\ref{lem:int-ineq}, we obtain that for \( \rho_t \)-a.e. \( x \) 
  \begin{displaymath}
    \left\langle 
        v_{t}(x) - \mathscr{V}\left(\rho_{t}\right)(x), x-y
    \right\rangle \leq \sigma(t,x)\cdot
    \left| x-y\right|^2
    \quad \forall y\in \bm C(t).
  \end{displaymath}
  This completes the proof.
\end{proofof}

\section{Continuous dependence}
\label{sec:cont}

Before passing to the continuous dependence, let us prove assertions (2) and (3)
of Theorem~\ref{thm:main}.

\begin{lemma}
  \label{lem:vbound}
  For each solution \( \rho \) of~\eqref{eq:sp} assertions \( (1) \) and \( (2) \)
  of Theorem~\ref{thm:main} hold.
\end{lemma}
\begin{proof}
  {\bf 1.} Since the velocity \( v \) of \( \rho \) can be tweaked on a
\( \rho \)-negligible set without changing the solution of the continuity equation, we may assume that~\eqref{eq:normalcone} holds for all \( t \)
and \( x \).

{\bf 2.} Let us show that \( \rho = E_{\sharp}(\lambda\times \bm \eta) \), where \( \bm \eta \) is defined as
in Theorem~\ref{thm:evaluation}, \( \lambda \) is the one dimensional Lebesgue measure, and
\( E\colon (t,x,\gamma)\mapsto (t,\gamma(t)) \). Indeed, take \( A\subset [0,T]\times \mathbb{R}^{d} \) and
denote by \( A_{t} \) its slice \( \{\xi\;\colon\; (t,\xi)\in A\} \). Then
\begin{displaymath}
  \rho(A) = \int_{0}^{T}\rho_{t}(A_{t})\d t = \int_{0}^{T}\bm \eta\left(e_{t}^{-1}(A_{t})\right)\d t = \int_{0}^{T}\bm\eta(\tilde A_{t})\d t = (\lambda \times\bm\eta)(\tilde A),
\end{displaymath}
where \( \tilde A = \{(t,x,\gamma)\;\colon\; (t,\gamma(t))\in A\} \). It remains to note
that \( \tilde A = E^{-1}(A) \).

{\bf 3.} Let \( \Gamma \) be defined as in Theorem~\ref{thm:evaluation} and
\( \tilde A \subset [0,T]\times\Gamma \) be the set of all triples \( (t,x,\gamma) \) such
that \( \dot \gamma(t) \) exists and equals to \( v_{t}(\gamma(t)) \).
We are going to show that \( \tilde A \) is a set of full measure \( \lambda\times \bm \eta \).
By Fubini's theorem,
\begin{displaymath}
  (\lambda \times \bm \eta)(\tilde A) = \int_{\Gamma} \lambda(\tilde A_{(x,\gamma)})\d\bm \eta(x,\gamma),\quad\text{where}\quad
  \tilde A_{(x,\gamma)} = \{t\;\colon\; (t,x,\gamma)\in \tilde A\},
\end{displaymath}
Now we obtain \( (\lambda\times\bm\eta)(\tilde A)=T \) because \( \lambda(\tilde A_{(x,\gamma)}) = T \),
for all \( x \) and \( \gamma \).

{\bf 4.} Since \( \rho = E_{\sharp}(\lambda\times\bm \eta) \), we conclude that \( E(\tilde A) \) is a set
of full measure \( \rho \).
In other words, for \( \rho \)-a.e. \( (t,x) \) there exists a
solution \( y \) of the sweeping process
\begin{displaymath}
  \dot y(t) \in \mathscr{V}(\rho_t)(y(t)) - N_{\bm C(t)}(y(t)), \quad\text{for a.e. } t\in [0,T],
\end{displaymath}
such that \( y(t) = x \), \( \dot y(t) \) exists and equals to \( v_{t}(x) \).

{\bf 5.} Now we deduce from~\cite[Theorem 2.4]{SeneThibault2014} that
 \( |v_{t}(x)|\leq 2L +M \) for \( \rho \)-a.e. \( (t,x) \)  and from Proposition~\ref{prop:minnorm}
that
\begin{displaymath}
\xi + \eta\cdot v_{t}(x) = 0\quad \forall (\xi,\eta)\in N_{\graph \bm C}(t,x)
\end{displaymath}
for \( \rho \)-a.e. \( (t,x) \), when \( \graph \bm C \) is \( r' \)-prox-regular.
\end{proof}

Let \(\rho^1,\rho^2\colon [0,T] \to \mathcal{P}_{2}(\mathbb{R}^{d})\) be solutions of
the sweeping processes~\eqref{eq:sp} corresponding to the set-valued maps \( \bm
C^1,\bm C^2\colon [0,T]\to \mathcal{K}_r(\mathbb{R} ^d ) \), respectively. By
\(v^1_t,v^{2}_{t}\) we denote their velocity fields.


In order to prove the continuous dependence, we are going to differentiate
the function \(\bm r(t) \doteq \frac{1}{2}W_2^2(\rho^1_t,\rho^2_t)\). 
Since both curves \( \rho^{1} \) and \( \rho^{2} \) are absolutely continuous, we can use
the formula
  \begin{displaymath}
    \frac{d}{dt} W_{2}^2(\rho^{1}_t,\rho^{2}_t) = 2\iint \left\langle v^{1}_t(x)-v^{2}_t(y), x-y\right\rangle\d \Pi_{\rho^{1}_t,\rho^{2}_t}(x,y),
  \end{displaymath}
whose proof repeats that of Theorem 8.4.7~\cite{Ambrosio2005}
(we put it in Appendix~\ref{sec:Diff}, for completeness). The measure \(
\Pi_{\rho^{1}_{t},\rho^{2}_{t}} \) in the right-hand side denotes an \emph{optimal} plan between \( \rho^{1}_{t} \) and \( \rho^{2}_{t} \).


Let \( i=1,2 \). By definition,
\(
  \mathscr{V}(\rho^i_t)(x) - v^i_t(x) \in N_{\bm C^i(t)}(x),
\)
for a.e. \( t\in [0,T] \) and \( \rho^{i}_{t} \)-a.e. \( x\in \mathbb{R}^{d} \).
Since the values of \( \bm C^i \) are \( r \)-prox-regular,
Proposition~\ref{prop:proxreg}(b) implies that
\begin{equation}
  \label{eq:vVprox}
  \left\langle v^i_t(x)-\mathscr{V}(\rho^i_t)(x), x-y\right\rangle \leq \frac{1}{2r}
  \left|v^i_t(x)-\mathscr{V}(\rho^i_t)(x)\right|\, |x-y|^2,
\end{equation}
for a.e. \( t\in [0,T] \),  \(\rho^i_t\)-a.e. \( x\in \mathbb{R}^{d} \), and all  \(y\in \bm C^i(t)\).

According to Lemma~\ref{lem:Wderivative}, we have
\begin{align}
    \frac{1}{2}\frac{d}{dt} W_{2}^2(\rho^1_t,\rho^2_t) &=
    \iint \left\langle v^1_t(x) - v^2_t(y), x-y\right\rangle\d\Pi_{\rho^1_t,\rho^2_t}(x,y)\notag\\
    &=
    \iint \left\langle v^1_t(x) - \mathscr{V}(\rho^1_t)(x), x-y\right\rangle
    \d\Pi_{\rho^1_t,\rho^2_t}(x,y)\notag\\
    &+
    \iint \left\langle v^{2}_{t}(y) - \mathscr{V}(\rho^2_t)(y), y-x\right\rangle
    \d\Pi_{\rho^1_t,\rho^2_t}(x,y)\notag\\
    &+
    \iint \left\langle \mathscr{V}(\rho^1_t)(x)-\mathscr{V}(\rho^2_t)(y), x-y\right\rangle
      \d\Pi_{\rho^1_t,\rho^2_t}(x,y)\notag\\
  &= I_{1} + I_{2} + I_{3}.
      \label{eq:Wdiff}
\end{align}
for a.e. \( t\in [0,T] \).

We split the first integral \( I_{1} \) as follows:
\begin{align}
I_{1}&=
    \iint \left\langle v^1_t(x) - \mathscr{V}(\rho^1_t)(x), x-P_{\bm C^{1}(t)}(y)\right\rangle
    \d\Pi_{\rho^1_t,\rho^2_t}(x,y)\notag\\
    &+
      \iint \left\langle v^1_t(x) - \mathscr{V}(\rho^1_t)(x), P_{\bm C^{1}(t)}(y)-y\right\rangle
      \d\Pi_{\rho^1_t,\rho^2_t}(x,y).
      \label{eq:Wdiff1}
\end{align}
Note that \( t\mapsto P_{\bm C^{1}(t)}(y) \) is, in general, a set-valued map. Here,
slightly abusing the notation, we denoted by \( P_{\bm C^{1}(t)}(y)  \)
its measurable selection, which always exists. Indeed, since
\(  P_{\bm C^{1}(t)}(y) = \{y+d_{\bm C^{1}(t)}(y)\cdot \bm B\}\cap C(t) \), it is
measurable as an intersection of two measurable set-valued maps; hence it has a
measurable selection \( (t,y)\mapsto f(t,y) \).

Taking into account the inclusions
\begin{displaymath}
\spt \Pi_{\rho^1_t,\rho^2_t} \subset \spt\rho_t^1 \times \spt \rho_t^2\subset \bm C^{1}(t)\times \bm C^{2}(t),
\end{displaymath}
we deduce from~\eqref{eq:vVprox} and assertion (2) of Theorem~\ref{thm:main} that
\begin{align*}
  \iint \left\langle v^1_t(x) - \mathscr{V}(\rho^1_t)(x), x-P_{\bm C^{1}(t)}(y)\right\rangle
  \d\Pi_{\rho^1_t,\rho^2_t}(x,y)
  &\leq \frac{3L+M}{2r}W_{2}^{2}(\rho_{1},\rho_{2}),\\
      \iint \left\langle v^1_t(x) - \mathscr{V}(\rho^1_t)(x), P_{\bm C^{1}(t)}(y)-y\right\rangle
  \d\Pi_{\rho^1_t,\rho^2_t}(x,y)
  &\leq(3L+M)\Delta(t),
\end{align*}
where \(\Delta(t)\doteq d_{H}\left(\bm C^{1}(t),\bm C^{2}(t)\right) \). This gives
\begin{equation*}
  \label{eq:term1}
    I_{1}\leq (3L+M)\Delta(t) + \frac{3L+M}{4r}\bm r(t).
\end{equation*}
The same inequality holds for \( I_{2} \).

We rewrite the last integral \( I_{3} \) as the sum
\begin{displaymath}
  \int\!\! \left\langle \mathscr{V}(\rho^1_t)(x)-\mathscr{V}(\rho^2_t)(x), x-y\right\rangle\!
  \d\Pi_{\rho^1_t,\rho^2_t}(x,y)+
  \int\!\! \left\langle \mathscr{V}(\rho^2_t)(x)-\mathscr{V}(\rho^2_t)(y), x-y\right\rangle\!
  \d\Pi_{\rho^1_t,\rho^2_t}(x,y).
\end{displaymath}
The first integral above is bounded by
\begin{displaymath}
  \left(\int \left|\mathscr{V}(\rho^1_t)(x)-\mathscr{V}(\rho^2_t)(x)\right|^2
  \d\rho^1_t(x)\right)^{1/2}\cdot
  \left(\int \left|x-y\right|^2
  \d \Pi_{\rho^1_t,\rho^2_t}(x,y)\right)^{1/2}\leq L W_{2}^2(\rho^1_t,\rho^2_t),
\end{displaymath}
thanks to \( L \)-Lipschitz continuity of \( \mathscr{V}\colon
\mathcal{P}_{2}(\mathbb{R}^{d})\to C(\mathbb{R}^{d};\mathbb{R}^{d}) \).
The second one is bounded by
\( L W_{2}^2(\rho^1_t,\rho^2_t) \) due to \( L \)-Lipschitz continuity of \(
\mathscr{V}(\rho^2_t)\in C(\mathbb{R}^{d};\mathbb{R}^{d}) \). Thus, we have
\(
I_{3}\leq 4L\bm r(t)
\).

Plugging the above estimates into~\eqref{eq:Wdiff} gives
\begin{displaymath}
  \dot {\bm r}(t) \leq (6L+2M)\Delta(t) + \left(4L + \frac{3L+M}{2r}\right)\bm r(t).
\end{displaymath}
By Gr\"onwall's lemma, we obtain~\eqref{eq:mainestim}
which completes the proof of assertion (4). Finally, note that uniqueness in
assertion (1) is a direct consequence of the above estimate.

\section{Application to environment optimization}
\label{sec:env}
An important task of crowd dynamics is to understand how environment affects the
crowd motion. Consider a specific question: can an obstacle, such as a column,
placed at the right spot help the crowd to evacuate a room? 
We know that under some circumstances it happens in the real life~\cite{Helbing2002}.
Numerical experiments (see Section~\ref{sec:numerics}) show that this phenomenon, called Braess's paradox, can be reproduced in our model.
But can we find the best shape and position of the obstacle?

Let us formulate this problem within our framework.
Suppose that \( r \) is a fixed positive constant, \( \Omega \) a compact
\( r \)-prox-regular set that represents the region where the crowd can move,
\( \theta \) a compactly supported measure on \( \Omega \) which defines agents'
distribution. We assume that \( \Omega \) consists of two parts:
the safe \( S \) and the dangerous \( D \) regions.
The crowd leaves the dangerous region moving along a given nonlocal vector field
\( v_{t} = \mathscr{V}(\rho_{t}) \). Our aim is to place an obstacle \( O\subset\Omega \) so that the
number of agents staying in \( D \) by a time moment \( T \) were minimal.
As was discussed before, each obstacle defines the corresponding viability region
\( C = \Omega\setminus O \). We assume that admissible viability regions \( C \) belong to the set
\begin{displaymath}
  \mathcal C = \left\{C\in \mathcal K_{r}(\mathbb{R}^{d})\;\colon\; C\subset \Omega,\;\theta(C)=1 \right\}.
\end{displaymath}

The following theorem says that among all admissible viability regions one can always
choose an optimal one.

\begin{theorem}
  \label{thm:opt}
  Let \( \rho^{C}\colon [0,T]\to \mathcal{P}_{2}(\mathbb{R}^{d}) \) denote the
  trajectory of~\eqref{eq:sp} which corresponds to \( \mathbf C(t)\equiv C \), for
  \( C\in\mathcal{C} \). If \( D \) is open then the minimization problem
  \begin{displaymath}
    \min\left\{\rho^{C}_{T}(D)\;\colon\; C\in \mathcal{C}\right\}
  \end{displaymath}
  admits a solution.
\end{theorem}
\begin{proof}
  We know that \( C\mapsto \rho_{T}^{C} \) is continuous as a map
  \( \mathcal K_{r}(\mathbb{R}^{d})\to \mathcal P_{2}(\mathbb{R}^{d}) \). Since
  \( D \) is open, we conclude, by the Portmanteau theorem, that
  \( C\mapsto \rho_{T}^{C}(D) \) is lower semicontinuous as
  \( \mathcal K_{r}(\mathbb{R}^{d})\mapsto \mathbb{R} \). To complete the proof, it
  suffices to show that \( \mathcal C \) is compact.
  
  We can always
  extract from any sequence
  \( C_{n}\in \mathcal C \)  a
  subsequence converging to some compact set \(C\subset \Omega \) (see, e.g., \cite[p. 120]{Rockafellar1998}). By Theorem 4.13 in~\cite{FedererCurvMeasures},
  \( C\in\mathcal K_{r}(\mathbb{R}^{d}) \). Hausdorff convergence implies
  that for any \( k\in\mathbb N \) one may find \( n(k) \) such
  that \( C_{n(k)} \subset C +\frac{1}{k}\mathbf B \). This means that
  \( \theta\left(C +\frac{1}{k}\mathbf B \right) = 1 \), for each \( k\in \mathbb N \).
  Therefore,
  \( 1 = \lim\limits_{k\to\infty}\theta\left(C+\frac{1}{k}\bm B\right) = \theta\Big(\bigcap_{k=1}^{\infty}(C+\frac{1}{k}\bm B)\Big) = \theta(C) \).
\end{proof}

\section{Numerical computations}
\label{sec:numerics}

While continuous dependence on the moving set leads to existence
results in environment optimization problems, continuous dependence
on the initial measure provides an algorithm for computing trajectories
of~\eqref{eq:sp}. Indeed, let \( \rho \) be a trajectory issuing
from \( \theta\in \mathcal P_{2}(\mathbb{R}^{d}) \). We can always approximate
\( \theta \) by a discrete measure \( \theta_{N}\doteq\frac{1}{N}\sum_{i=1}^{N}\delta_{x_{i}} \)
(because such measures are dense in
\( \mathcal P_{2}(\mathbb{R}^{d}) \)~\cite{Villani2009}). The corresponding
trajectory \( \rho_{N} \), being absolutely continuous,
consists of discrete measures as well (note that several \( \delta \)-functions could
be glued into one along the way, but they can never be split again). Now,
we can easily compute \( \rho_{N} \) by applying the catching-up scheme.
Theorem~\ref{thm:main} shows that \( \rho_{N}\to \rho \) in \( C([0,T];\mathcal P_{2}(\mathbb{R}^{d})) \) as
\( N\to\infty \).

Below we provide computations for two simple models of crowd dynamics  taken
from~\cite{Mogilner1999} and~\cite{Colombo2011}.

\begin{example}[Attraction/repulsion model]
The first model~\cite{Mogilner1999}
corresponds to \( \mathscr V \) given by
 \begin{displaymath}
   \mathscr{V}(\mu)(x) = w(x) + \int K(x-y)\d\mu(y),
 \end{displaymath}
 where \( w\colon \mathbb{R}^{2}\to\mathbb R^{2} \) is a drift and \( K \) the
 attraction/repulsion kernel of the form
  \begin{displaymath}
    K(x) = - \frac{A_{a}x}{2a^{2}}\exp\Big(\!-\frac{|x|^{2}}{2a^{2}}\Big) +
    \frac{A_{r}x}{2r^{2}}\exp\Big(\!-\frac{|x|^{2}}{2r^{2}}\Big),
  \end{displaymath}
  Here \( a \) and \( r \) determine the attraction and repulsion ranges,
  \( A_{a} \) and \( A_{r} \) the
  attraction and repulsion intensities. It is common to take \( r<a \), so agents repulse each other at short distances and
  attract at large ones. One can
  easily verify that \( \mathscr V \) satisfies our assumptions if \( w \)
  is bounded and Lipschitz.

  For the computations presented in Figure~\ref{fig:mathbiol}, we choose
  \( A_{a}= 4 \), \( A_{r} = 7 \), \( a = 1/\sqrt 2 \), \( r=0.5 \),
  \( w\equiv -0.3 \), \( \tau=0.01 \).
  The moving set is given by
    \(\bm C(t) = \{x\in \mathbb{R}^{2}\;\colon\; f(t,x)\leq 0\}\)
  with
  \begin{displaymath}
    f(t,x) = -(x_{1}-0.5t+2)^2 - 4(x_{2}-0.5t+4)^2 + 2,
  \end{displaymath}
  that is, our obstacle is an ellipse crossing the crowd.
  We approximate the initial measure \( \theta \) (the Gaussian measure with mean
  \( (4,0) \) and variance \( \id \)) by a discrete measure
  \(\frac{1}{N}\sum_{i=1}^{N}\delta_{x_{i}} \) with \( x_{i} \) randomly distributed
  according to \( \theta \), \( N=300 \).
\end{example}

\begin{example}[Congestion model]
  The second model~\cite{Colombo2011} corresponds to the choice
 \begin{displaymath}
   \mathscr{V}(\mu)(x) = w(x)\cdot\psi\left(\int\eta(|x-y|)\d\mu(y)\right),
 \end{displaymath}
 where \( w\colon \mathbb{R}^{2}\to \mathbb{R}^{2} \) is a given vector field, \( \eta\colon \mathbb{R}\to \mathbb{R}_{+} \) is a smooth bell-shaped
 function, \( \psi\colon \mathbb{R}_{+}\to [0,1] \) is Lipschitz and non-increasing. The idea behind this model
 is that the velocity of an agent located at \( x \) decreases as the number
 of agents around \( x \)
 (estimated by \( \int\eta(|x-y|)\d \mu(y) \)) grows.

 To define the non-local vector field we choose the following functions:
 
\begin{minipage}[c]{0.4\textwidth}
  \includegraphics[width=\linewidth,keepaspectratio=true]{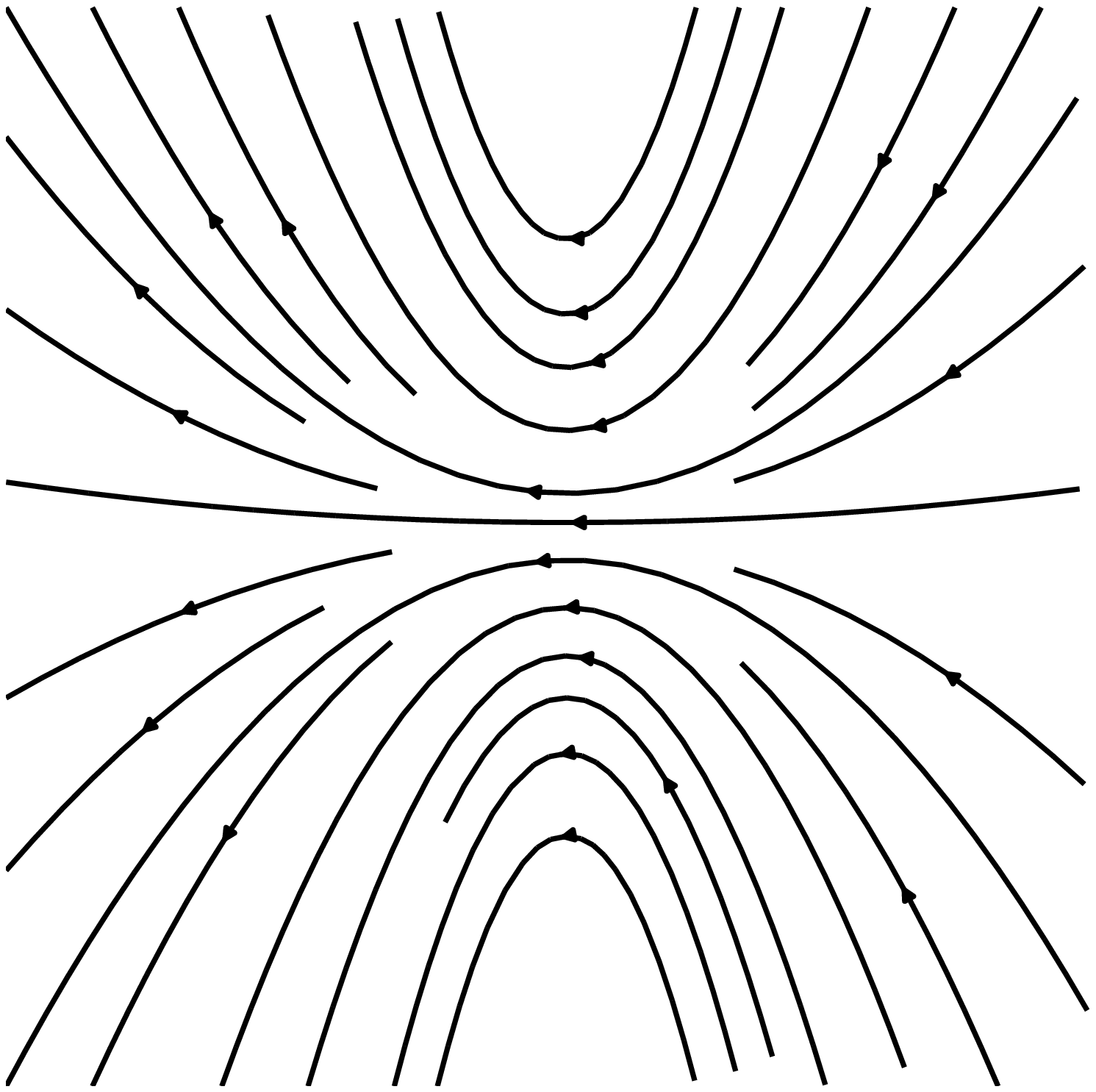}
\end{minipage}
\begin{minipage}[c]{0.5\textwidth}
  \begin{align*}
    &w(x) = -\frac{1}{2|x|}(1+x_{1}^{2},2x_{1}x_{2}) ,\\
    &\psi(r) = 1-\frac{2}{\pi}\arctan \kappa x^{2},\\
    &\eta(r) =
      \begin{cases}
        \frac{1}{\beta}e^{\frac{1}{(r/\epsilon)^{2}-1}}, & r<\epsilon,\\
        0,&\text{otherwise},
      \end{cases}\\
    & \epsilon = 0.3,\; \kappa = 1000, \; \beta = 0.466.
  \end{align*}
   \null
\par\xdef\tpd{\the\prevdepth}
\end{minipage}\\
Field lines of \( w \) are the parabolas depicted above. The moving set is given
by
\begin{gather*}
  \bm C(t) = \{x\in \mathbb{R}^{2}\;\colon\; f(t,x)\leq 0\} \setminus (I+\delta\bm B),
  \quad \text{where}\quad I = \big\{x = 0,\; |y|>b \big\},\quad b,\delta >0,\\
   f(t,x) =\! -\Big(\frac{(x_{1}-c_{1})\cos\omega t - (x_{2}-c_{2})\sin\omega t}{a_{1}}\Big)^{2}
   \!-\Big(\frac{(x_{1}-c_{1})\sin\omega t + (x_{2}-c_{2})\cos\omega t}{a_{2}}\Big)^{2}\!+1.
 \end{gather*}
 Here \( I+\delta \bm B \) models a wall with an exit and \( f \) an elliptic
 obstacle with semi-axis \( a_{1} \), \( a_{2} \) rotating
 around its center \( c=(c_{1},c_{2}) \). In our case, \( b = 0.6 \),
 \( \delta = 0.1 \), \( \tau = 0.01 \). The initial measure is absolutely
 continuous with density \(\frac{1}{32} \bm 1_{[2,6]\times[-4,4]} \). We approximate
 it by a discrete measure
  \(\frac{1}{N}\sum_{i=1}^{N}\delta_{x_{i}} \), where \( x_{i} \) are uniformly  distributed
  on the rectangle \( [2,6]\times[-4,4] \), \( N=300 \).  Solutions of~\eqref{eq:sp}
  for various \( c \), \( a \), \( \omega \) are presented in
  Figure~\ref{fig:colombo}. Note that, by the time moment \( T=20 \),  the dangerous region \( D = \{x>0\} \)
  contains \( 19.67 \% \) of the total mass if there
 are no obstacles, \( 14.67\% \) for the stationary obstacle, \( 0.00\% \)
  for the moving obstacle. Hence Braess's paradox may indeed occur in~\eqref{eq:sp}.

\end{example}
\begin{figure}
  {\includegraphics[width=0.19\textwidth]{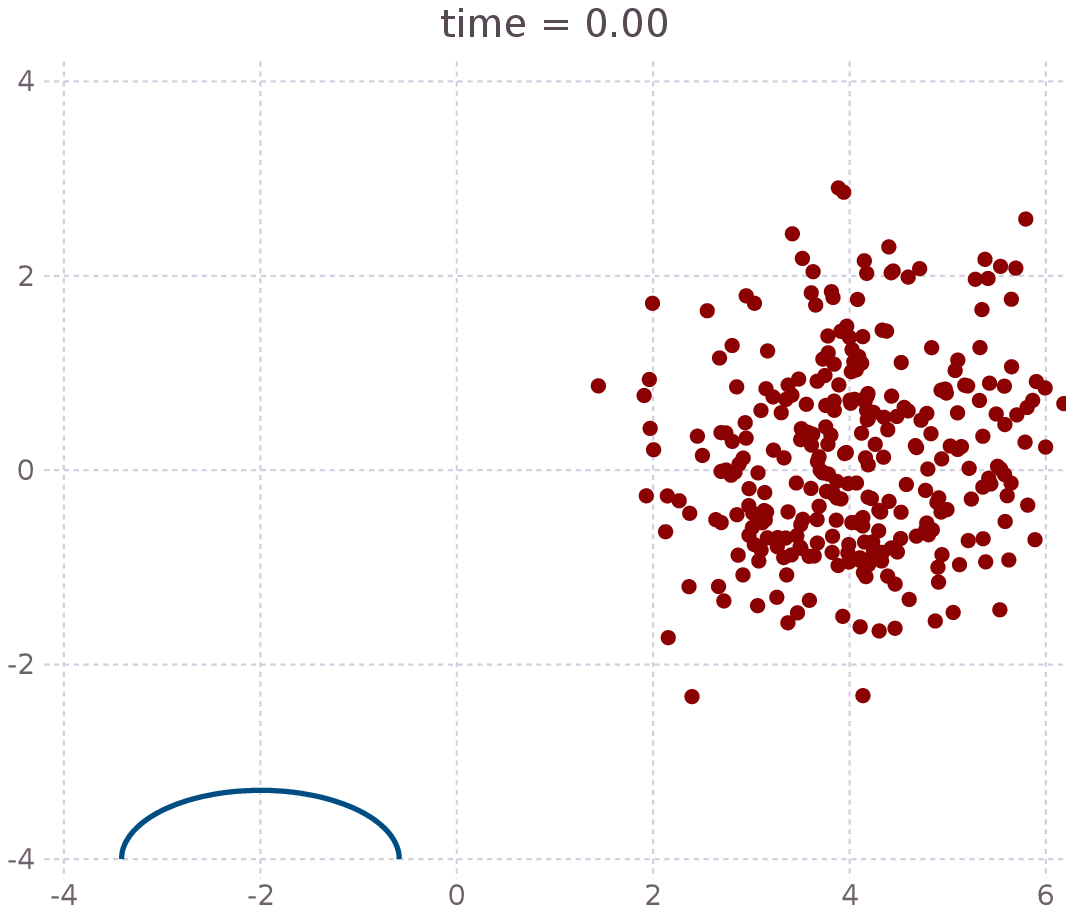}}
  {\includegraphics[width=0.19\textwidth]{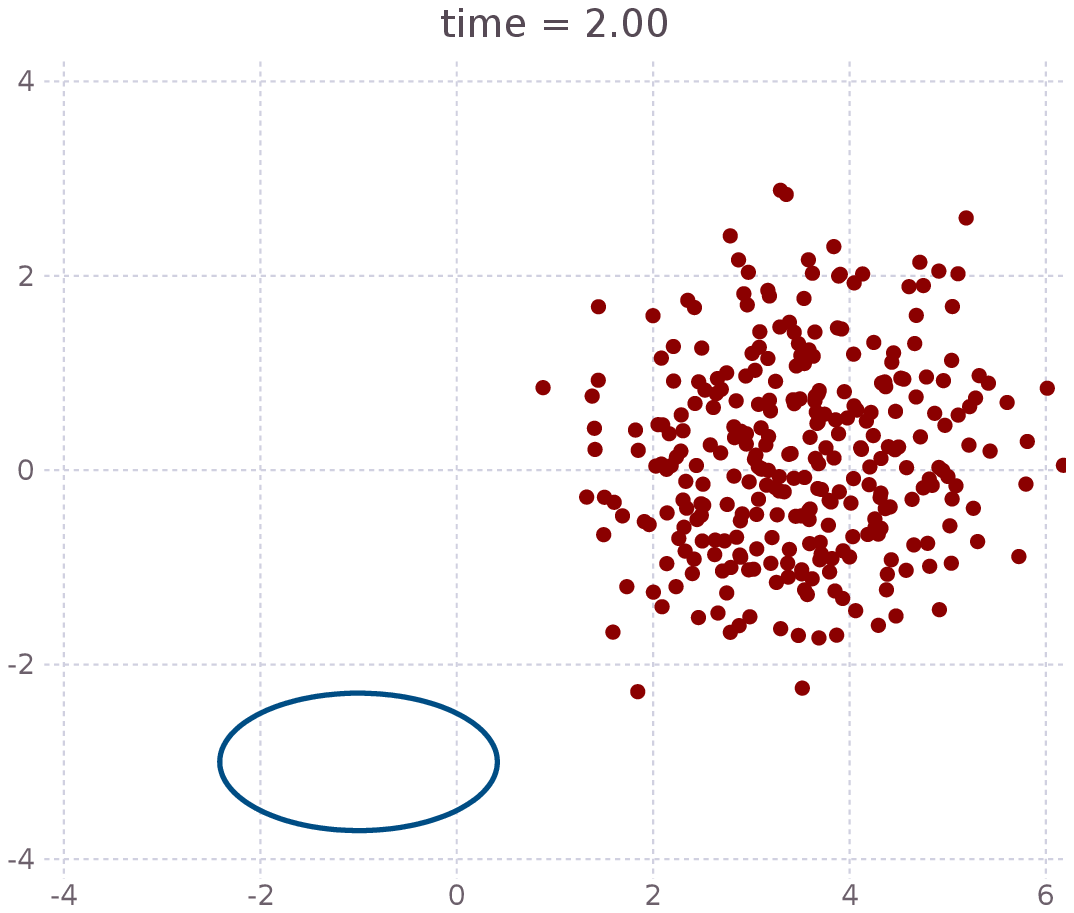}}
  {\includegraphics[width=0.19\textwidth]{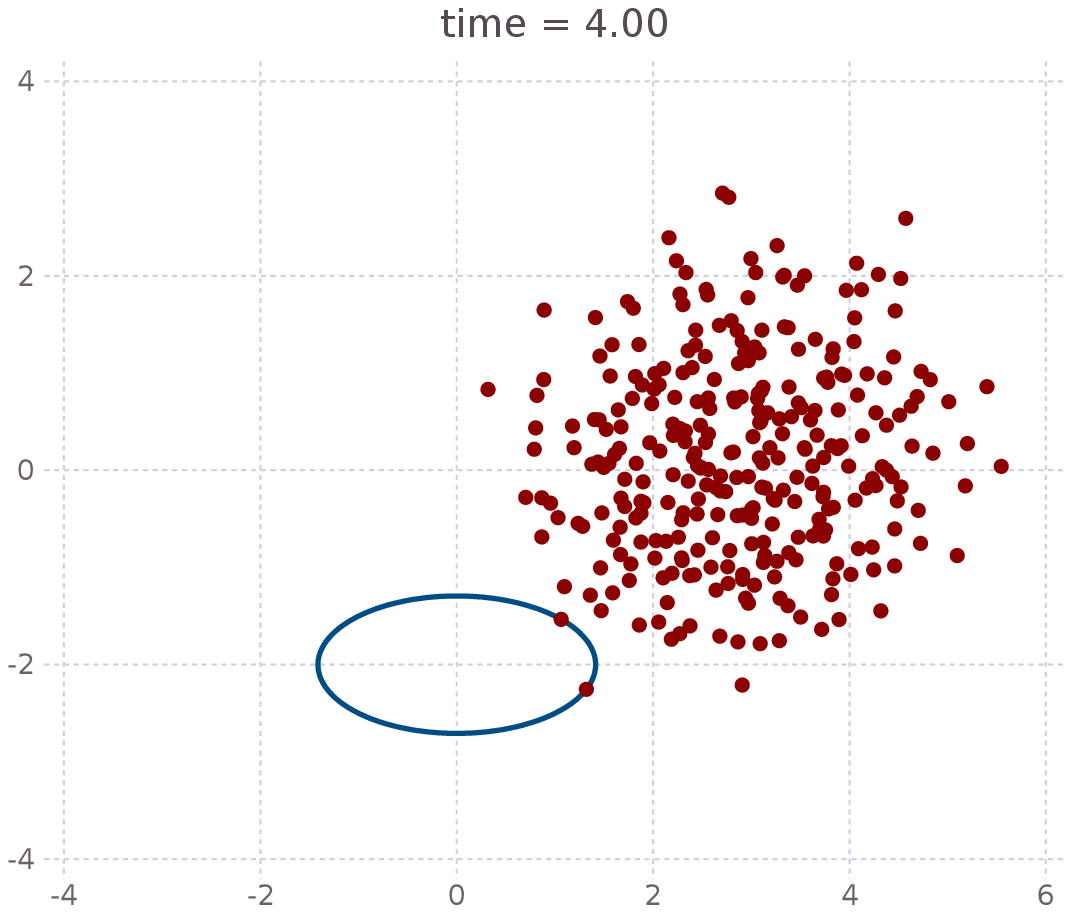}}
  {\includegraphics[width=0.19\textwidth]{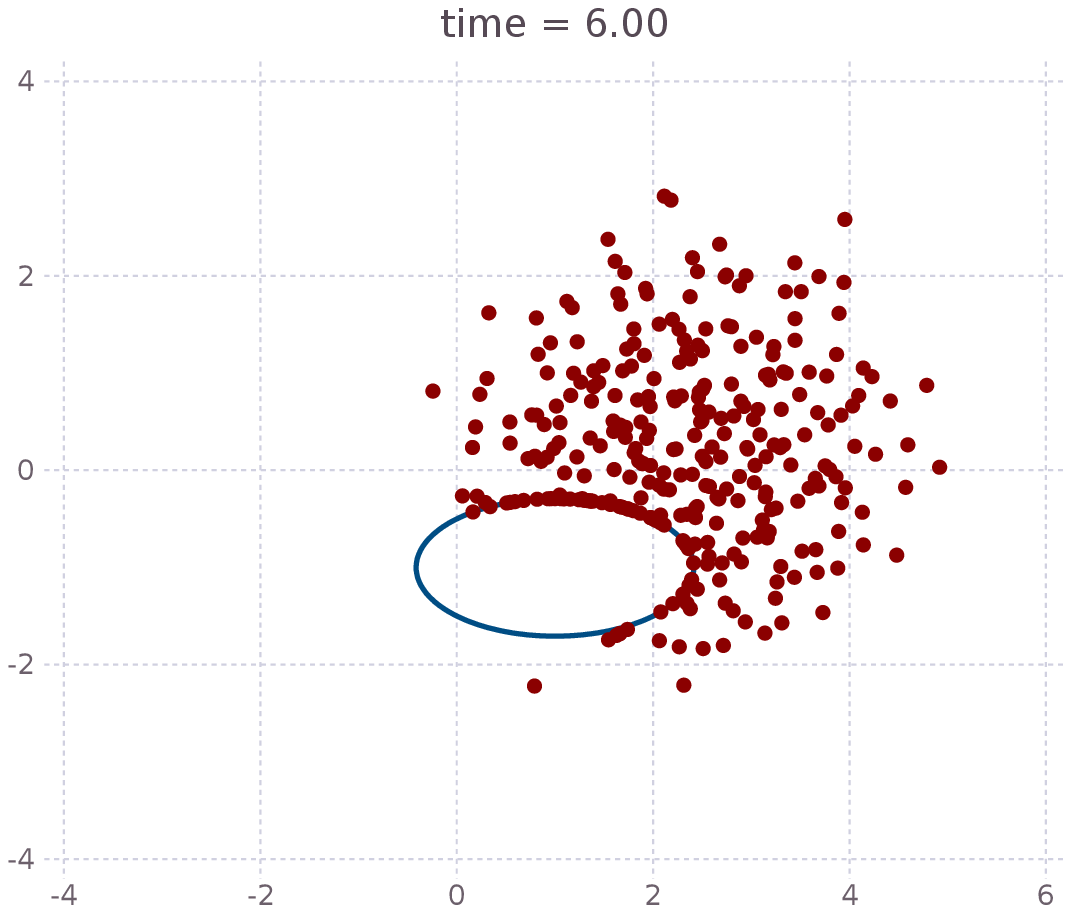}}
  {\includegraphics[width=0.19\textwidth]{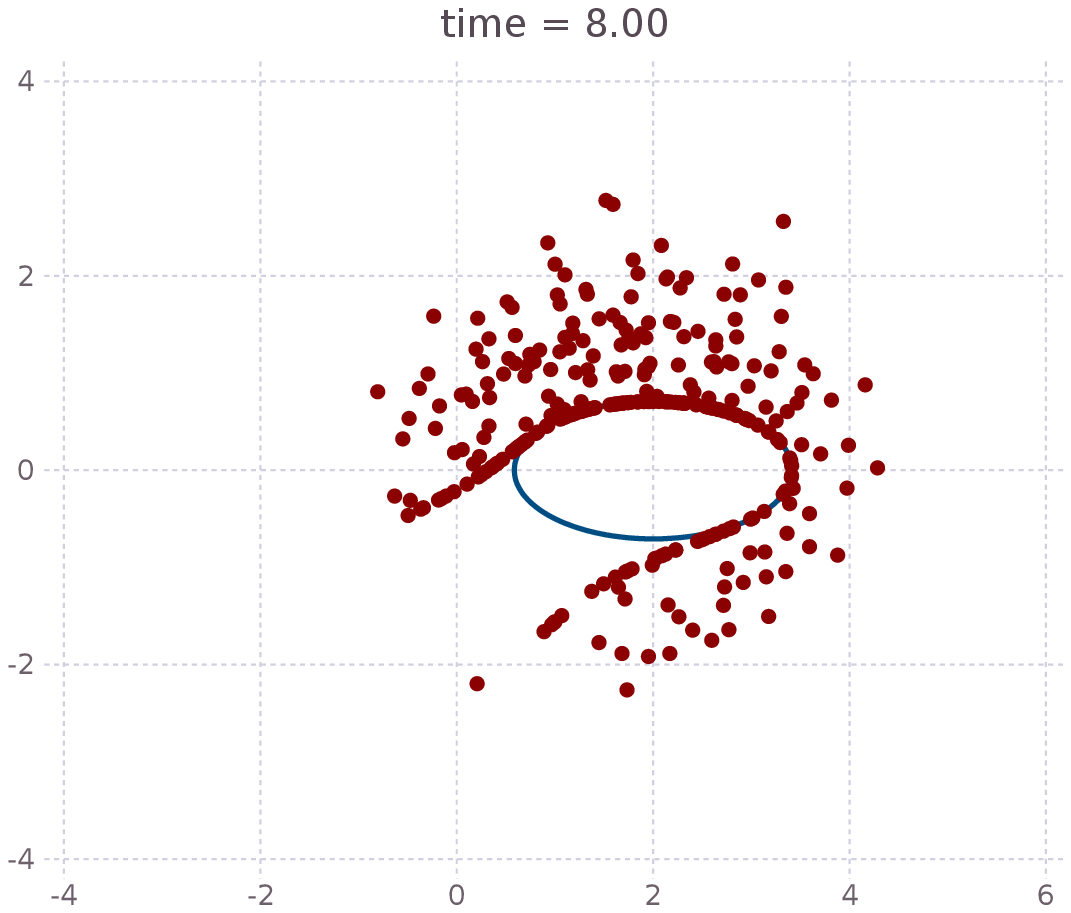}}\\
  {\includegraphics[width=0.19\textwidth]{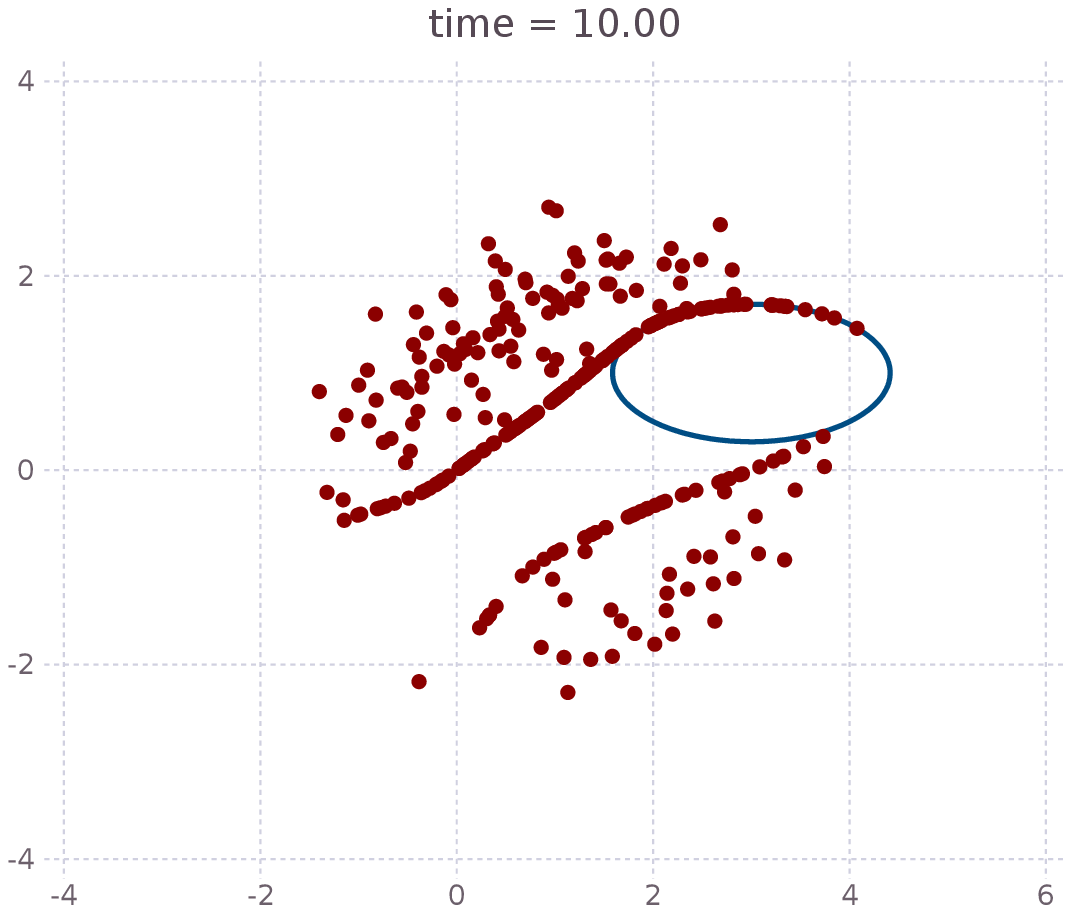}}
  {\includegraphics[width=0.19\textwidth]{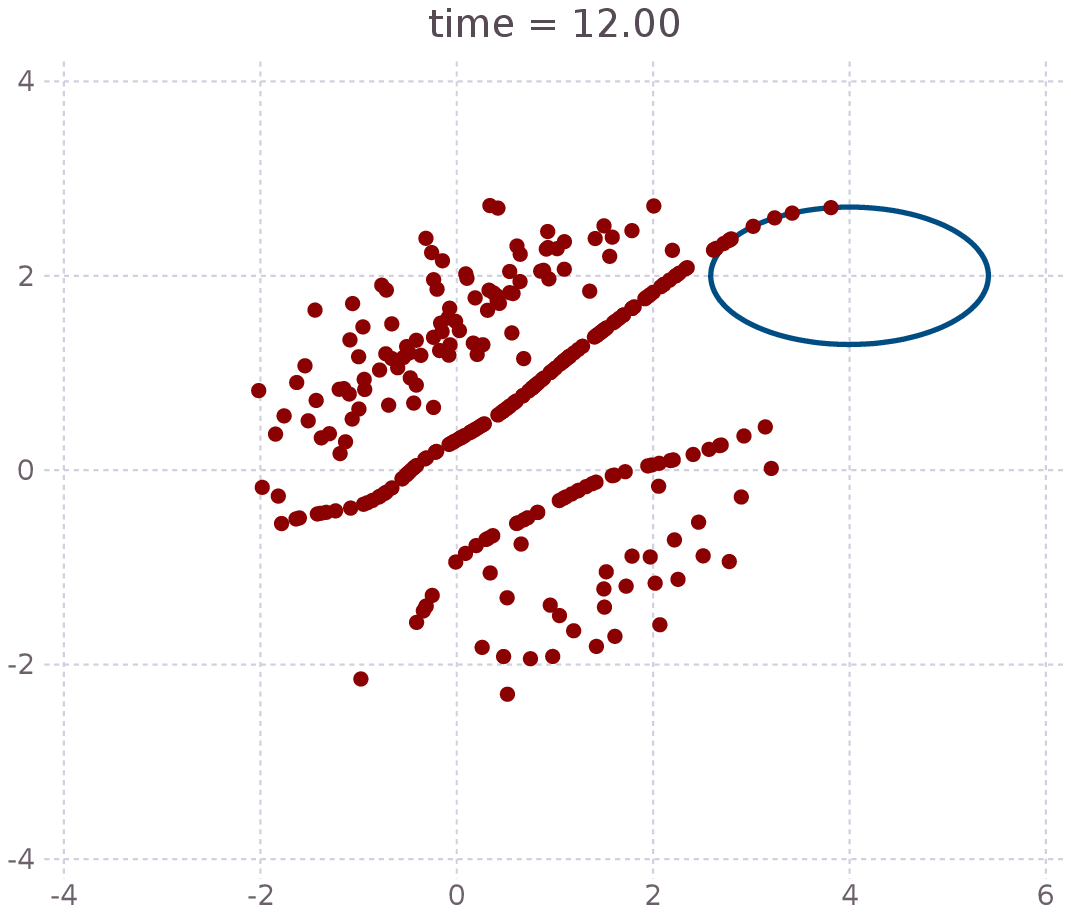}}
  {\includegraphics[width=0.19\textwidth]{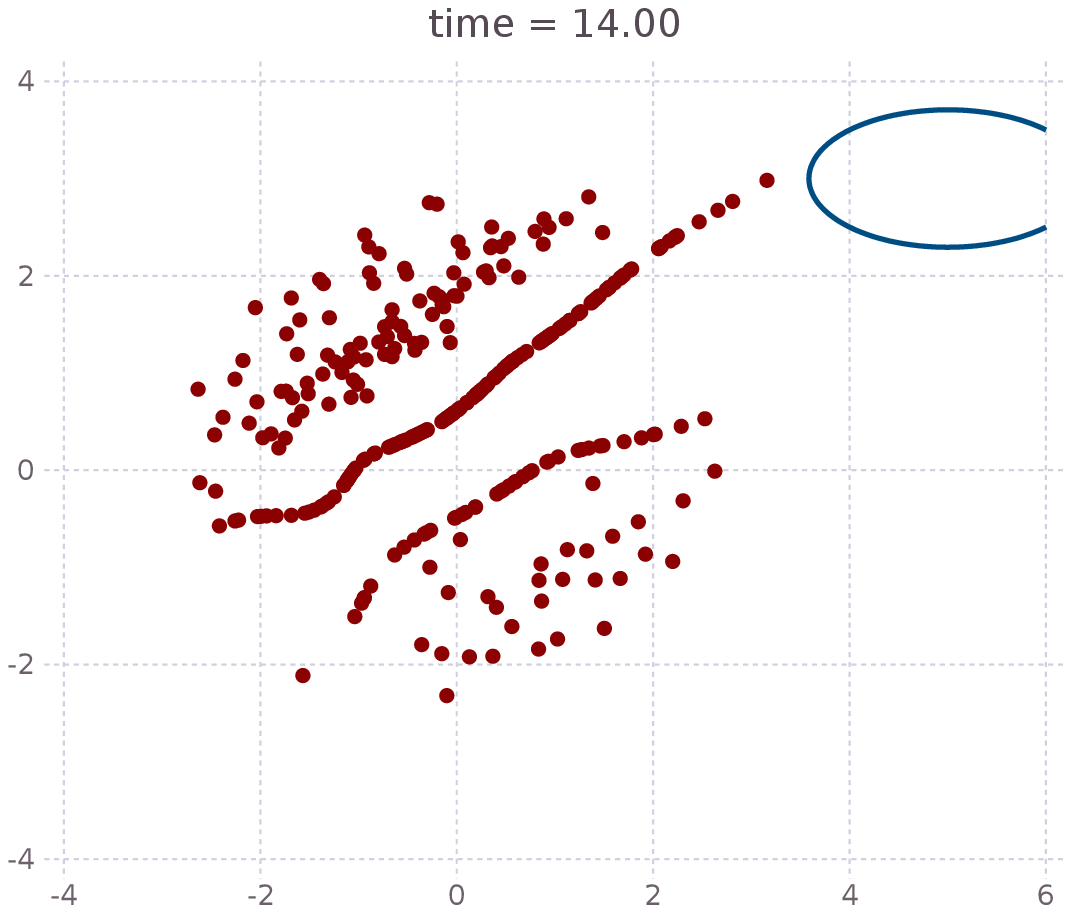}}
  {\includegraphics[width=0.19\textwidth]{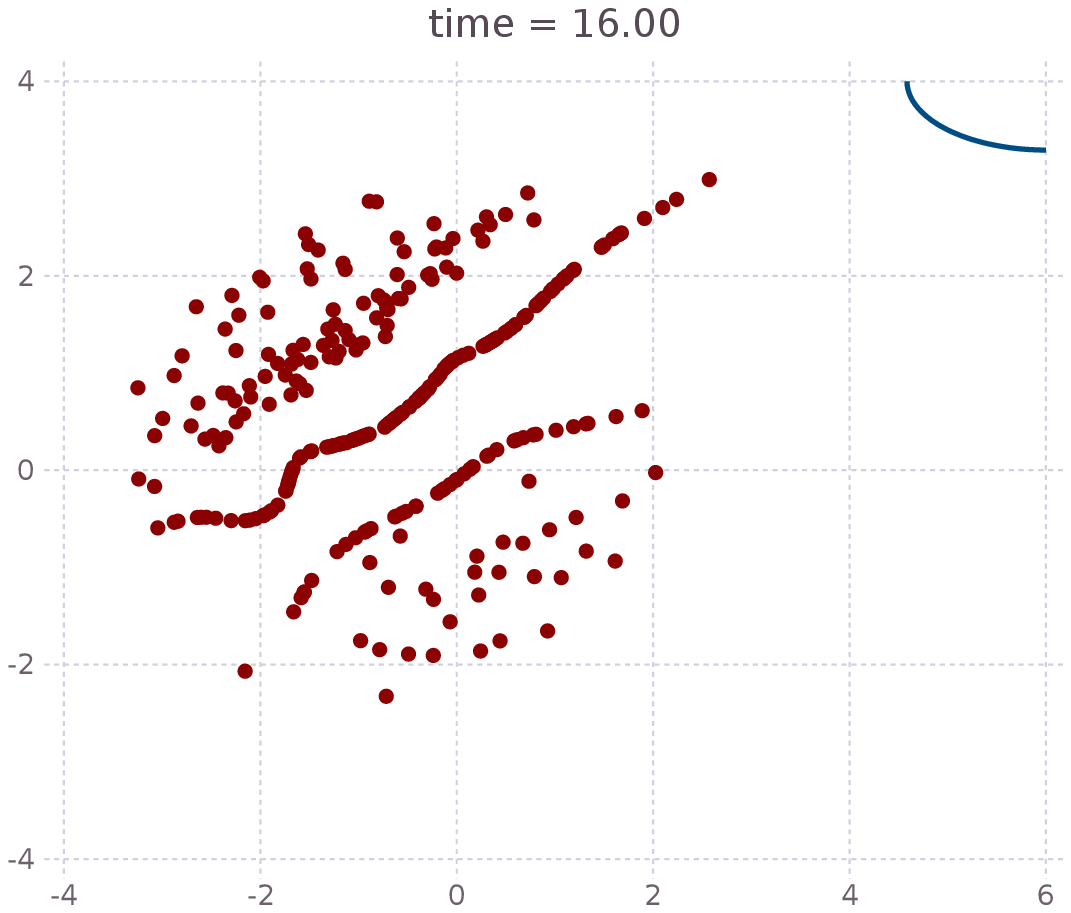}}
  {\includegraphics[width=0.19\textwidth]{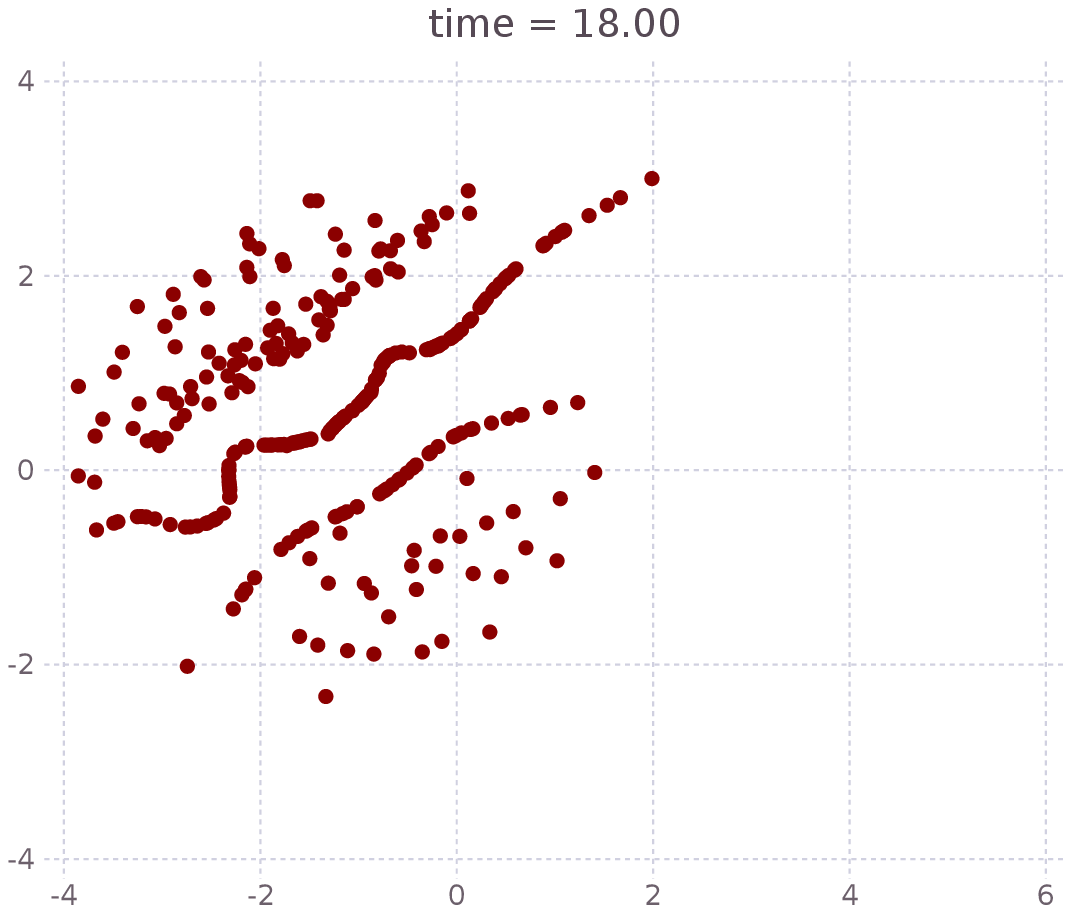}}
  \caption{The attraction/repulsion model: solutions at time moments \( t = 0,2,4,\ldots,18 \).}
  \label{fig:mathbiol}
\end{figure}
\begin{figure}
  \centering
  \begin{subfigure}{\textwidth}
  \includegraphics[width=0.19\textwidth]{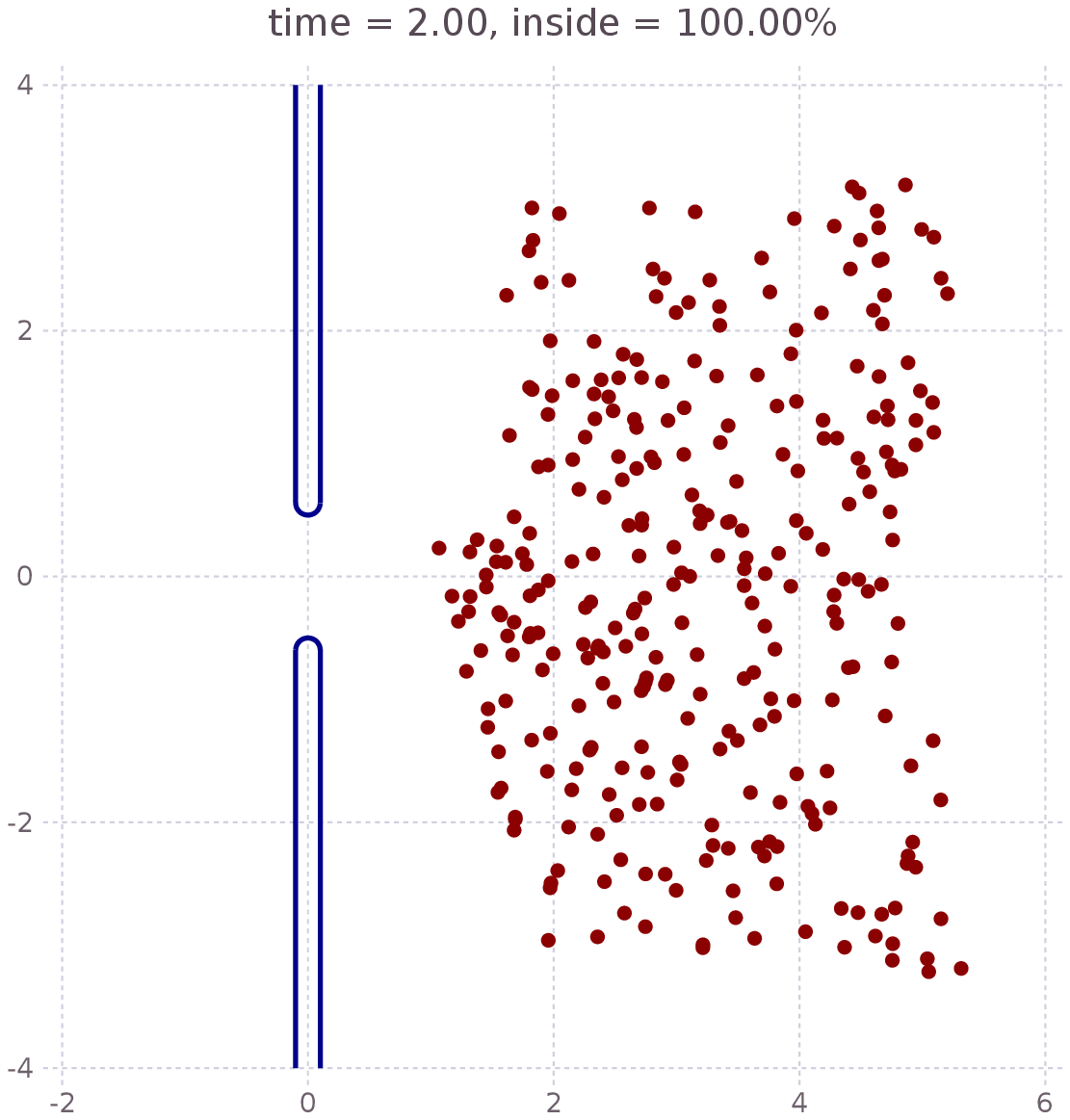}
  \includegraphics[width=0.19\textwidth]{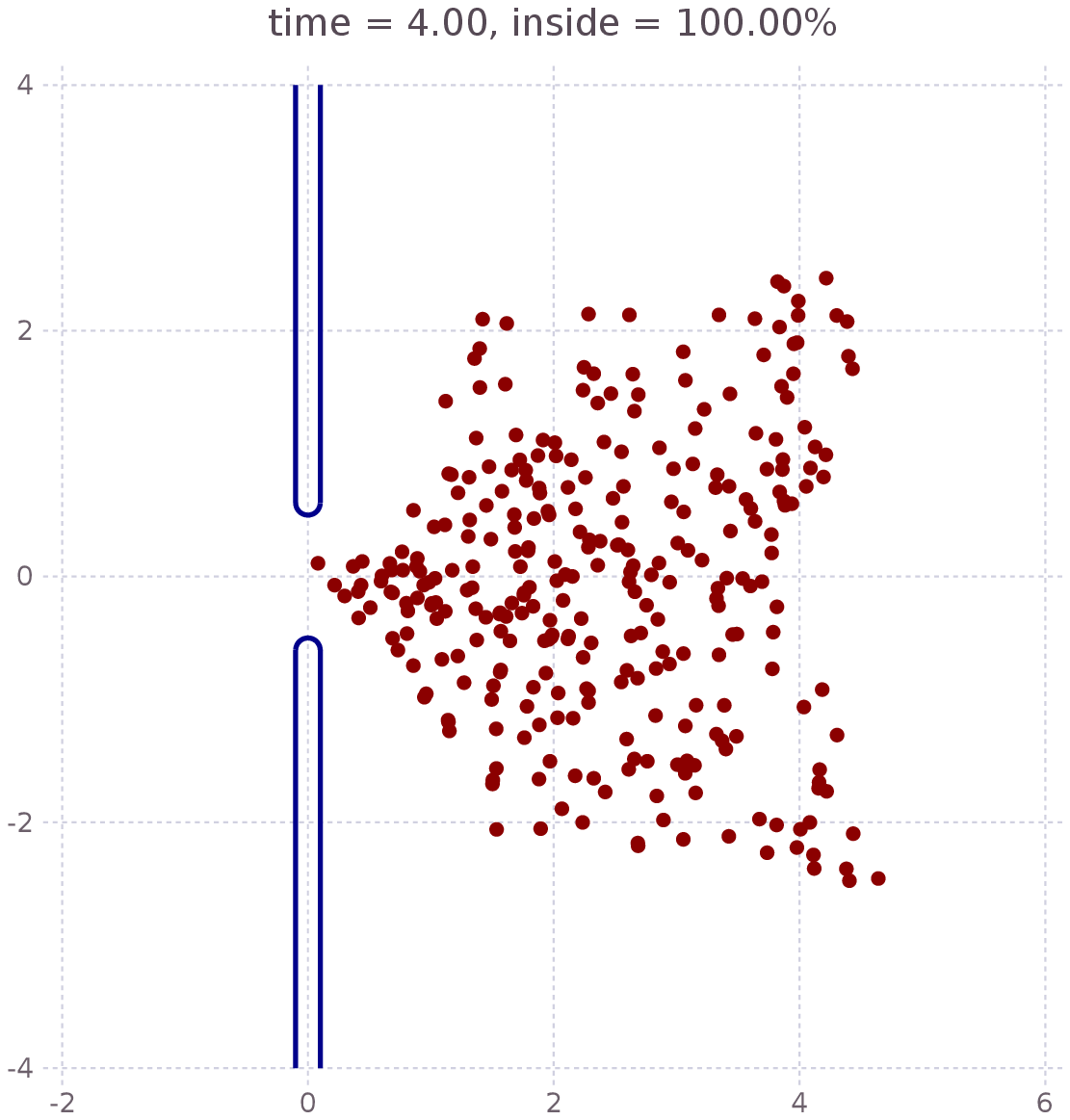}
  \includegraphics[width=0.19\textwidth]{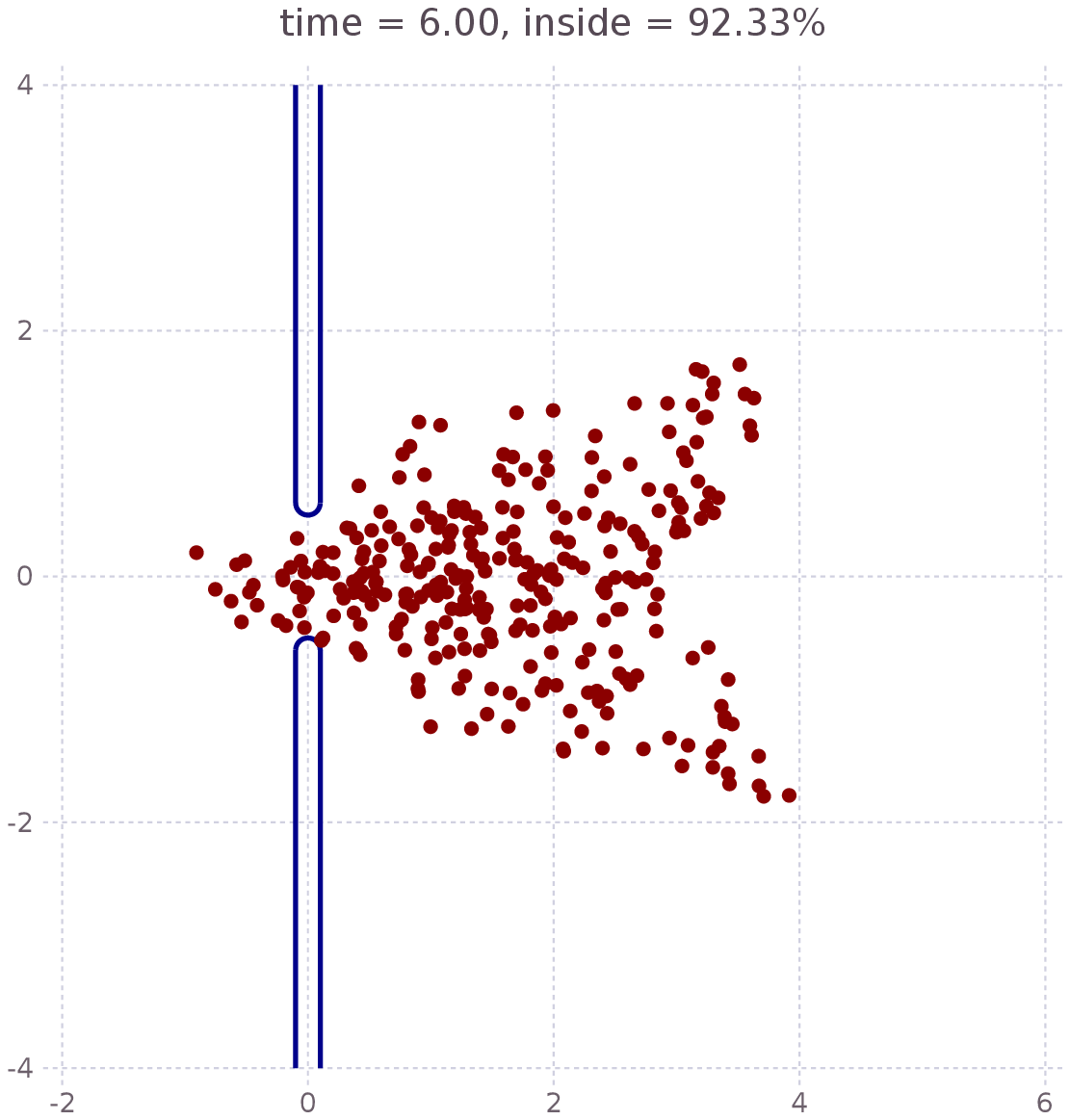}
  \includegraphics[width=0.19\textwidth]{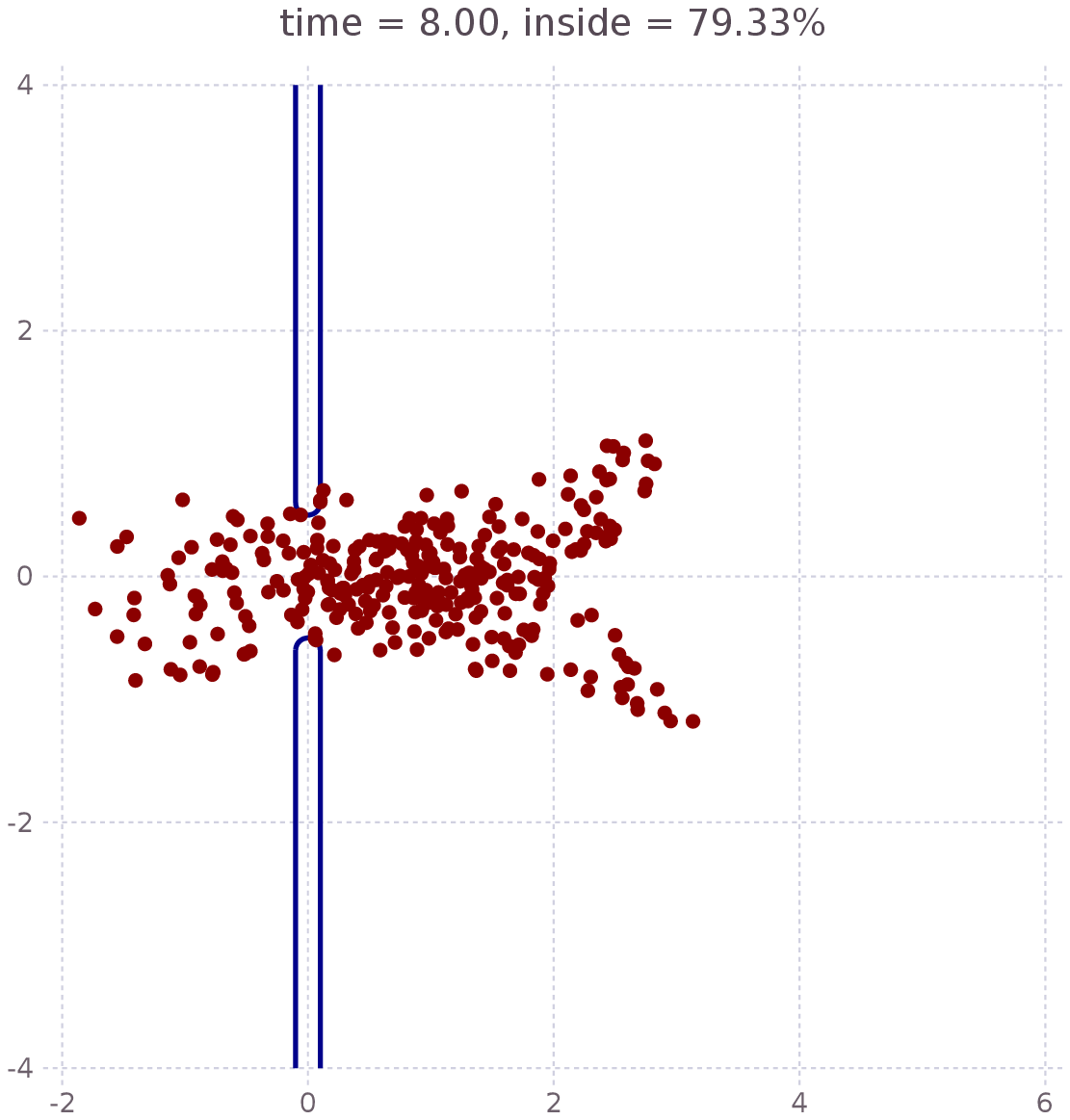}
  \includegraphics[width=0.19\textwidth]{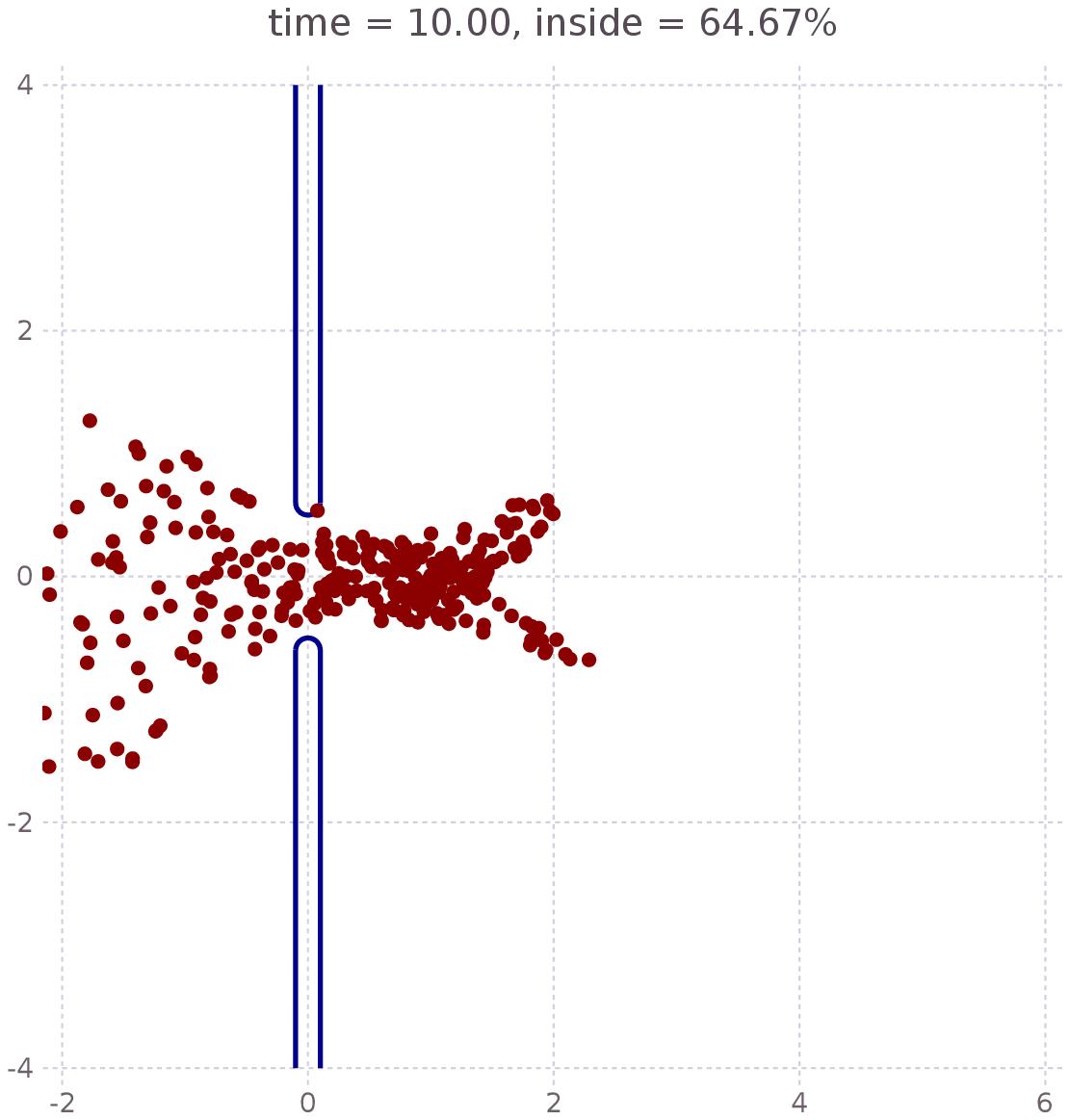}\\
  \includegraphics[width=0.19\textwidth]{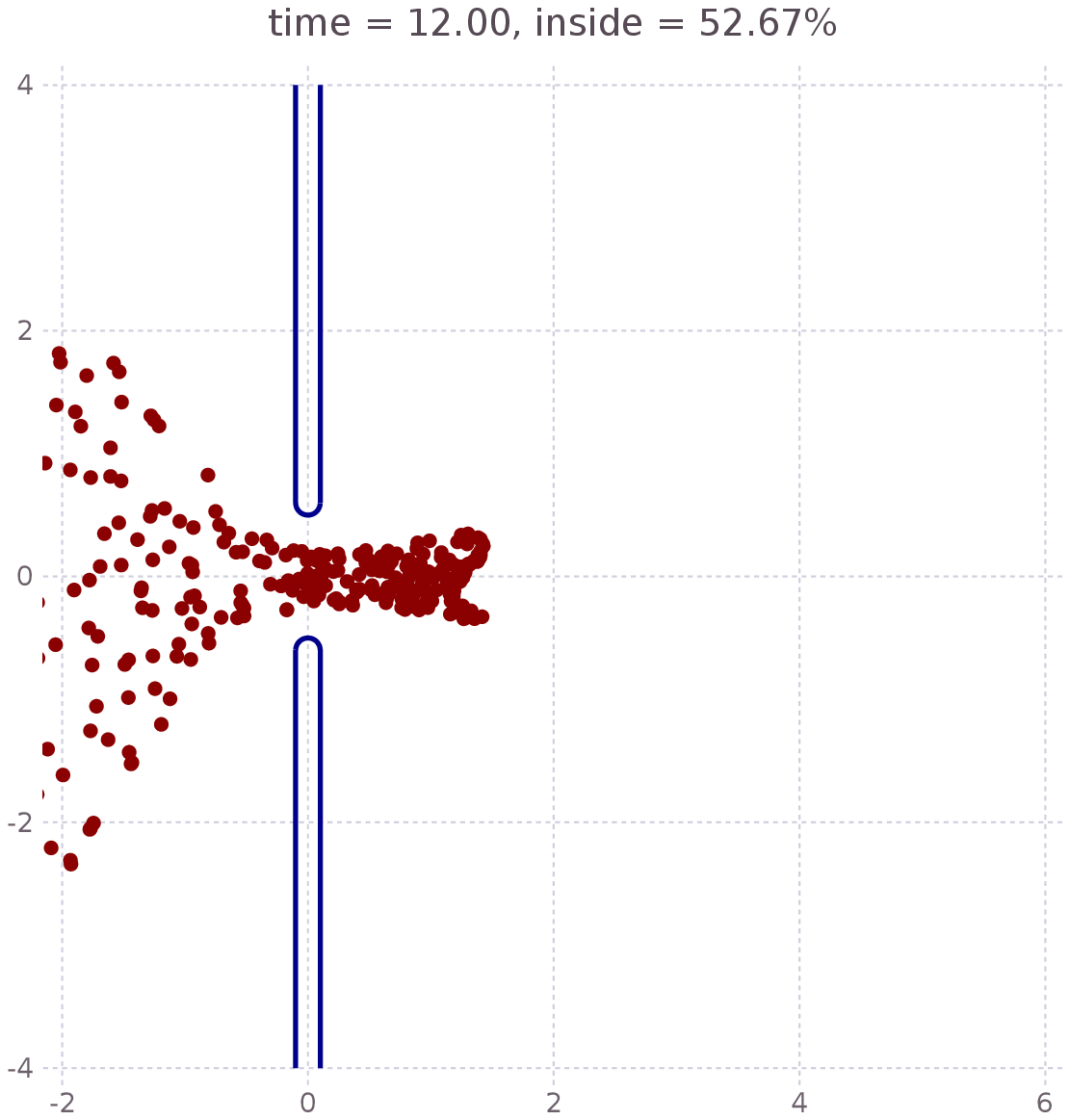}
  \includegraphics[width=0.19\textwidth]{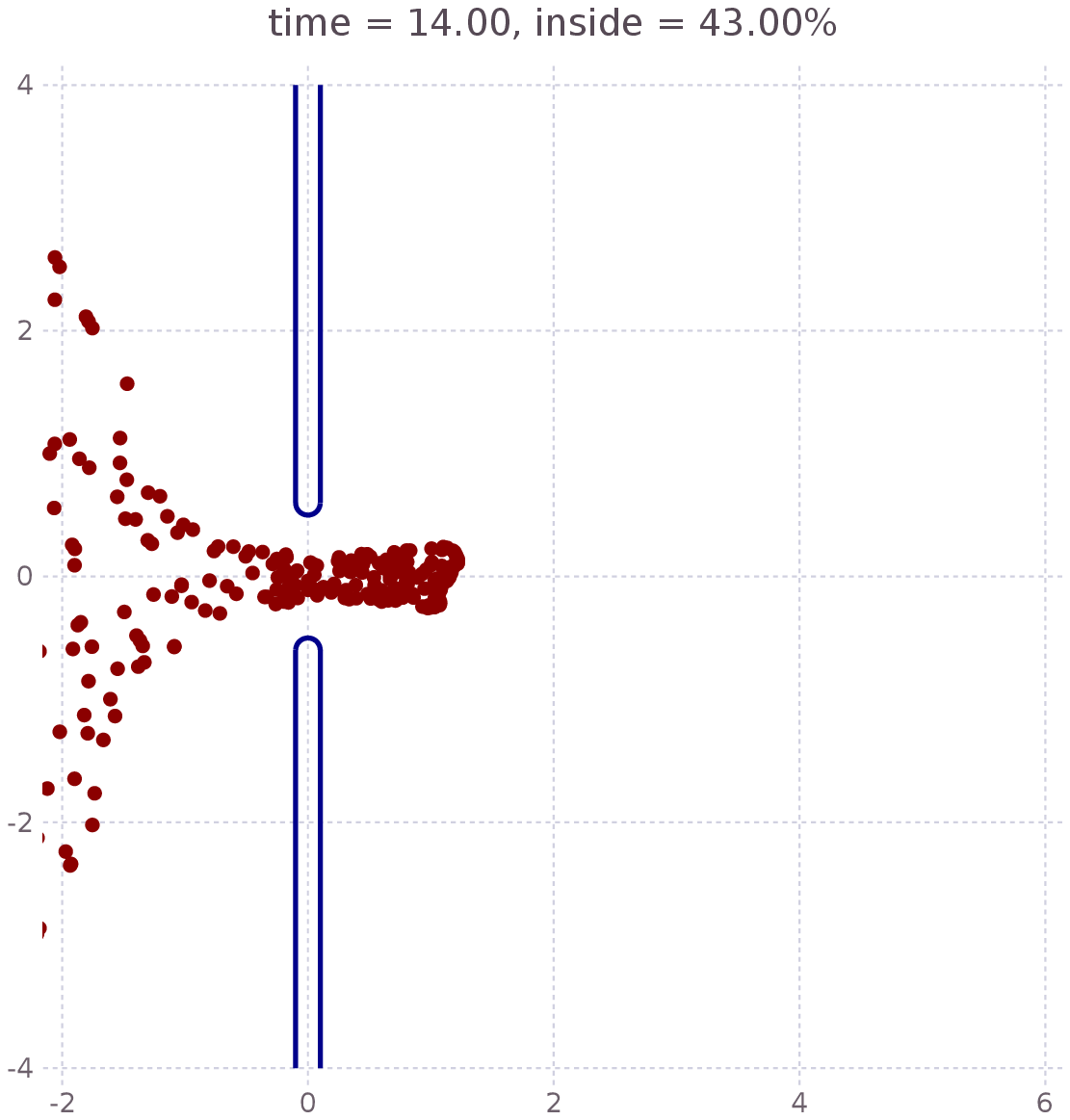}
  \includegraphics[width=0.19\textwidth]{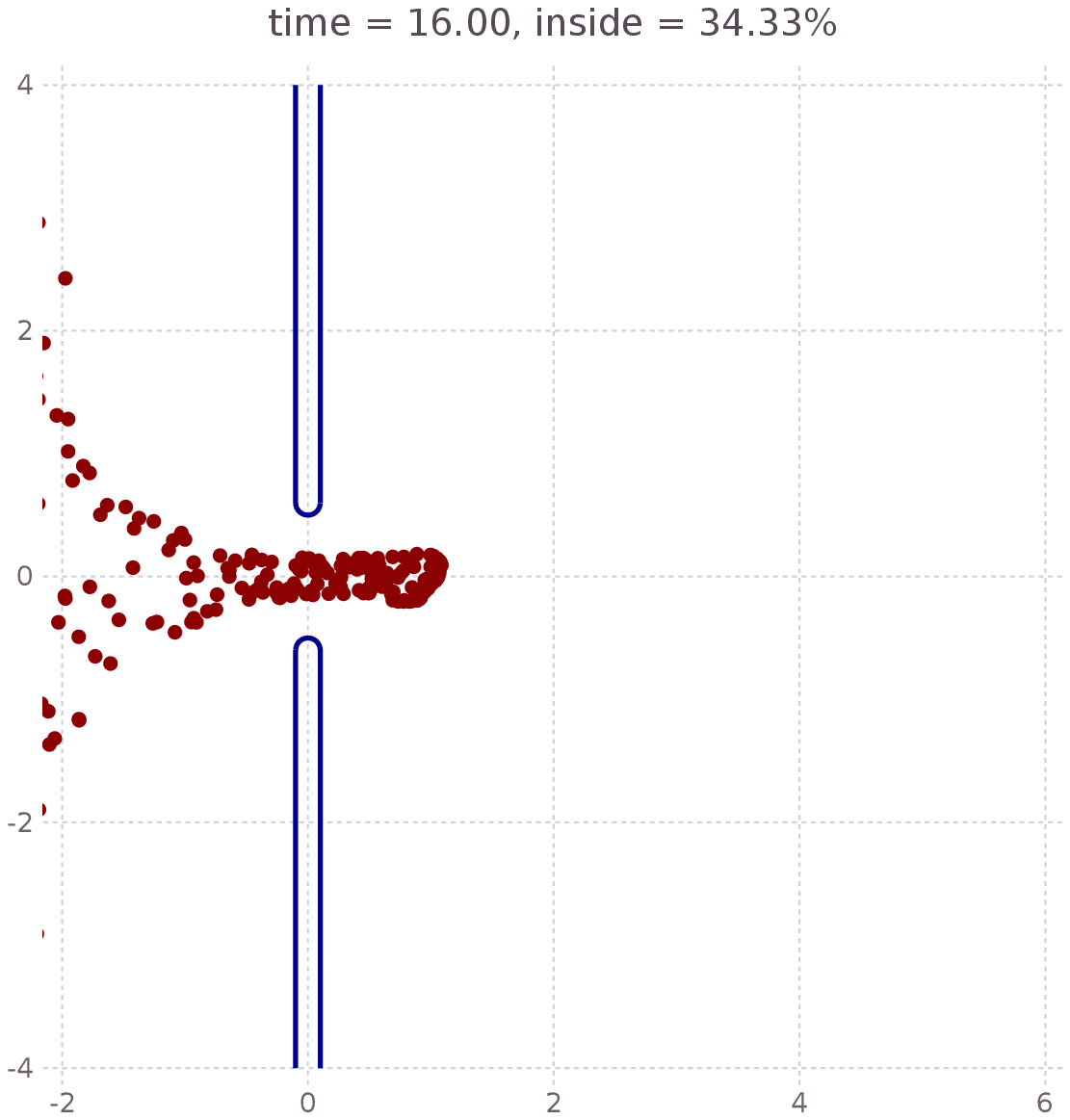}
  \includegraphics[width=0.19\textwidth]{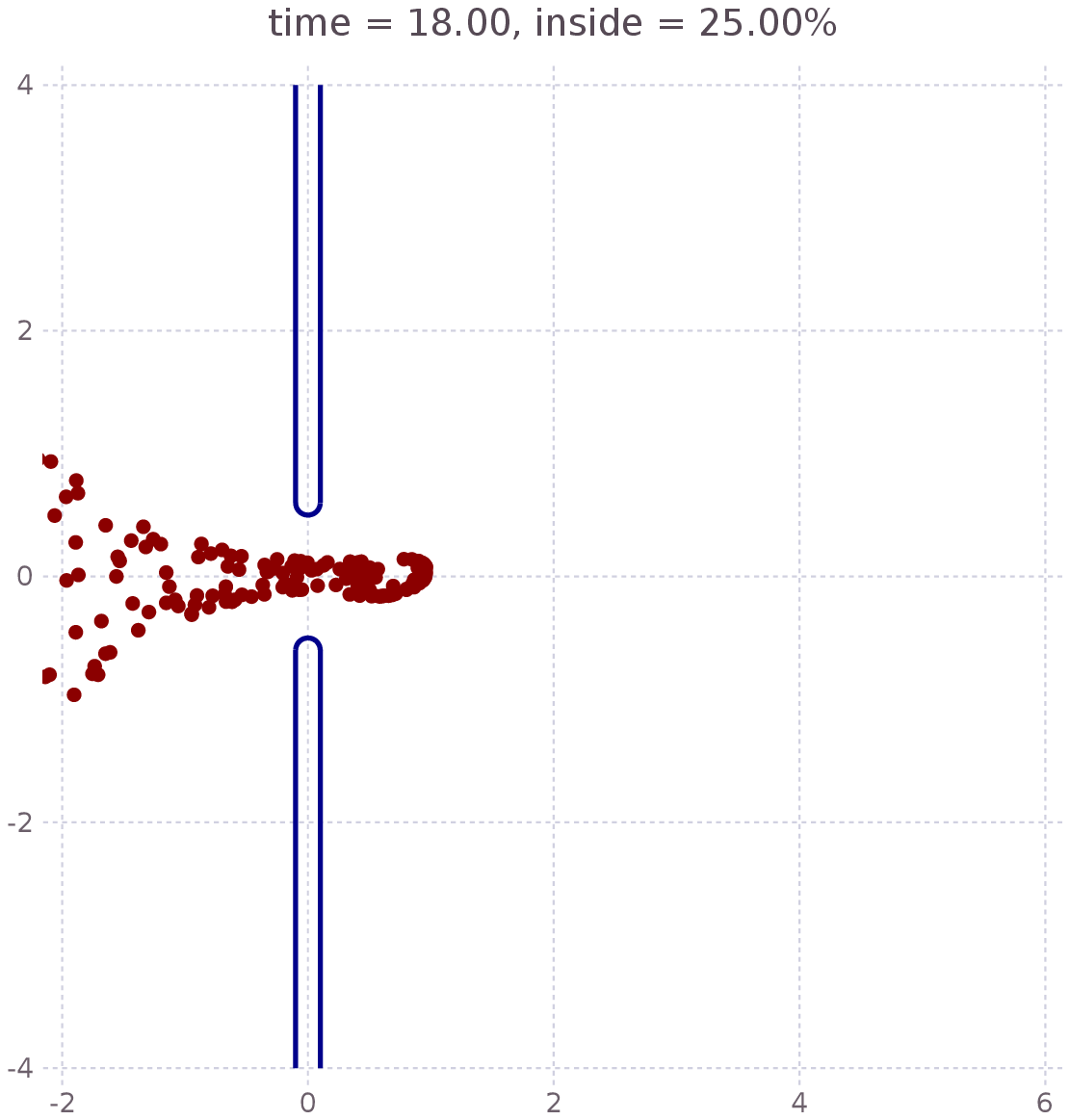}
  \includegraphics[width=0.19\textwidth]{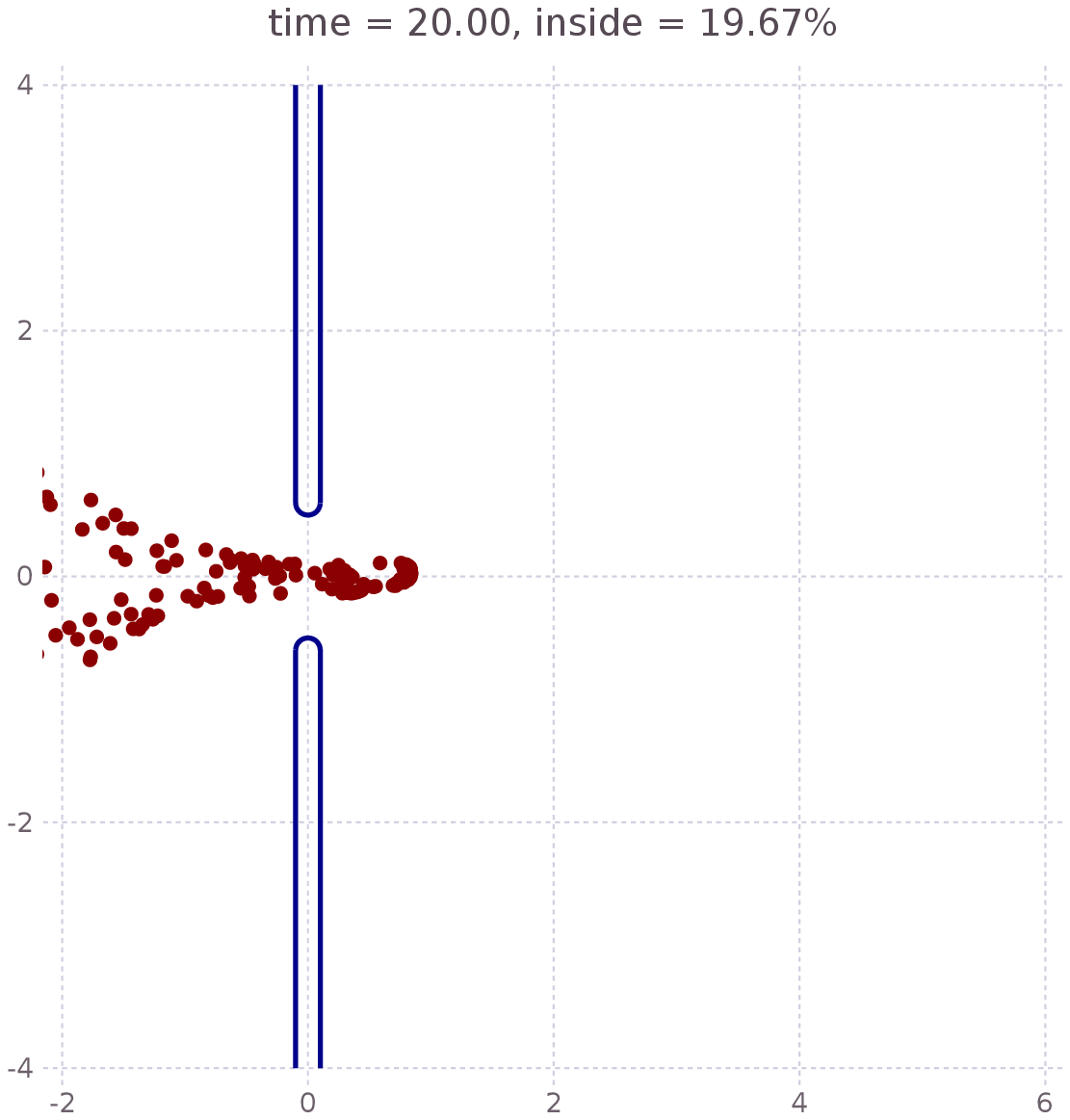}
  \caption{No obstacle: \( c = (100,100) \), \( a = (0.9,0.16) \), \( \omega = 0 \).}
\end{subfigure}\\
  \begin{subfigure}{\textwidth}
  \includegraphics[width=0.19\textwidth]{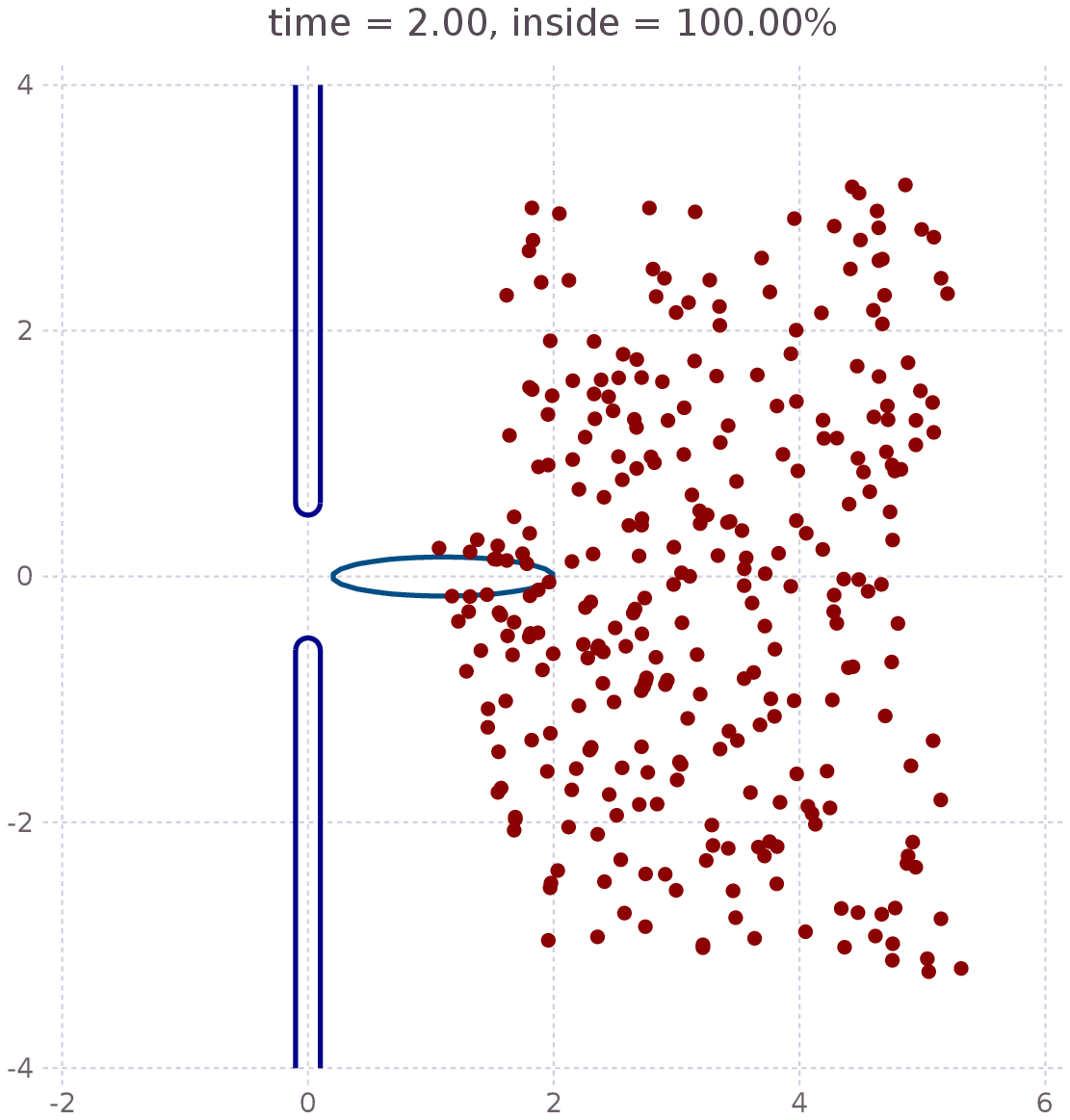}
  \includegraphics[width=0.19\textwidth]{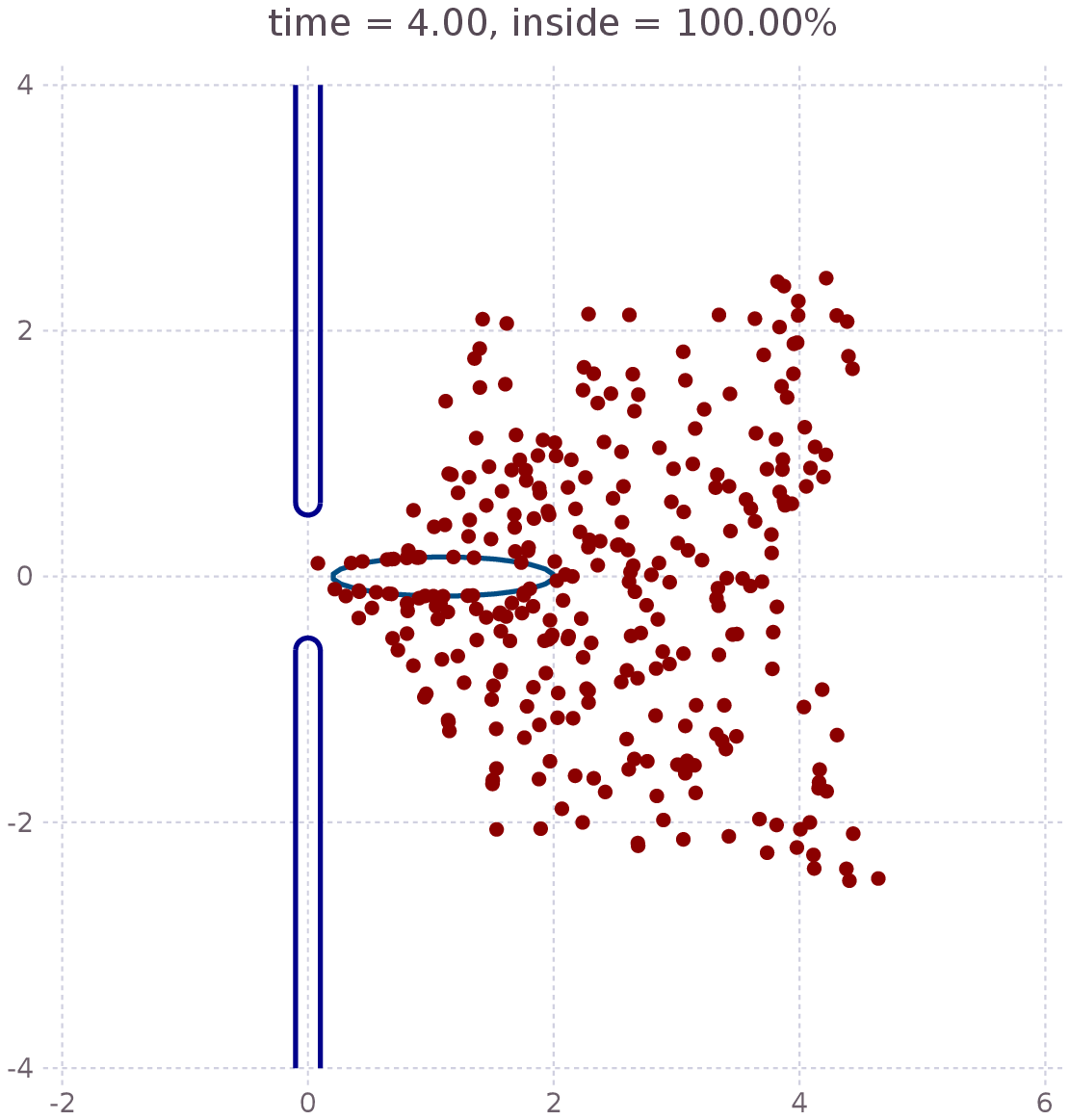}
  \includegraphics[width=0.19\textwidth]{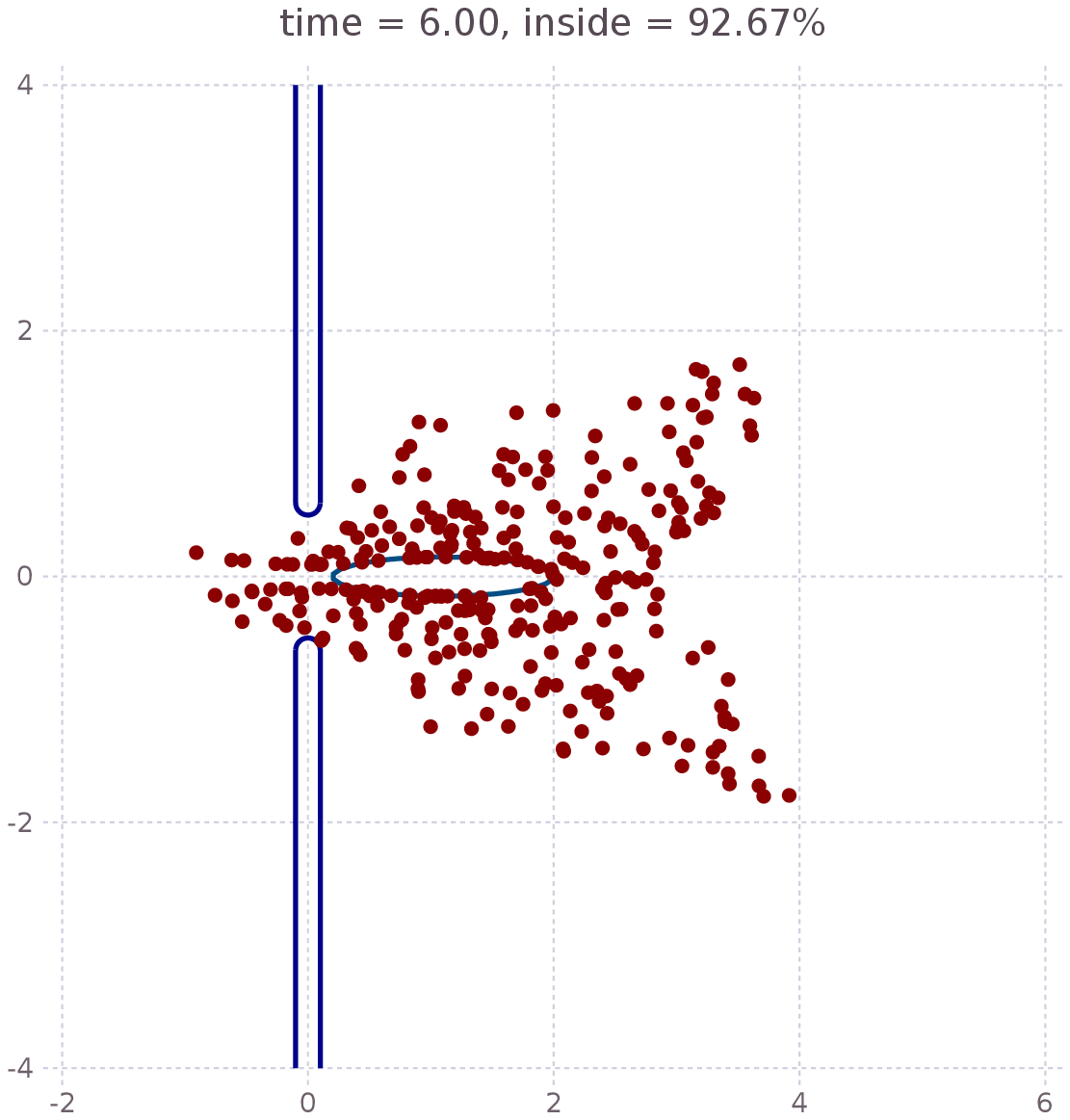}
  \includegraphics[width=0.19\textwidth]{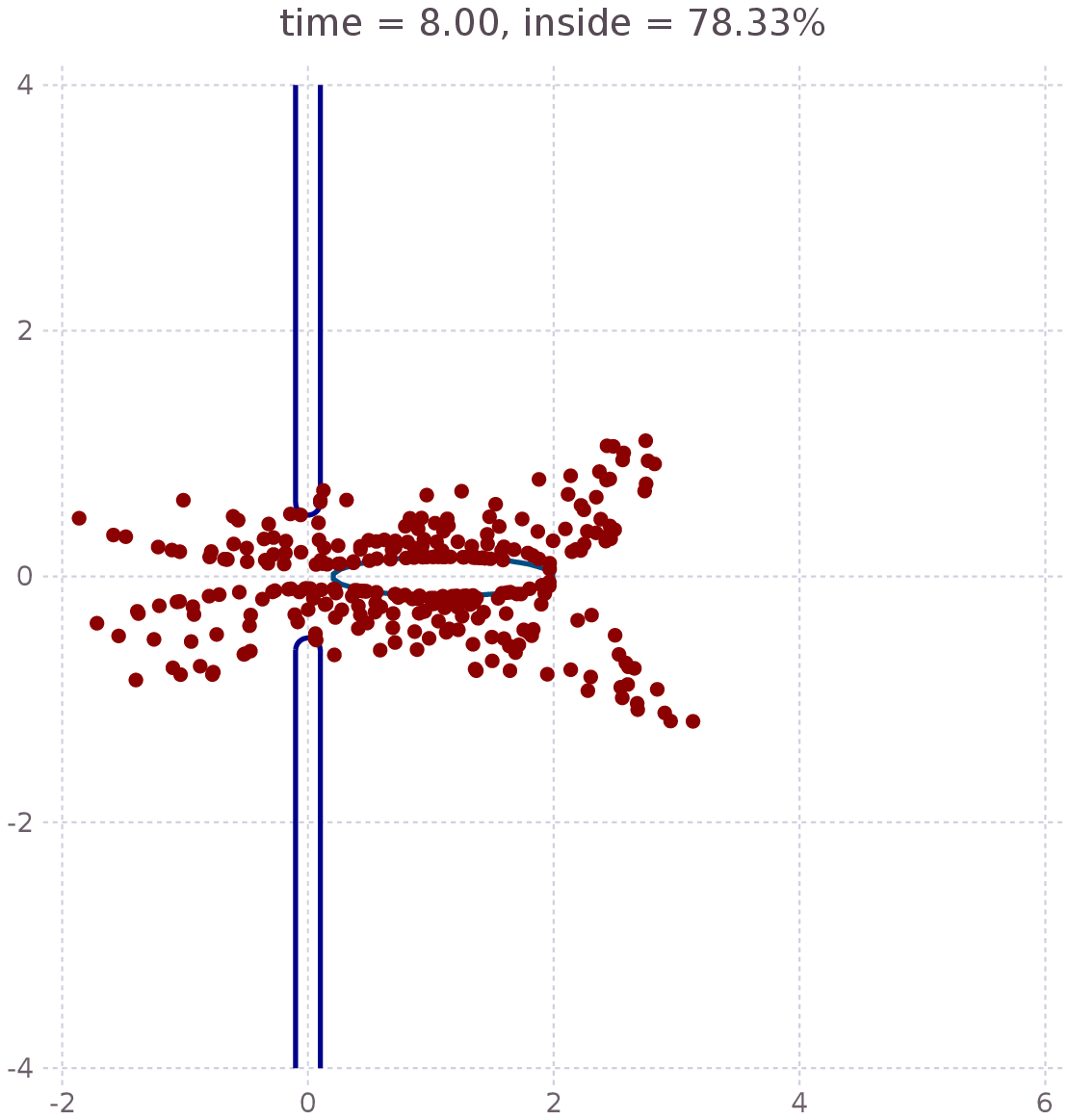}
  \includegraphics[width=0.19\textwidth]{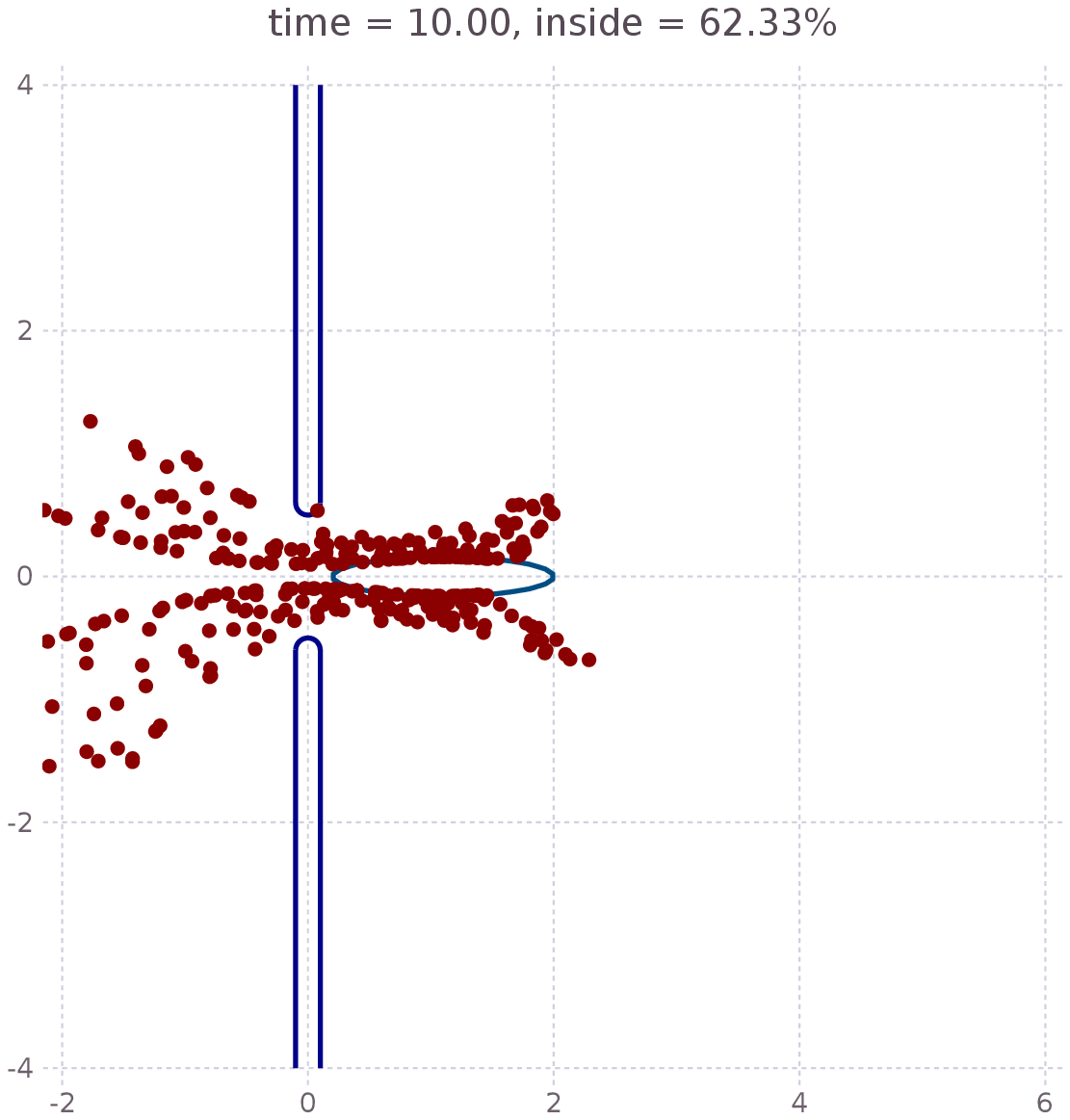}\\
  \includegraphics[width=0.19\textwidth]{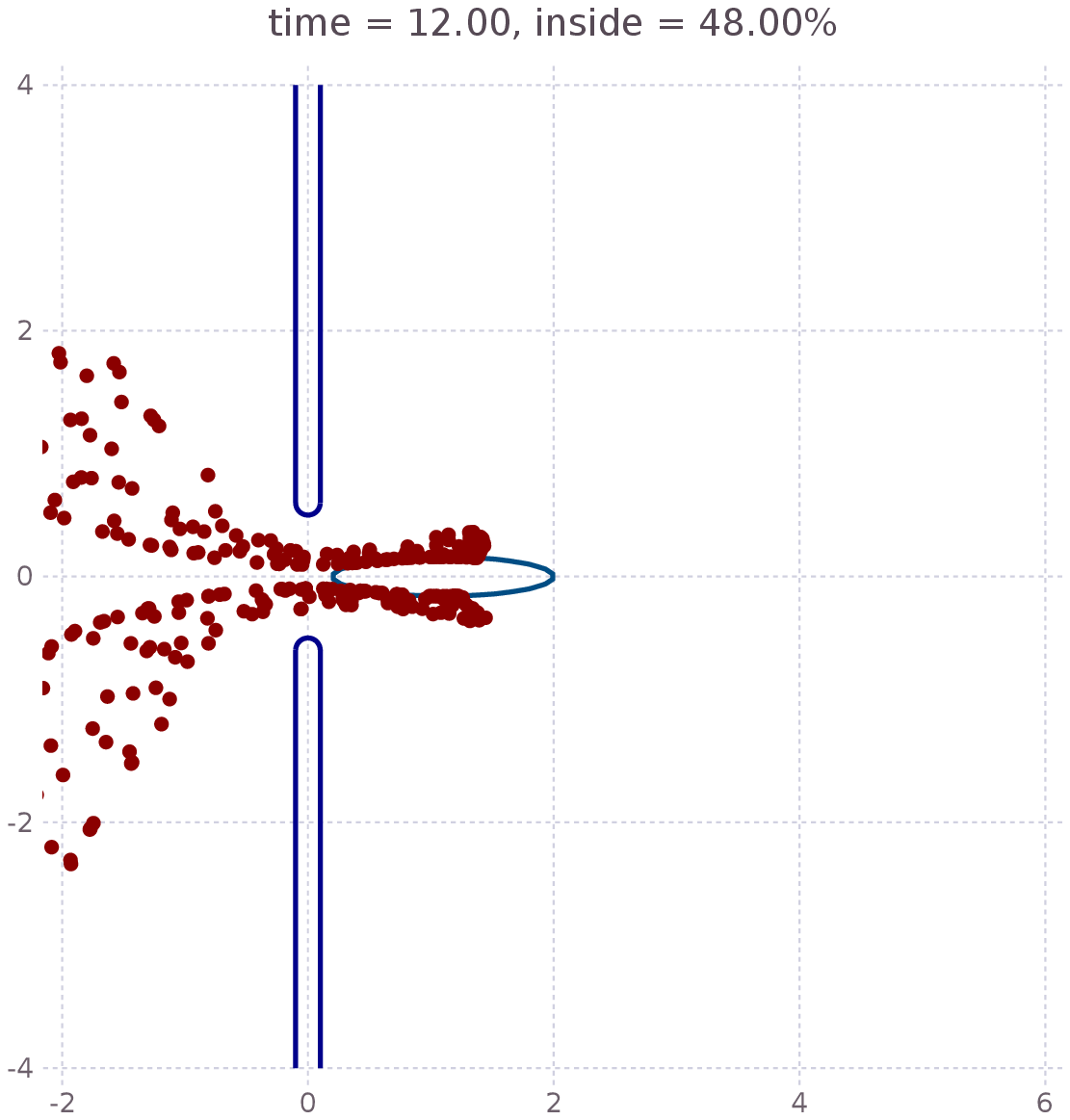}
  \includegraphics[width=0.19\textwidth]{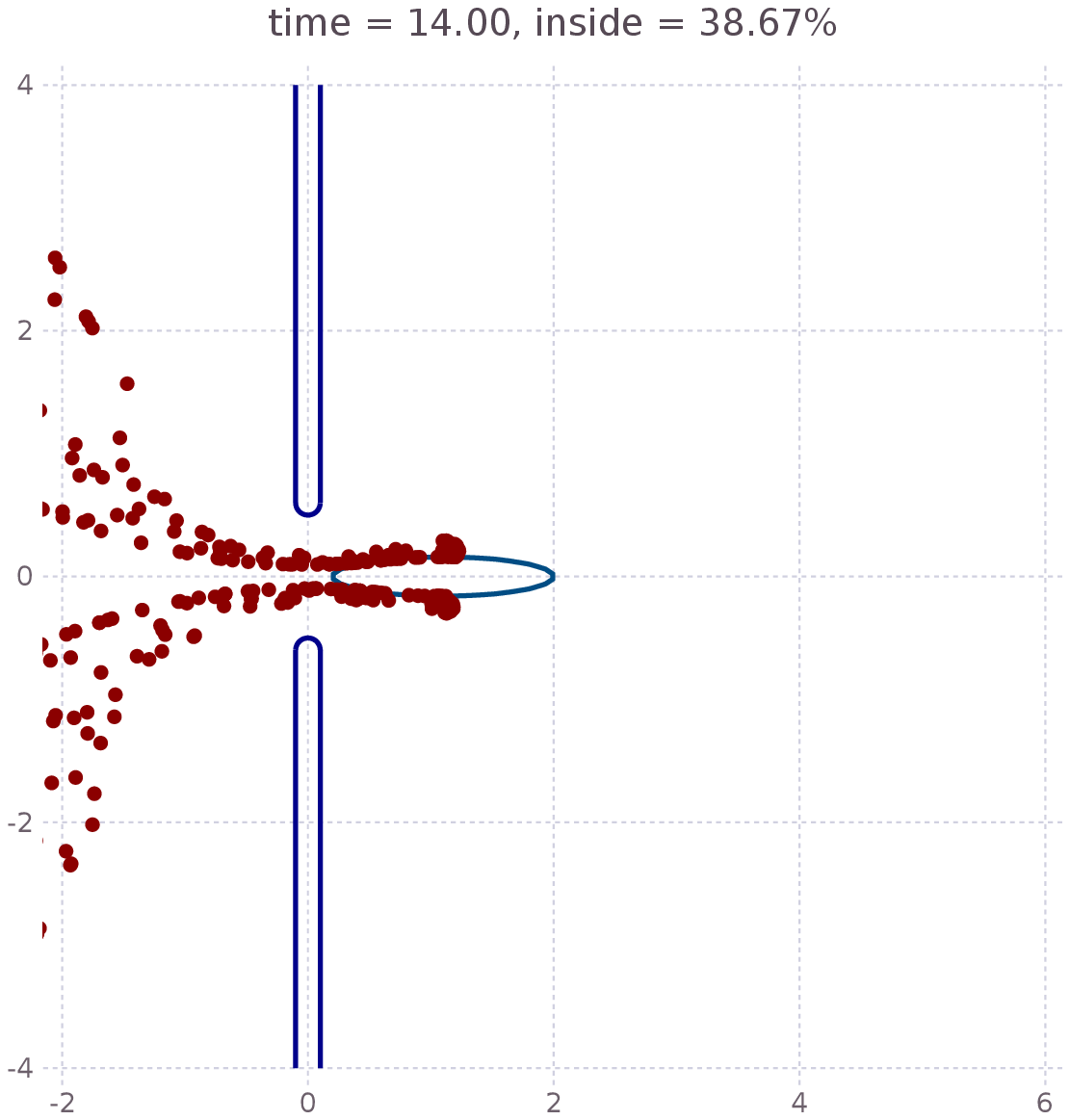}
  \includegraphics[width=0.19\textwidth]{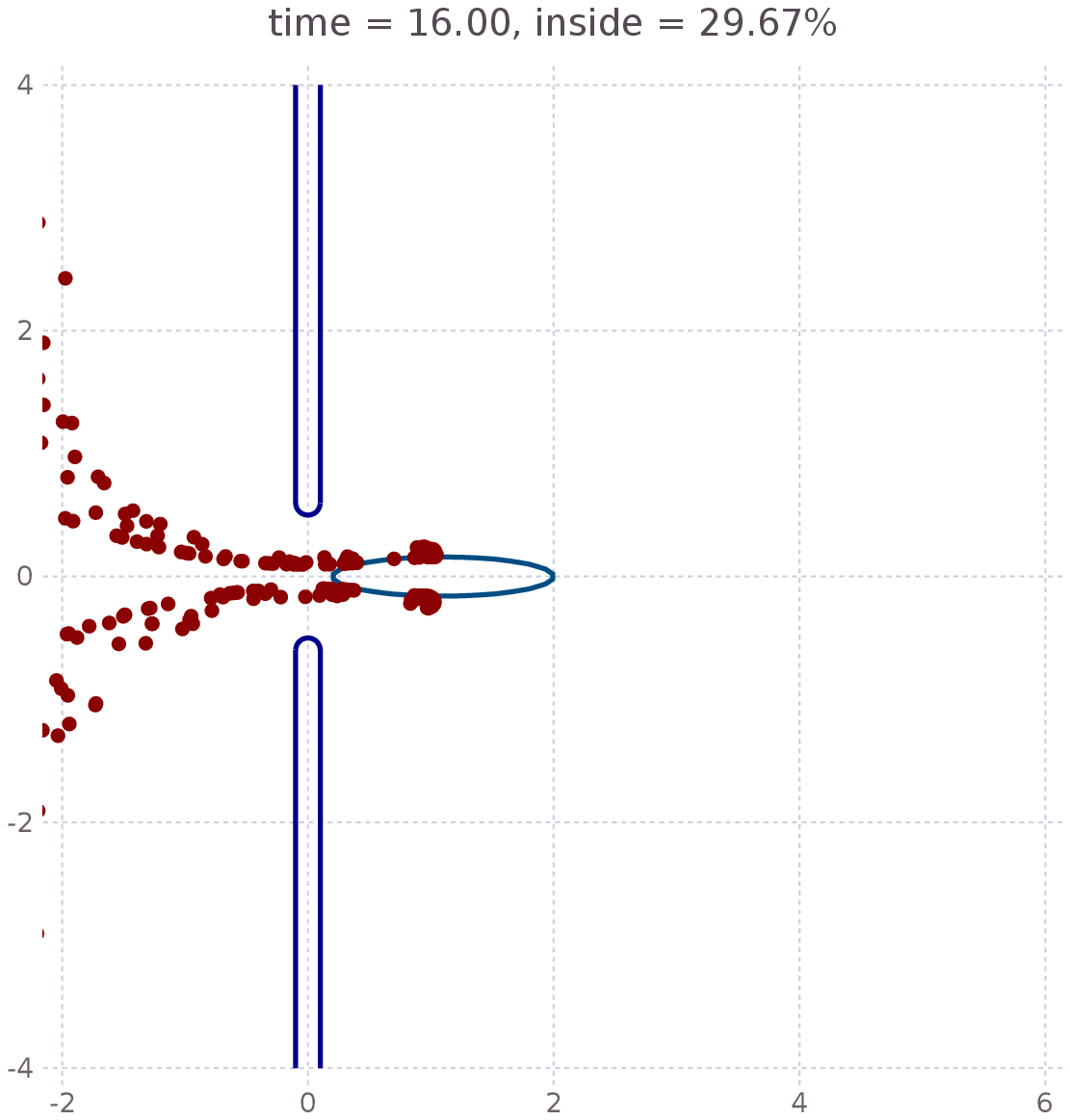}
  \includegraphics[width=0.19\textwidth]{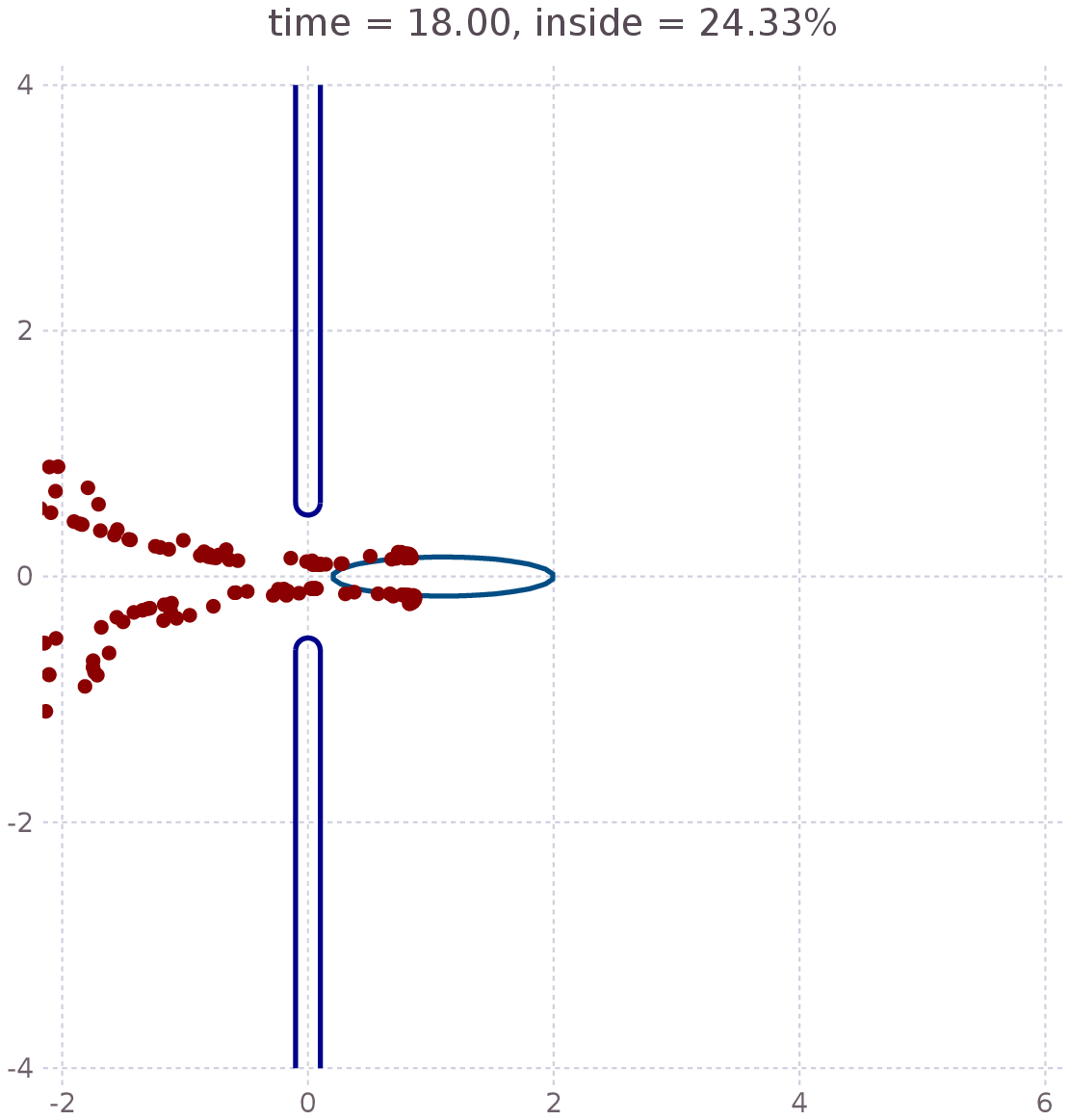}
  \includegraphics[width=0.19\textwidth]{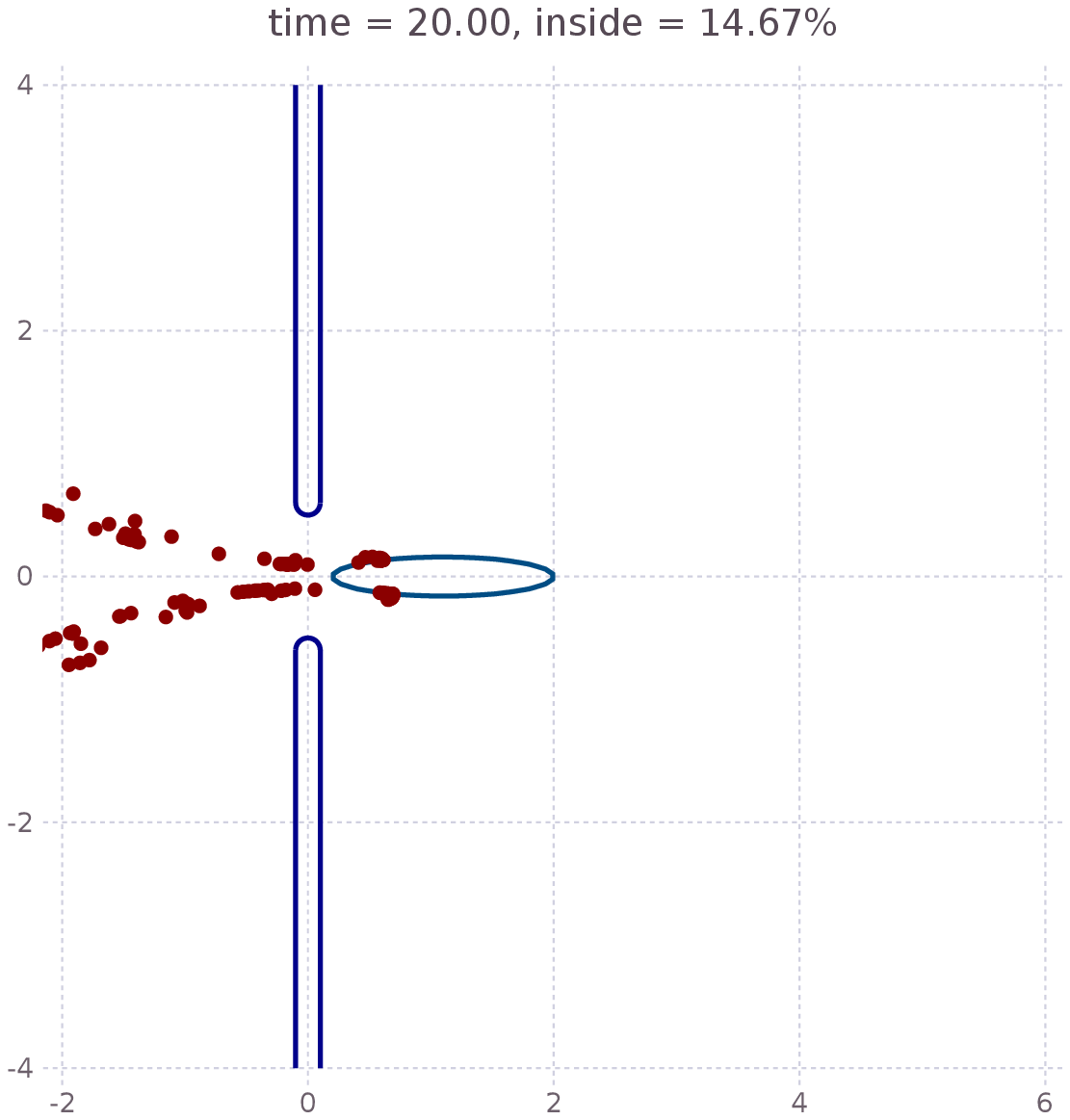}
  \caption{Stationary obstacle: \( c = (1.1,0) \), \( a = (0.9,0.16) \), \( \omega = 0 \).}
\end{subfigure}\\
\begin{subfigure}{\textwidth}
  \includegraphics[width=0.19\textwidth]{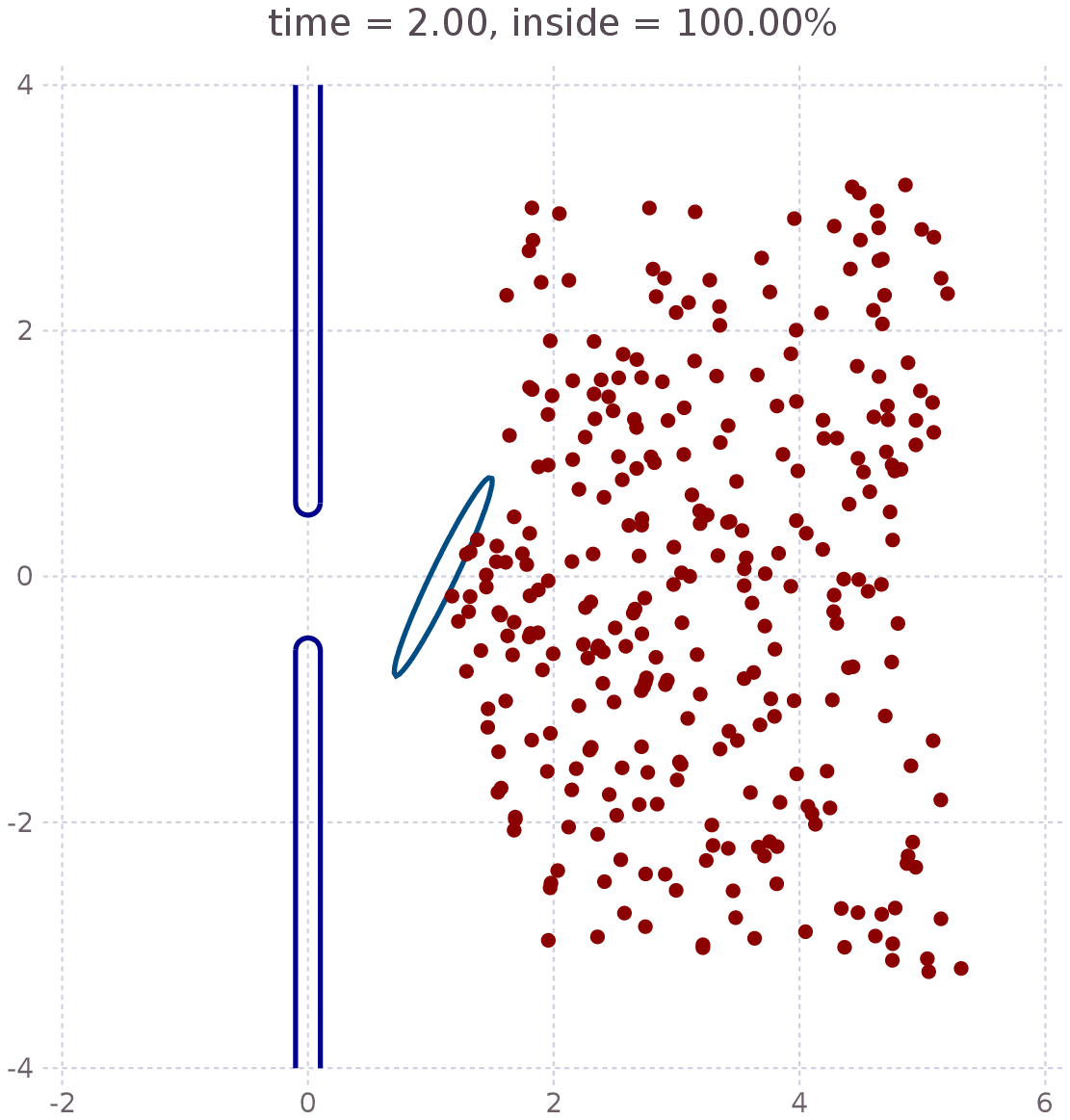}
  \includegraphics[width=0.19\textwidth]{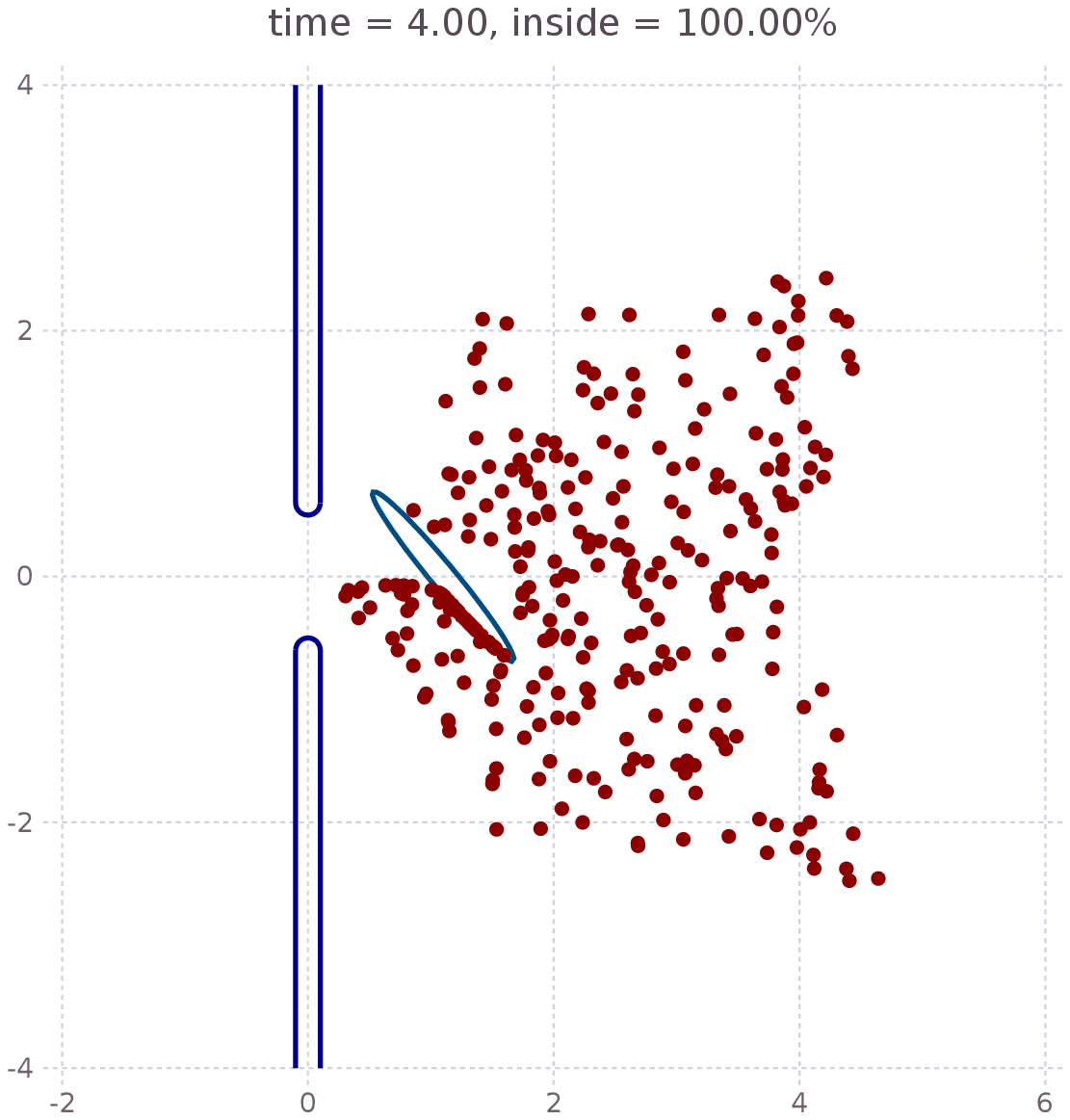}
  \includegraphics[width=0.19\textwidth]{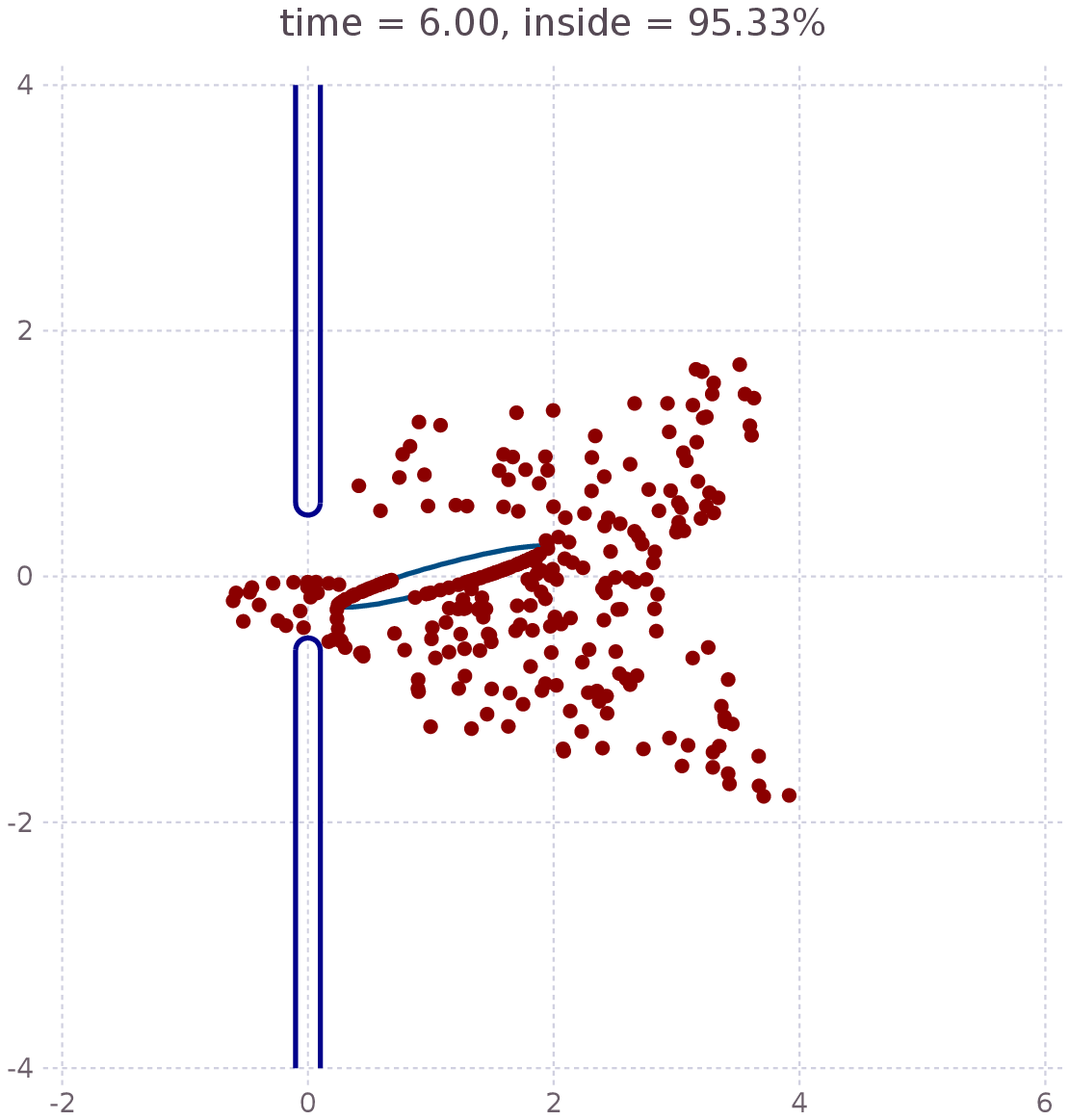}
  \includegraphics[width=0.19\textwidth]{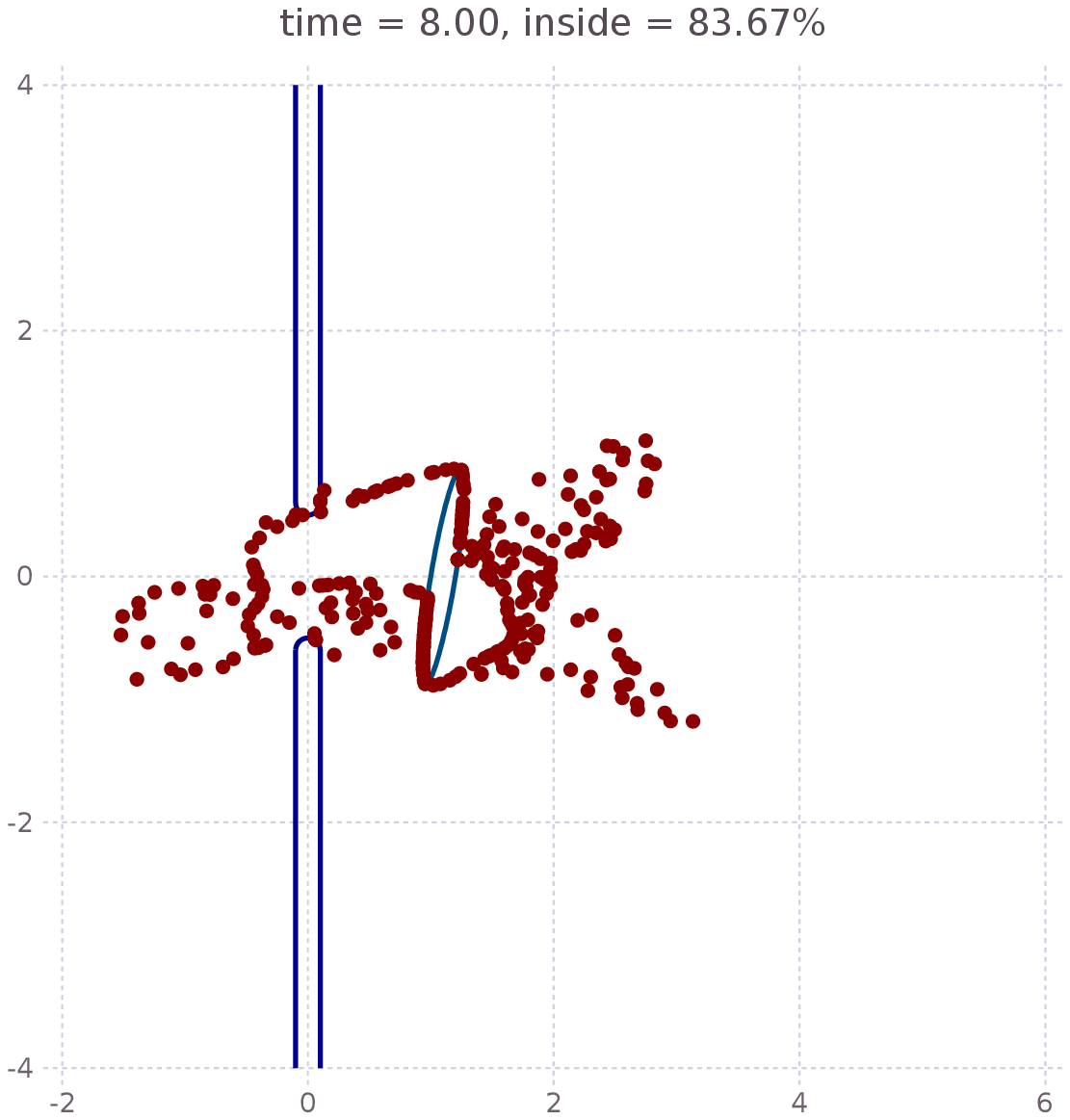}
  \includegraphics[width=0.19\textwidth]{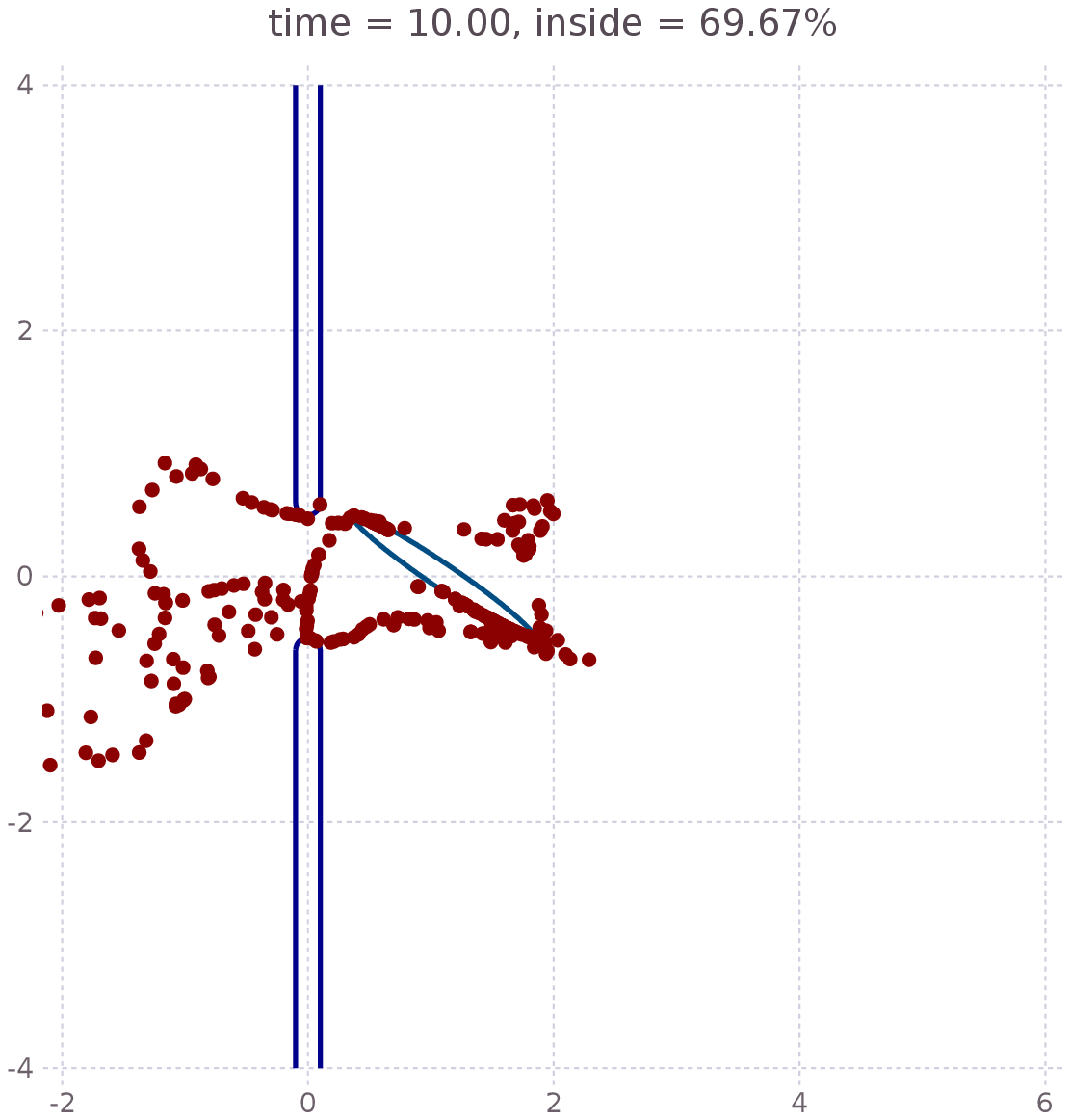}\\
  \includegraphics[width=0.19\textwidth]{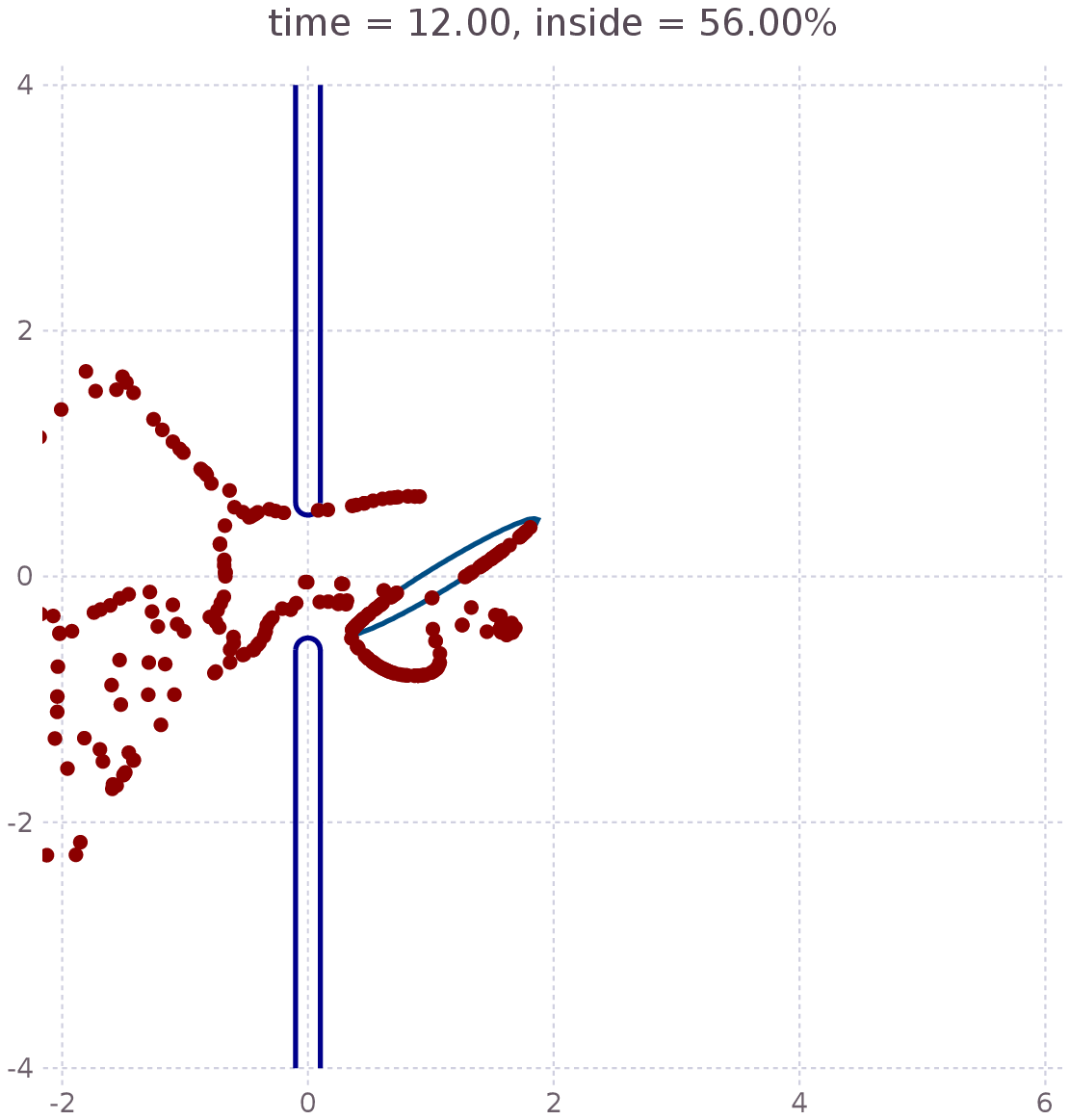}
  \includegraphics[width=0.19\textwidth]{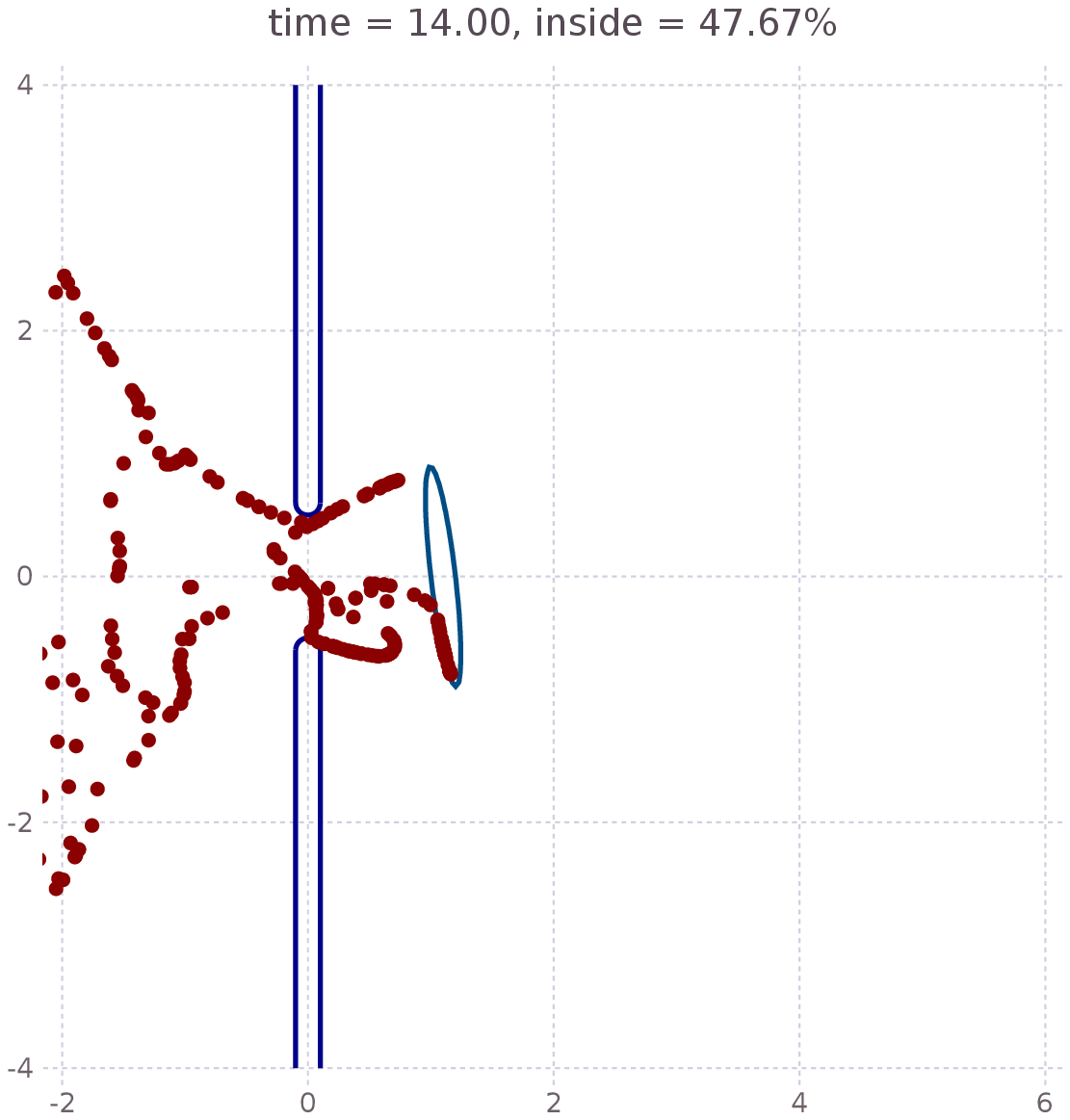}
  \includegraphics[width=0.19\textwidth]{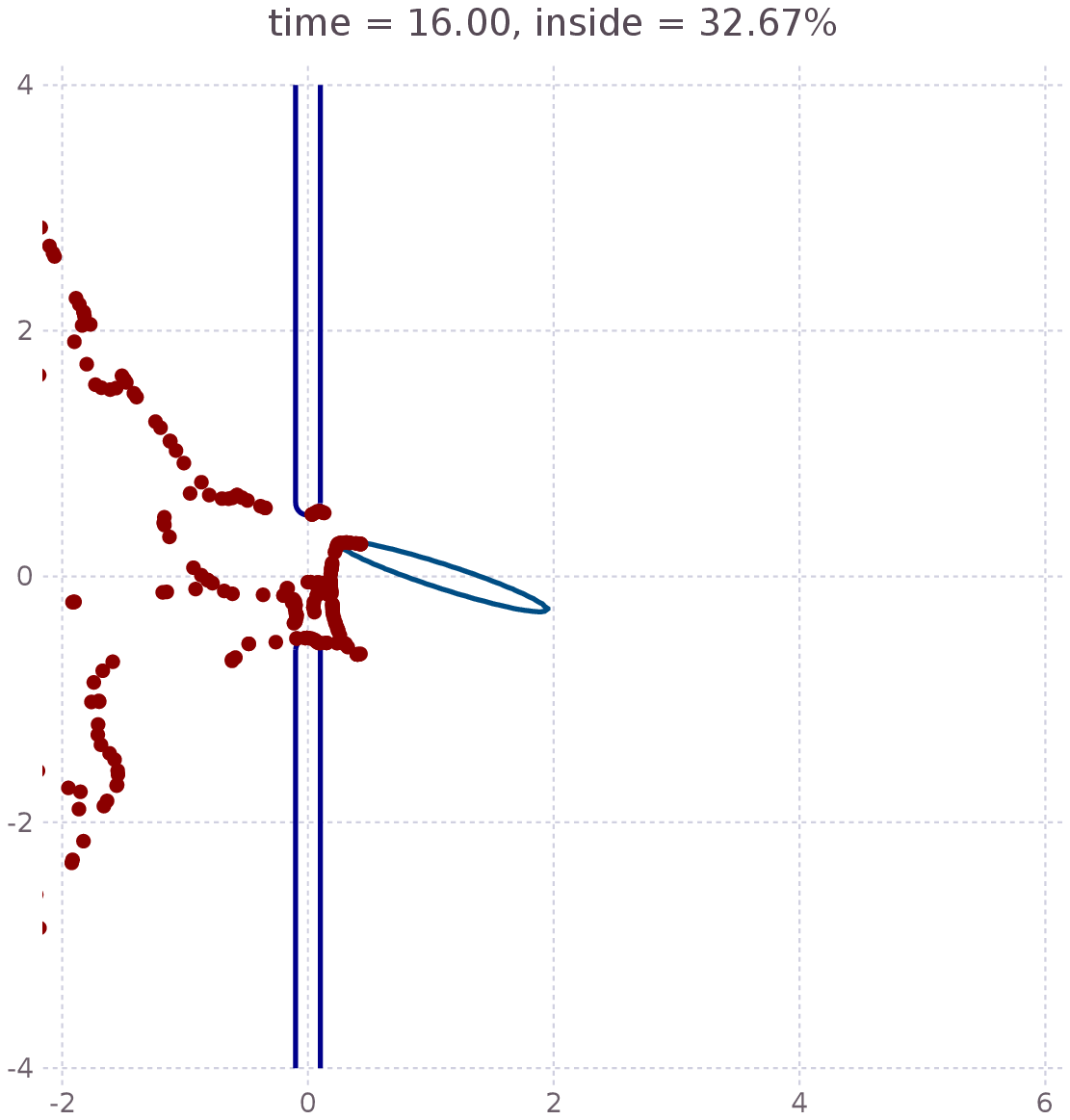}
  \includegraphics[width=0.19\textwidth]{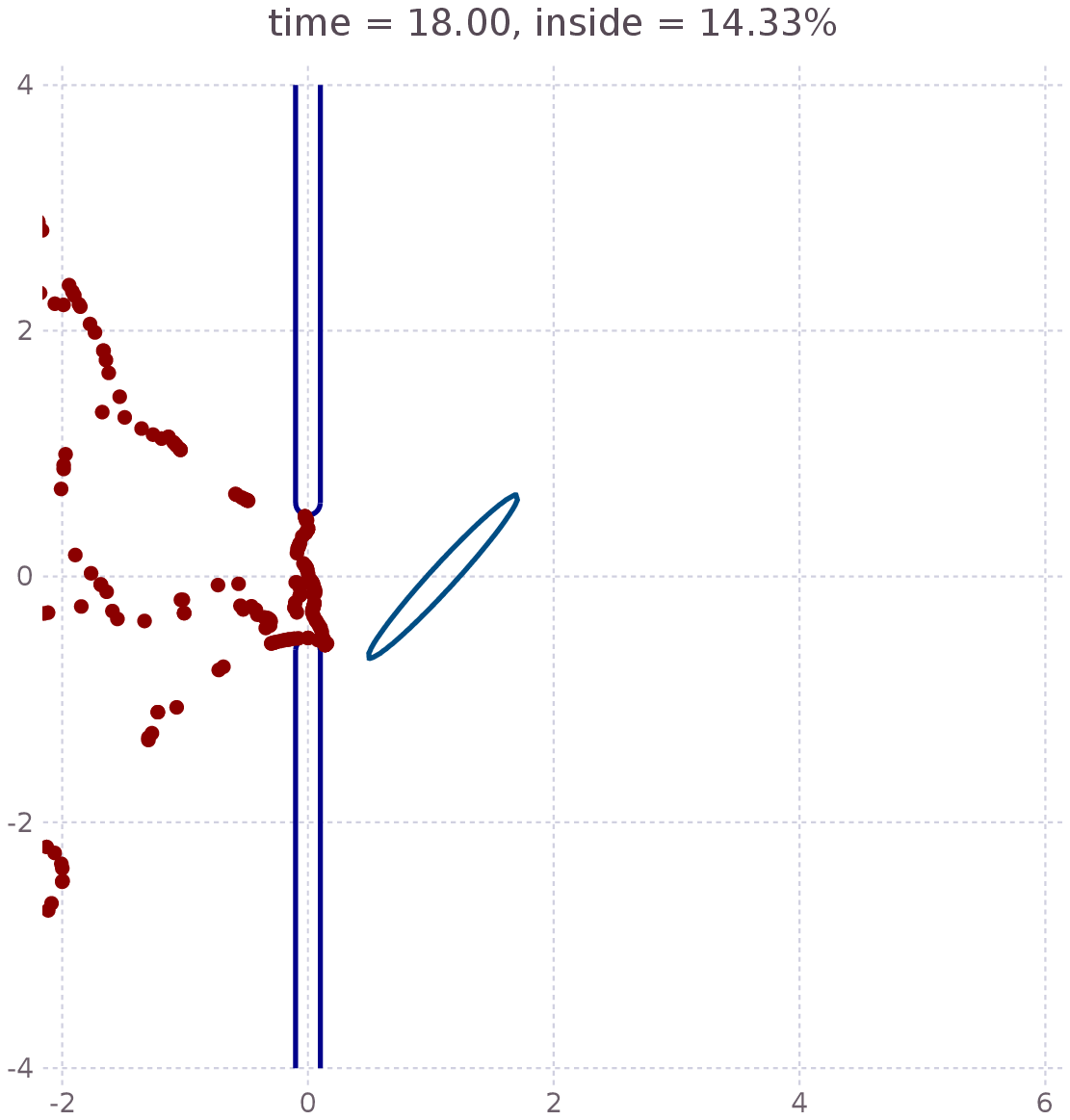}
  \includegraphics[width=0.19\textwidth]{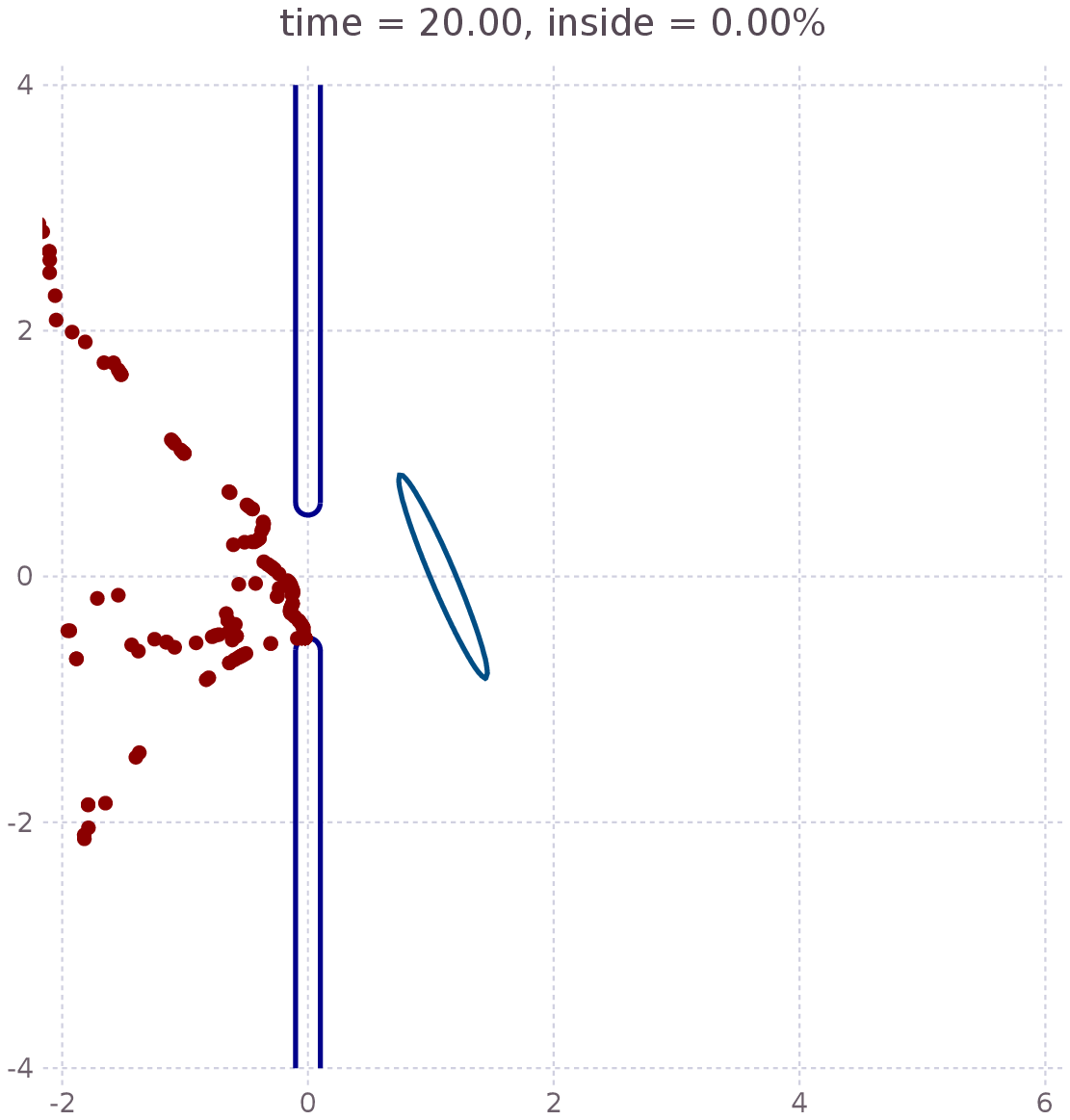}
  \caption{Moving obstacle: \( c = (1.1,0) \), \( a = (0.9,0.1) \), \( \omega = 1 \).}
\end{subfigure}
\caption{The congestion model: solutions at time moments \( t = 2,4,\ldots,20 \).}
\label{fig:colombo}
\end{figure}

\appendix

\section{The Benamou-Brenier functional}
\label{sec:A-BB}

For each couple $(\rho,E)$, where $\rho\in\mathcal{M}(X)$ is a
measure and $E\in\mathcal{M}(X;\mathbb{R}^d)$ is a vector measure, we
correspond the number
\begin{displaymath}
  \mathcal{B}_2(\rho,E) = \sup
  \left\{
    \int_X a(x)\d\rho(x) + \int_X b(x)\cdot \d E(x)\;\colon\;
    (a,b)\in C_b(X;K_2)
  \right\},
\end{displaymath}
where 
\begin{displaymath}
  K_2 = 
  \left\{
    (a,b)\in \mathbb{R}\times \mathbb{R}^d\colon a + \frac{1}{2}|b|^2\leq 0
  \right\}.
\end{displaymath}

\begin{proposition}[Proposition 5.18~\cite{Santambrogio2015}]
  \label{prop:BB}
  The map $\mathcal{B}_2$ is convex and lower semicontinuous on
  $\mathcal{M}(X)\times\mathcal{M}(X;\mathbb{R}^d)$. Moreover,
  \begin{enumerate}[(i)]
  \item $\mathcal{B}_2\geq 0$,
  \item $C_b(X;K_2)$ can be replaced with $L^\infty(X;K_2)$ in the
    definition of $\mathcal{B}_2$,
  \item if $\rho$ and $E$ are absolutely continuous with respect to
    a positive measure $\lambda$ then
    \begin{displaymath}
      \mathcal{B}_2(\rho,E) = \int f_2(\rho(x),E(x))\d\lambda(x),
    \end{displaymath}
      where
      \begin{displaymath}
        f_{2}(t,x) = \sup_{(a,b)\in K_{2}} \left(at +b\cdot x\right) =
        \begin{cases}
          \frac{1}{2t}|x|^{2} & \text{if } t=0,\\
          0 & \text{if } t=0,\; x=0,\\
          +\infty&\text{otherwise};
        \end{cases}
      \end{displaymath}
  \item $\mathcal{B}_2(\rho,E)<+\infty$ only if $\rho\geq 0$ and $E
    \ll \rho$,
  \item for $\rho\geq 0$ and $E \ll \rho$, we have $E = v\rho$ and
    $\mathcal{B}_2(\rho,E) = \frac{1}{2}\int|v|^2\d\rho$.
  \end{enumerate}
\end{proposition}



\section{Continuity of the projection map}
\label{sec:Pcont}

\begin{lemma}
  \label{lem:prox}
  Let \( A_n\xrightarrow{d_{H}} A \) and \( x_n\to x \). If the projections \( P_A(x) \) and \( P_{A_n}(x_n) \) are unique then \( P_{A_n}(x_n) \to P_A(x) \).
\end{lemma}
\begin{proof}
  Consider the following functions
  \begin{displaymath}
    F_n(y) = \chi_{A_n}(y) + |x_n-y|^2,\quad F(y) = \chi_A(y) + |x-y|^2,
  \end{displaymath}
  where \( \chi_{A} \) denotes the indicator function of \( A \).
  From~\cite[Proposition 4.15]{DalMasoBook} it follows that \( F = \Gamma\text{-}\lim
  F_{n} \).
  By our assumptions, each \( P_{A_n}(x_n) \) is a unique minimizer of \( F_n \) and \( P_A(x) \) is
  a unique minimizer of \( F \). Since all \( P_{A_n}(x_n) \) belong to a
  compact set (because \( \{A_{n}\} \) is convergent),
  one may extract a converging subsequence. By one of the key properties of \( \Gamma
  \)-convergence~\cite[Corollary 7.17]{DalMasoBook}, its limit is a minimizer of \( F \), i.e.,
  \( P_A(x) \). This means that \( P_{A_n}(x_n) \to P_A(x)\).
\end{proof}

\section{Derivative of the squared Wasserstein distance}
\label{sec:Diff}

\begin{lemma}
  \label{lem:Wderivative}
  Let \(\mu_t\), \(\nu_t\) be two absolutely continuous curves in \( \mathcal{P}_{2}(\mathbb{R}^{d}) \) and  \(u_t\), \(v_t\) be their velocity vector fields. Then, for a.e. \(t\), one has
  \begin{displaymath}
    \frac{d}{dt} W_{2}^2(\mu_t,\nu_t) = 2\iint \left\langle u_t(x)-v_t(y), x-y\right\rangle\d \Pi_{\mu_t,\nu_t}(x,y),
  \end{displaymath}
  where \(\Pi_{\mu_t,\nu_t}\) is an optimal transport plan from \(\mu_t\) to \(\nu_t\).
\end{lemma}
\begin{proof}
  We shall prove the formula for all \(t\) satisfying the following three conditions:
  1) \(W_{2}^2(\mu_t,\nu_t)\) is differentiable,
  2) \(\lim_{h\to 0}\frac{W_{2}(\mu_{t+h},S^h_{t\sharp}\mu_t)}{h}=0\),
  3) \(\lim_{h\to 0}\frac{W_{2}(\nu_{t+h},P^h_{t\sharp}\mu_t)}{h}=0\).
  Here
  \begin{displaymath}
    S^h_t = \id + hu_t,\quad P^h_t = \id + hv_t.
  \end{displaymath}
  Proposition 8.4.6 from~\cite{Ambrosio2005} says that all such \(t\) compose a set of full measure.

  First, we show that
  \begin{displaymath}
    \frac{d}{ds}W^2_{2}(\mu_s,\nu_s)|_{s=t} = \lim_{h\to 0}
    \frac{W^2_{2}(S^h_{t\sharp}\mu_t, P^h_{t\sharp}\nu_t) - W_{2}^2(\mu_t,\nu_t)}{h}.
  \end{displaymath}
  Indeed,
  \begin{align*}
    W^2_{2}(\mu_{t+h},\nu_{t+h}) - W^2_{2}(\mu_t,\nu_t)
    &=
    W^2_{2}(S^h_{t\sharp}\mu_t,P^h_{t\sharp}\nu_t) - W^2_{2}(\mu_t,\nu_t)\\
    &+
    W^2_{2}(\mu_{t+h},\nu_{t+h}) - W^2_{2}(S^h_{t\sharp}\mu_t,P^h_{t\sharp}\nu_t),
  \end{align*}
  so if we show that
  \begin{displaymath}
   \lim_{h\to 0} \frac{
    W_{2}^2(\mu_{t+h},\nu_{t+h}) - W_{2}^2(S^h_{t\sharp}\mu_t,P^h_{t\sharp}\nu_t)}{h}=0,
  \end{displaymath}
  we are done. Let us note that
  \begin{displaymath}
    \left|W_{2}^2(\mu_{t+h},\nu_{t+h}) - W_{2}^2(S^h_{t\sharp}\mu_t,P^h_{t\sharp}\nu_t)\right|\leq
    C\left|W_{2}(\mu_{t+h},\nu_{t+h}) - W_{2}(S^h_{t\sharp}\mu_t,P^h_{t\sharp}\nu_t)\right|,
  \end{displaymath}
  for some \(C>0\). Now,
  \begin{align*}
    \left|W_{2}(\mu_{t+h},\nu_{t+h}) - W_{2}(S^h_{t\sharp}\mu_t,P^h_{t\sharp}\nu_t)\right|
    &\leq
    \left|W_{2}(\mu_{t+h},\nu_{t+h}) - W_{2}(S^h_{t\sharp}\mu_t,\nu_{t+h})\right|\\
    &+
    \left|W_{2}(S^h_{t\sharp}\mu_t,\nu_{t+h})- W_{2}(S^h_{t\sharp}\mu_t,P^h_{t\sharp}\nu_t)\right|\\
    &\leq W_{2}(\mu_{t+h},S^h_{t\sharp}\mu_t) + W_{2}(\nu_{t+h},P^h_{t\sharp}\nu_t).
  \end{align*}
  It remains to apply properties 2 and 3.

  Choose any optimal plan \(\Pi\) between \(\mu_t\) and \(\nu_t\).
  Note that the plan \((S^h_t\circ\pi^1,P^h_t\circ\pi^2)_\sharp\Pi\)
  transports \(S^h_{t\sharp}\mu_t\) to \(P^h_{t\sharp}\nu_t\). Hence
  \begin{align*}
    W_{2}^2 \left(S^h_{t\sharp}\mu_t,P^h_{t\sharp}\nu_t\right) &\leq
    \int \left|x + hu_t(x) - y - h v_t(y)\right|^2\d\Pi(x,y)\\
    &\leq W_{2}^2(\mu_t,\nu_t) + h^2\int|u_t(x)-v_t(y)|^2\d \Pi(x,y)\\
    &+ 2h\int\langle u_t(x)-v_t(y),x-y\rangle\d\Pi(x,y).
  \end{align*}
  Therefore, if \(h>0\), we get
  \begin{displaymath}
      \frac{d}{ds}W_{2}^2(\mu_s,\nu_s)|_{s=t} \leq 2 \int\langle u_t(x)-v_t(y),x-y\rangle\d\Pi(x,y).
  \end{displaymath}
  If \(h<0\), we get the opposite inequality.
\end{proof}

\section{No-flux property}

Here we prove a simple property of the perturbed sweeping process
\begin{equation}
  \label{eq:clsp}
  \dot y(t) \in v_{t}(y(t))- N_{\bm C(t)}(y(t)) \quad \text{for a.e. }t\in[0,T].
\end{equation}
that we failed to find in the literature. Below
\( \langle (s,x),(t,y)\rangle \) denotes the scalar product in \( \mathbb{R}^{d+1} \) and
\( x\cdot y \) the scalar product in \( \mathbb{R}^{d} \).

\begin{proposition}
  \label{prop:minnorm}
  Let \( (t,x)\mapsto v_{t}(x) \) be measurable in \( t \), \( L \)-Lipschitz in
  \( x \) and \( L \)-bounded, \( \bm C\colon [0,T]\rightrightarrows \mathbb{R}^{d} \) satisfy
  \( (\mathbf{A_{2}}) \) and have \( r' \)-prox-regular graph,
  \( r'>0 \). Let \( y \) be a solution of~\eqref{eq:clsp}. If
  \( t_{0}\in (0,T) \) is so that \( \dot y(t_{0}) \) exists then
  \( (1,\dot y(t_{0})) \) is tangent to \( \graph \bm C \) at
      \( (t_{0},y(t_{0})) \) in the sense that
  \begin{displaymath}
    \xi + \dot y(t_{0})\cdot \eta = 0\quad \forall (\xi,\eta)\in N_{\graph\bm C}(t_{0},y(t_{0})).
  \end{displaymath}
\end{proposition}
\begin{proof}
  Pick some \( (\xi,\eta) \in N_{\graph \bm C}(t_{0},y(t_{0})) \). By
  Proposition~\ref{prop:proxreg}(b), we have
  \begin{equation}
    \label{eq:proxmin}
    \frac{|(\xi,\eta)|}{2r'}\left|(t,y(t))-(t_{0},y(t_{0}))\right|^{2} -\left\langle (\xi,\eta),(t,y(t)) - (t_{0},y(t_{0})) \right\rangle \geq 0,
  \end{equation}
  for all \( t \) sufficiently close to \( t_{0} \). Since for \( t=t_{0} \) the
  function on the left-hand side of~\eqref{eq:proxmin} becomes \( 0 \), we
  conclude that \( t_{0} \) is its extremal point. By Fermat's rule,
  \( \xi + \dot y(t_{0})\cdot \eta = 0 \).
\end{proof}

\small{

  \bibliography{references}

\begin{thebibliography}{10}

\bibitem{AmbrosioCrippa2014}
L.~{Ambrosio} and G.~{Crippa}.
\newblock {Continuity equations and ODE flows with non-smooth velocity.}
\newblock {\em {Proc. R. Soc. Edinb., Sect. A, Math.}}, 144(6):1191--1244,
  2014.

\bibitem{Ambrosio2005}
L.~Ambrosio, N.~Gigli, and G.~Savar{\'{e}}.
\newblock {\em {Gradient flows in metric spaces and in the space of probability
  measures.}}
\newblock Basel: Birkh{\"{a}}user, 2005.

\bibitem{Bogachev2007}
V.~I. Bogachev.
\newblock {\em {Measure theory. Vol. I, II}}.
\newblock Springer-Verlag, Berlin, 2007.

\bibitem{Burago2001}
D.~Burago, Y.~Burago, and S.~Ivanov.
\newblock {\em {A course in metric geometry.}}, volume~33.
\newblock Providence, RI: American Mathematical Society (AMS), 2001.

\bibitem{Carrillo2010}
J.~A. Carrillo, M.~Fornasier, G.~Toscani, and F.~Vecil.
\newblock {Particle, kinetic, and hydrodynamic models of swarming.}
\newblock In {\em Mathematical modeling of collective behavior in
  socio-economic and life sciences}, pages 297--336. Boston, MA:
  Birkh{\"{a}}user, 2010.

\bibitem{Colombo2011}
R.~Colombo.
\newblock {Control of the continuity equation with a non local flow}.
\newblock {\em ESAIM: Control, Optimisation and Calculus of Variations},
  17(2):353--379, 2011.

\bibitem{ColomboNonloc2011}
R.~M. Colombo, M.~Garavello, and M.~L{\'{e}}cureux-Mercier.
\newblock {Non-local crowd dynamics}.
\newblock {\em Comptes Rendus Mathematique}, 349(13-14):769--772, 2011.

\bibitem{ColomboRossi2018}
R.~M. Colombo and E.~Rossi.
\newblock {Nonlocal conservation laws in bounded domains.}
\newblock {\em SIAM J. Math. Anal.}, 50(4):4041--4065, 2018.

\bibitem{DalMasoBook}
G.~{Dal Maso}.
\newblock {\em {An introduction to {$\Gamma$}-convergence.}}, volume~8.
\newblock Basel: Birkh{\"{a}}user, 1993.

\bibitem{Marino2015}
S.~{Di Marino}, B.~{Maury}, and F.~{Santambrogio}.
\newblock {Measure sweeping processes.}
\newblock {\em {J. Convex Anal.}}, 23(2):567--601, 2016.

\bibitem{FedererCurvMeasures}
H.~Federer.
\newblock {Curvature measures}.
\newblock {\em Trans. Am. Math. Soc.}, 93:418--491, 1959.

\bibitem{Figalli2010}
A.~Figalli and N.~Gigli.
\newblock {A new transportation distance between non-negative measures, with
  applications to gradients flows with Dirichlet boundary conditions}.
\newblock {\em Journal des Mathematiques Pures et Appliquees}, 94(2):107--130,
  2010.

\bibitem{Helbing2002}
D.~Helbing, I.~J. Farkas, and T.~Vicsek.
\newblock {\em {Crowd Disasters and Simulation of Panic Situations}}, pages
  330--350.
\newblock Springer Berlin Heidelberg, Berlin, Heidelberg, 2002.

\bibitem{MauryBook2019}
B.~Maury and S.~Faure.
\newblock {\em {Crowds in equations. An introduction to the microscopic
  modeling of crowds.}}
\newblock Hackensack, NJ: World Scientific, 2019.

\bibitem{Mogilner1999}
A.~Mogilner and L.~Edelstein-Keshet.
\newblock {A non-local model for a swarm}.
\newblock {\em Journal of Mathematical Biology}, 38(6):534--570, 1999.

\bibitem{PiccoliRossi2013}
B.~Piccoli and F.~Rossi.
\newblock {Transport equation with nonlocal velocity in Wasserstein spaces:
  convergence of numerical schemes.}
\newblock {\em Acta Appl. Math.}, 124(1):73--105, 2013.

\bibitem{PiccoliRossi2018}
B.~Piccoli and F.~Rossi.
\newblock {Measure-theoretic models for crowd dynamics.}
\newblock In {\em Crowd dynamics, Volume 1. Theory, models, and safety
  problems}, pages 137--165. Cham: Birkh{\"{a}}user, 2018.

\bibitem{Rockafellar1998}
R.~T. Rockafellar and R.~J.-B. Wets.
\newblock {\em {Variational analysis}}, volume 317 of {\em Grundlehren der
  Mathematischen Wissenschaften [Fundamental Principles of Mathematical
  Sciences]}.
\newblock Springer-Verlag, Berlin, 1998.

\bibitem{Santambrogio2015}
F.~Santambrogio.
\newblock {\em {Optimal transport for applied mathematicians. Calculus of
  variations, PDEs, and modeling.}}, volume~87.
\newblock Cham: Birkh{\"{a}}user/Springer, 2015.

\bibitem{SeneThibault2014}
M.~Sene and L.~Thibault.
\newblock {Regularization of dynamical systems associated with prox-regular
  moving sets.}
\newblock {\em J. Nonlinear Convex Anal.}, 15(4):647--663, 2014.

\bibitem{Thibault2008}
L.~Thibault.
\newblock {Regularization of Nonconvex Sweeping Process in Hilbert Space}.
\newblock {\em Set-Valued Analysis}, 16(2-3):319--333, jun 2008.

\bibitem{Villani2009}
C.~Villani.
\newblock {\em {Optimal transport. Old and new.}}, volume 338.
\newblock Berlin: Springer, 2009.

\end{thebibliography}

  \bibliographystyle{abbrv}

}

\end{document}